\documentclass[a4paper,11pt]{article}
\usepackage[utf8]{inputenc}
\usepackage[T1]{fontenc}
\usepackage{textcomp}
\usepackage{fancyhdr}
\usepackage{fullpage}
\usepackage{lmodern}
\usepackage{amsmath,amssymb,bm,bbm}
\usepackage{algorithm}
\usepackage{algorithmic}
\usepackage{float}
\usepackage{graphicx}
\usepackage{extarrows}
\usepackage{caption}
\usepackage{graphicx, subfig}
\usepackage{titling}
\usepackage{yhmath}
\usepackage{mathrsfs}
\usepackage{cite}
\usepackage{footnote}
\usepackage{amsthm}
\usepackage{color}
\usepackage{slashed}
\usepackage{verbatim}
\usepackage[colorlinks=true,linkcolor=red,citecolor=blue]{hyperref}
\usepackage{amsfonts,euscript}
\usepackage{latexsym, multicol}
\usepackage{epstopdf}
\usepackage{psfrag}
\usepackage{mathrsfs}
\usepackage{xcolor}
\usepackage{comment}
\usepackage{authblk}
\usepackage{indentfirst}
\usepackage{lipsum}
\DeclareFontFamily{U}{mathx}{\hyphenchar\font45}
\DeclareFontShape{U}{mathx}{m}{n}{
<5> <6> <7> <8> <9> <10>
<10.95> <12> <14.4> <17.28> <20.74> <24.88>
mathx10}{}
\DeclareSymbolFont{mathx}{U}{mathx}{m}{n}
\DeclareFontSubstitution{U}{mathx}{m}{n}
\DeclareMathAccent{\widecheck}{0}{mathx}{"71}
\numberwithin{equation}{section}
\allowdisplaybreaks[3]
\renewcommand{\O}{\mathbb{O}}
\def\vphig{\overset{\circ}{\vphi}}
\def\db{{\de_B}}
\def\Ve{{^eV}}
\def\Vi{{^iV}}
\def\trchc{\widecheck{\trch}}
\def\Eg{\overset{\circ}{E}}
\def\Hg{\overset{\circ}{H}}
\def\Xg{\overset{\circ}{X}}

\def\dko{\dkb^{\leq 1}}
\def\Lieh{\hat{\Lie}}

\def\phic{\widecheck{\phi}}
\def\Vie{{^i_eV}}
\def\Vii{{^i_iV}}
\def\Des{\slashed{\De}}
\def\slg{\slashed{g}}
\def\Lie{\mathcal{L}}
\def\sdivs{{\slashed{\sdiv}\,}}
\def\curls{{\slashed{\curl}\,}}
\def\kah{\widehat{\ka}}
\def\thh{\widehat{\th}}
\def\th{\theta}
\def\eps{\epsilon}
\def\FF{\mathcal{F}}
\def\BB{\mathcal{B}}
\def\DD{\mathcal{D}}
\def\BBe{{^e\BB}}
\def\BBi{{^i\BB}}
\def\DDe{{^e\DD}}
\def\DDi{{^i\DD}}
\def\EEe{{^e\EE}}
\def\EEi{{^i\EE}}
\def\FFe{{^e\FF}}
\def\FFi{{^i\FF}}

\def\nabs{\slashed{\nab}}
\def\Gaw{\Ga_w}

\def\sic{\wideparen{\si}}

\def\F{\mathbb{E}}
\def\Fb{\mathbb{F}}

\def\mk{\mathcal{K}}
\def\mre{{^{E}\mr}}
\def\mri{{^{I}\mr}}
\def\moe{{^{E}\mo}}
\def\moi{{^{I}\mo}}
\def\mke{{^{E}\mk}}
\def\mki{{^{I}\mk}}
\def\mrt{{^{T}\mr}}
\def\mot{{^{T}\mo}}
\def\mkt{{^{T}\mk}}
\def\Dinf{D^{(\infty)}}
\renewcommand{\c}{\cdot}

\newcommand{\pa}{\partial}
\newcommand{\pr}{\pa}

\newcommand{\una}{\underline{\alpha}}
\newcommand{\unb}{\underline{\beta}}
\newcommand{\bg}{\mathbf{g}}
\newcommand{\mo}{\mathcal{O}}
\newcommand{\unchi}{\underline{\chi}}
\newcommand{\omb}{\underline{\omega}}
\newcommand{\etab}{\underline{\eta}}
\newcommand{\sld}{\slashed{d}}
\def\kg{\overset{\circ}{k}}
\def\ujp{{\langle u \rangle}}

\def\Vphi{\varPhi}
\def\vphi{\varphi}
\def\EE{\mathcal{E}}
\def\ee{\mathcal{N}}

\def\slu{{\slashed{U}}}

\def\aac{\mathring{\aa}}
\def\aas{\underline{\slashed{\a}}}

\def\sfr{\mathfrak{s}}
\def\sk{\sfr}

\def\Rfk{\mathfrak{R}}

\newcommand{\kk}{\mathcal{V}}
\newcommand{\mr}{\mathcal{R}}

\newcommand{\mv}{\mathcal{B}}
\newcommand{\M}{\mathcal{M}}
\newcommand{\D}{\mathbf{D}}

\newcommand{\slr}{\slashed{R}}

\newcommand{\uf}{\underline{F}}

\newcommand{\cuv}{{C_u^V}}
\newcommand{\ucuv}{{\Si_t^V}}
\newcommand{\dd}{{\mathfrak{d}}}

\newcommand{\Ric}{{\ric}}
\DeclareMathOperator{\sRic}{Ric}
\newcommand{\g}{\bg}
\newcommand{\R}{{\mathbf{R}}}

\def\ac{{\a_4}}
\def\as{\slashed{\a}}
\def\hot{\widehat{\otimes}}

\def\II{\mathcal{I}}
\def\la{\lambda}
\def\rhoc{\wideparen{\rho}}

\def\Rk{\mathfrak{R}_0^2}
\def\fc{\widecheck{f}}

\def\MM{\mathcal{M}}

\def\chib{\unchi}

\newcommand{\C}{\mathbf{C}}
\def\om{\omega}
\def\ze{\zeta}

\def\aa{\una}
\def\bb{\unb}
\def\Si{\Sigma}
\def\si{\sigma}

\def\dk{\dd}
\def\dkb{\slashed{\dk}}
\def\BB{\mv}
\def\XX{\mathcal{X}}
\def\ga{\gamma}
\def\Ga{\Gamma}
\def\xib{\underline{\xi}}
\def\a{\alpha}
\def\b{\beta}
\def\hch{\widehat{\chi}}
\def\hchb{\widehat{\chib}}
\def\trch{\tr\chi}
\def\trchb{\tr\chib}
\def\trchc{\widecheck{\trch}}

\def\ka{\kappa}

\def\de{\delta}
\def\De{\Delta}
\def\nab{\nabla}
\def\ov{\overline}

\def\Gab{\Ga_b}
\def\Gag{\Ga_g}

\def\tb{\underline{t}}
\def\ep{\varepsilon}
\def\les{\lesssim}
\def\RR{\mr}

\def\DDinf{\DD^{(\infty)}}
\def\pih{{\widehat{\pi}}}

\DeclareMathOperator{\curl}{curl}

\DeclareMathOperator{\grad}{grad}

\DeclareMathOperator{\tr}{tr}
\DeclareMathOperator{\sdiv}{div}
\def\bdiv{\mathbf{Div}}
\def\ric{\mathbf{Ric}}
\def\aac{\aa_3}

\newtheorem{thm}{Theorem}[section]
\newtheorem{prop}[thm]{Proposition}
\newtheorem{lem}[thm]{Lemma}
\newtheorem{cor}[thm]{Corollary}
\newtheorem{rk}[thm]{Remark}
\newtheorem{df}[thm]{Definition}

\title{Global stability of Minkowski spacetime with minimal decay}
\author{Dawei Shen\footnote{Email adress: dawei.shen@polytechnique.edu \par \indent\hspace{0.26cm} Laboratoire Jacques-Louis Lions, Sorbonne Universit\'e, 75252 Paris, France}}
\begin{document}
\maketitle
\begin{abstract}
The global stability of Minkowski spacetime, a milestone in the field, has been proven in the celebrated work of Christodoulou and Klainerman \cite{Ch-Kl} in 1993. In 2007, Bieri \cite{Bieri} has extended the result of \cite{Ch-Kl} under lower decay and regularity assumptions on the initial data. In this paper, we extend the result of \cite{Bieri} to minimal decay assumptions. Also, concerning the treatment of curvature estimates, we replace the vectorfield method used in \cite{Ch-Kl,Bieri} by the $r^p$--weighted estimates of Dafermos and Rodnianski \cite{Da-Ro}.

{\centering\subsubsection*{\small Keywords}}
\noindent Minkowski stability, Maximal-null foliation, $r^p$--weighted estimates
\end{abstract}
\tableofcontents
\tableofcontents
\section{Introduction}\label{sec6}
\subsection{Einstein vacuum equations and the Cauchy problem}
A Lorentzian $4$--manifold $(\MM,\g)$ is called a vacuum spacetime if it solves the Einstein vacuum equations:
\begin{equation}\label{EVE}
    \Ric(\g)=0\quad \mbox{ in }\MM,
\end{equation}
where $\Ric$ denotes the Ricci tensor of the Lorentzian metric $\g$. The Einstein vacuum equations are invariant under diffeomorphisms, and therefore one considers equivalence classes of solutions. Expressed in general coordinates, \eqref{EVE} is a non-linear geometric coupled system of partial differential equations of order 2 for $\g$. In suitable coordinates, for example so-called wave coordinates, it can be shown that \eqref{EVE} is hyperbolic and hence admits an initial value formulation. \\ \\
The corresponding initial data for the Einstein vacuum equations is given by specifying a triplet $(\Si,g,k)$ where $(\Si,g)$ is a Riemannian $3$--manifold and $k$ is the traceless symmetric $2$--tensor on $\Si$ satisfying the constraint equations:
\begin{align}
    \begin{split}\label{constraintk}
        R&=|k|^2-(\tr k)^2,\\
        \nab^jk_{ij}&=\nab_i\tr k,
    \end{split}
\end{align}
where $R$ denotes the scalar curvature of $g$, $\nab$ denotes the Levi-Civita connection of $g$ and
\begin{align*}
    |k|^2:=g^{im}g^{jl}k_{ij}k_{lm},\qquad\quad\tr k:=g^{ij}k_{ij}.
\end{align*}
In the future development $(\MM,\g)$ of such initial data $(\Si,g,k)$, $\Si\subset \MM$ is a spacelike hypersurface with induced metric $g$ and second fundamental form $k$.\\ \\
The seminal well-posedness results for the Cauchy problem obtained in \cite{cb,cbg} ensure that for any smooth Cauchy data, there exists a unique smooth maximal globally hyperbolic development $(\MM,\g)$ solution of Einstein equations \eqref{EVE} such that $\Si\subset \MM$ and $g$, $k$ are respectively the first and second fundamental forms of $\Si$ in $\MM$. \\ \\
The prime example of a vacuum spacetime is Minkowski space $(\MM,\eta)$:
\begin{equation*}
    \MM=\mathbb{R}^4,\qquad \eta=-(dt)^2+(dx^1)^2 +(dx^2)^2+(dx^3)^2,
\end{equation*}
for which Cauchy data are given by
\begin{equation*}
    \Si=\mathbb{R}^3,\qquad g=(dx^1)^2+(dx^2)^2+(dx^3)^2,\qquad k=0.
\end{equation*}
In the present work, we consider the problem of the stability of Minkowski spacetime and start by reviewing the state of the art on this problem.
\subsection{Previous works on the stability of Minkowski spacetime}\label{ssec6.1}
In 1993, Christodoulou and Klainerman \cite{Ch-Kl} proved the global stability of Minkowski for the Einstein-vacuum equations, a milestone in the domain of mathematical general relativity. In 2007, Bieri \cite{Bieri} gave a new proof of global stability of Minkowski requiring one less derivative and less vectorfields compared to \cite{Ch-Kl}. Both \cite{Ch-Kl} and \cite{Bieri} rely on the maximal foliation. Given that the goal of this paper is to extend the result of \cite{Bieri}, we will state the results of \cite{Ch-Kl,Bieri} in Section \ref{ssec6.2}.\\ \\
We now mention proofs of Minkowski stability using other gauges. In 2003, Klainerman and Nicol\`o \cite{Kl-Ni} proved the Minkowski stability in the exterior of an outgoing cone using the double null foliation. Moreover, Klainerman and Nicol\`o \cite{knpeeling} showed that under stronger asymptotic decay and regularity properties than those used in \cite{Ch-Kl,Kl-Ni}, asymptotically flat initial data sets lead to solutions of the Einstein vacuum equations which have strong peeling properties. Lindblad and Rodnianski \cite{lr1,lr2} gave a new proof of the stability of the Minkowski spacetime using \emph{wave-coordinates} and showing that the Einstein equations verify the so called \emph{weak null structure} in that gauge. Huneau \cite{huneau} proved the nonlinear stability of Minkowski spacetime with a translation Killing field using generalised wave-coordinates. Using the framework of Melrose’s b-analysis, Hintz and Vasy \cite{hv} reproved the stability of Minkowski space. Graf \cite{graf} proved the global nonlinear stability of Minkowski space in the context of the spacelike-characteristic Cauchy problem for Einstein vacuum equations, which together with \cite{Kl-Ni} allows to reobtain \cite{Ch-Kl}. Under the framework of \cite{Kl-Ni} and using $r^p$--weighted estimates of Dafermos and Rodnianski \cite{Da-Ro}, the author \cite{ShenMink} reproved the Minkowski stability in exterior regions. More recently, Hintz \cite{Hintz} reproved the Minkowski stability in exterior regions by using the framework of \cite{hv}.\\ \\
There are also stability results concerning Einstein's equations coupled with non trivial matter fields:
\begin{itemize}
\item Einstein-Maxwell system: Zipser \cite{zipser} extended the framework of \cite{Ch-Kl} to show the stability of the Minkowski spacetime solution to the Einstein–Maxwell system. In \cite{lo09}, Loizelet used the framework of \cite{lr2} to demonstrate the stability of the Minkowski spacetime solution of the Einstein-scalar field-Maxwell system in $(1+n)$-dimensions $(n\geq 3)$. Speck \cite{speck} gave a proof of the global nonlinear stability of the $(1+3)$-dimensional Minkowski spacetime solution to the coupled system for a family of electromagnetic fields, which includes the standard Maxwell fields.
\item Einstein-Klein-Gordon system: Lefloch and Ma \cite{lefloch} and Wang \cite{wang} proved the global stability of Minkowski for the Einstein-Klein-Gordon system with initial data coinciding with the Schwarzschild solution with small mass outside a compact set. Ionescu and Pausader \cite{ionescu} proved the global stability of Minkowski for the Einstein-Klein-Gordon system for general initial data.
\item Einstein-Vlasov system: Taylor \cite{taylor} considered the massless case where the initial data for the Vlasov part is compactly supported on the mass shell. Fajman, Joudioux and Smulevici \cite{fajman} considered the massive case where the initial data coincides with Schwarzschild in the exterior region and with compact support assumption only in space on the Vlasov part. Lindblad and Taylor \cite{lt} considered the massive case where the initial data has compact support for the Vlasov part. Bigorgne, Fajman, Joudioux, Smulevici and Thaller \cite{bigo} considered the massless case for general initial data. Wang \cite{wxc} considered the massive case for general initial data.
\end{itemize}
\subsection{Minkowski stability in \texorpdfstring{\cite{Ch-Kl,Bieri}}{}}\label{ssec6.2}
We recall in this section the results in \cite{Ch-Kl,Bieri}. First, we recall the definition of a \emph{maximal hypersurface}, which plays an important role in the statements of the main theorems in \cite{Ch-Kl,Bieri}.
\begin{df}\label{def6.1}
An initial data $(\Si,g,k)$ is posed on a maximal hypersurface if it satisfies 
\begin{equation}
    \tr k=0.
\end{equation}
In this case, we say that $(\Si,g,k)$ is a maximal initial data set, and the constraint equations \eqref{constraintk} reduce to
\begin{equation}
    R=|k|^2,\qquad \sdiv k=0,\qquad \tr k=0.
\end{equation}
\end{df}
We introduce the notion of \emph{$(s,q)$--asymptotically flat initial data}.
\begin{df}\label{def6.3}
Given $s>1$ and $q\in\mathbb{N}$, we say that a data set $(\Si_0,g,k)$ is $(s,q)$--asymptotically flat if there exists a coordinate system $(x^1,x^2,x^3)$ defined outside a sufficiently large compact set such that:
\begin{itemize}
\item In the case $s\geq 3$\footnote{The notation $f=o_q(r^{-m})$ means $\pr^\a f=o(r^{-m-|\a|})$, $|\a|\leq q$.}
\begin{align}
    \begin{split}\label{old1.3}
    g_{ij}&=\left(1-\frac{2M}{r}\right)^{-1}dr^2+r^2 d\si_{\mathbb{S}^2}+o_{q+1}(r^{-\frac{s-1}{2}}),\\
    k_{ij}&=o_q(r^{-\frac{s+1}{2}}).
    \end{split}
\end{align}
\item In the case $1<s<3$
\begin{align}
    \begin{split}\label{gks1}
        g_{ij}&=\de_{ij}+o_{q+1}(r^{-\frac{s-1}{2}}),\\
        k_{ij}&=o_q(r^{-\frac{s+1}{2}}).
    \end{split}
\end{align}
\end{itemize}
\end{df}
\begin{rk}
    According to \eqref{gks1}, the assumption $s>1$ is necessary to define an asymptotic perturbation of the flat metric on $\Si_0$. See also Remark \ref{restrictions} for more explanations.
\end{rk}
\begin{df}\label{dffunctional}
We denote $d_0$ the geodesic distance from a fixed point $O\in\Si_0$, $B_{ij}:=(\curl\widehat{{R}})_{ij}$ the \emph{Bach tensor} and  $\widehat{{R}}$ is the traceless part of $\sRic$. Then, we define for $q\geq 2$:
\begin{itemize}
\item In the case $s\geq 3$:
\begin{align}
\begin{split}\label{Ckfunctional}
J_0^{(q)}(\Si_0,g,k)&:=\sup_{\Si_0}\Big((d_0^2+1)^{3}|\sRic|^2 \Big)+\sum_{l=0}^q\int_{\Si_0}(d_0^2+1)^{l+\frac{s-2}{2}}|\nab^lk|^2 \\
&+\sum_{l=0}^{q-2}\int_{\Si_0}(d_0^2+1)^{l+\frac{s+2}{2}}|\nab^lB|^2.
\end{split}
\end{align}
\item In the case $1<s<3$:
\begin{align}
\begin{split}\label{bierifunctional}
J_0^{(q)}(\Si_0,g,k)&:=\sum_{l=0}^q\int_{\Si_0}(d_0^2+1)^{l+\frac{s-2}{2}}|\nab^lk|^2 +\sum_{l=0}^{q-1}\int_{\Si_0}(d_0^2+1)^{l+\frac{s}{2}}|\nab^l\sRic|^2.
\end{split}
\end{align}
\end{itemize}
\end{df}
We can now state the main theorems of \cite{Ch-Kl} and \cite{Bieri}.
\begin{thm}[Global stability of Minkowski spacetime \cite{Ch-Kl}]\label{ckmain}
There exists an $\ep_0>0$ sufficiently small such that if $J_0^{(3)}(\Sigma_0,g,k)\leq\ep_0^2$, then the initial data set $(\Sigma_0,g,k)$, $(4,3)$--asymptotically flat (in the sense of Definition \ref{def6.3}) and maximal, has a unique, globally hyperbolic, smooth, geodesically complete solution. This development is globally asymptotically flat, i.e. the Riemann curvature tensor tends to zero along any causal or space-like geodesic. Moreover, there exists a global maximal time function $t$ and an optical function\footnote{An optical function $u$ is a scalar function satisfying $\g^{\a\b}\pr_\a u\pr_\b u=0$.} $u$ defined everywhere in an external region.
\end{thm}
\begin{thm}[Global stability of Minkowski spacetime \cite{Bieri}]\label{bierimain}
There exists an $\ep_0>0$ sufficiently small such that if $J_0^{(2)}(\Sigma_0,g,k)\leq\ep_0^2$, then the initial data set $(\Sigma_0,g,k)$, $(2,2)$--asymptotically flat and maximal, has a unique, globally hyperbolic, smooth, geodesically complete solution. This development is globally asymptotically flat. Moreover, there exists a global maximal time function $t$ and an optical function $u$ defined everywhere in an external region.
\end{thm}
\begin{rk}
Various choices of the parameter $s$ in Definition \ref{def6.3} have been made in the literature. In particular:\footnote{This list focuses on low values of $s$. For large values of $s$, note for example \cite{knpeeling} valid for $s>7$ and \cite{ShenMink} valid for $s>3$.}
\begin{itemize}
    \item $s=4$ is used for example in \cite{Ch-Kl}.
    \item $s=3+\de$\footnote{Here and in the line below, $\de$ denotes a constant satisfying $0<\de\ll 1$.} is used for example in \cite{lr2}.
    \item $s=3-\de$ is used in \cite{ionescu} for the Einstein-Klein-Gordon system.
    \item $s=2$ is used in \cite{Bieri}.
\end{itemize}
\end{rk}
The goal of this paper is to extend the results of \cite{Bieri} to lower values of $s$. More precisely, we prove the global stability of Minkowski spacetime for $(s,2)$--asymptotically flat initial data for all $s\in(1,2]$.
\subsection{Rough version of the main theorem}\label{ssec6.3}
In this section, we state a simple version of our main theorem. For the precise statement, see Theorem \ref{th8.1}.
\begin{thm}[Main Theorem, first version]\label{old1.6}
Let $s\in(1,2]$ and let an initial data set $(\Si_0,g,k)$ which is $(s,2)$--asymptotically flat in the sense of Definition \ref{def6.3}. Assume that we have a smallness conditions in an initial layer region $\kk_{(0)}$\footnote{See \eqref{dfinitiallayer} for the definition of $\kk_{(0)}$.} near $\Si_0$. Then, there exists a unique future development $(\M,\g)$ in its future domain of dependence with the following properties:
\begin{itemize}
    \item $(\M,\g)$ can be foliated by a maximal-null foliation $(C_u,\Si_t)$ whose outgoing leaves $C_u$ are complete for all $u$;
    \item We have detailed control of all the quantities associated with the maximal-null foliation of the spacetime, see Theorem \ref{th8.1}.
\end{itemize}
\end{thm}
\begin{rk}
    In the particular case $s=2$, we reobtain the results of \cite{Bieri} restated above in Theorem \ref{bierimain}.
\end{rk}
The structure of the proof of Theorem \ref{old1.6} is similar to \cite{Ch-Kl,Bieri}, see Section \ref{ssec8.6}. Below, we compare the proof of this paper and that of \cite{Bieri}:
\begin{enumerate}
\item In \cite{Bieri}, the functional \eqref{bierifunctional} is introduced to fix the initial conditions on the initial hypersurface $\Si_0$. Here we fix the initial conditions in an initial layer region $\kk_{(0)}$ near the initial hypersurface $\Si_0$.
\item To estimate the norms of curvature components, \cite{Bieri} uses the vectorfield method introduced in \cite{Ch-Kl}. Here, we estimate the curvature norms by relying on $r^p$-weighted estimates, a method introduced by Dafermos and Rodnianski in \cite{Da-Ro}. This allows for a simpler treatment of the curvature estimates.
\item In order to obtain optimal decay for all the connection coefficients, \cite{Bieri} uses two foliations constructed respectively forward and backward. In this paper, we derive decay for the connection coefficients based on only one foliation constructed forward from the initial hypersurface $\Si_0$ and the symmetry axis $\Vphi$. This avoids having to compare forward and backward gauges.
\end{enumerate}
\subsection{Structure of the paper}\label{ssec6.4}
\begin{itemize}
    \item In Section \ref{sec7}, we introduce the geometric set-up and the basic equations.
    \item In Section \ref{sec8}, we state a precise version of the main theorem. We then state intermediate results, and use them to prove the main theorem. The rest of the paper then focuses on the proof of these intermediary results.
    \item In Section \ref{firstboot}, we make bootstrap assumptions and prove first consequences. These consequences will be used frequently in the rest of the paper.
    \item In Section \ref{sec9}, we apply $r^p$--weighted estimates to Bianchi equations to control curvature.
    \item In Section \ref{seck}, we estimate maximal connection coefficients and the lapse function using the elliptic systems satisfied on maximal hypersurfaces.
    \item In Section \ref{sec10}, we estimate the null connection coefficients using the null structure equations.
    \item In Appendix \ref{secA}, we prove Theorem \ref{wonderfulrp}, which provides a unified treatment of nonlinear error terms appearing in the curvature estimates of Section \ref{sec9}.
\end{itemize}
\subsection{Acknowledgements} The author is very grateful to J\'er\'emie Szeftel for his support, discussions, encouragements and patient guidance.
\section{Preliminaries}\label{sec7}
\subsection{Geometric set-up}\label{ssec7.1}
\subsubsection{Maximal foliation}\label{maximalfoliation}
We foliate the spacetime $(\MM,\g)$ by spacelike hypersurfaces $\Si_t$ as level hypersurfaces of a time function $t$. We denote by $T$ the unit, future oriented, normal to $\Si_t$. We then introduce a coordinates system $(x^1,x^2,x^3)$ on $\Si_0$, and extend it to $\M$ by
\begin{align*}
    T(x^i)=0,\qquad\quad i=1,2,3.
\end{align*}
Relative to this foliation, the spacetime metric takes the form:
\begin{equation}\label{metricSi}
    \g=-\phi^2 dt^2+\sum_{i,j=1}^3 g_{ij}dx^idx^j.
\end{equation}
We define the \emph{lapse function} of the foliation by
\begin{equation*}
    \phi(t,x):=(\g(\D t,\D t))^{-1}.
\end{equation*}
Then, we have
\begin{align*}
    T=\phi\D t.
\end{align*}
We define the second fundamental form:
\begin{align}\label{dfk}
    k_{ij}:=-\g(\D_i T,\pr_j).
\end{align}
We denote by $\nab$ and $\sRic$ respectively the induced covariant derivative and the corresponding Ricci curvature tensor on the leaves $\Si_t$. Relative to an orthonormal frame $e_1,e_2,e_3$ tangent to the leaves of the foliation, we have the formulae:
\begin{align*}
    \D_i e_j&=\nab_ie_j-k_{ij} T,\\
    \D_i T&=-k_{ij}e_j,\\
    \D_T e_i&=(\phi^{-1}\nab_i\phi)T,\\
    \D_T T&=(\phi^{-1}\nab_i\phi)e_i,
\end{align*}
where $\nab_Te_i$ denotes the projection of $\D_T e_i$ to the tangent space of $\Si_t$.
\begin{lem}\label{nullk}
    We have the following structure equations of the foliation:
    \begin{align*}
        \pr_t k_{ij}&=-\nab_i\nab_j\phi+\phi(\R_{iTjT}-k_{il}{k^l}_j),\\
        \nab_ik_{jl}-\nab_jk_{il}&=\R_{lTij},\\
        \sRic_{ij}-k_{il}{k^l}_j+k_{ij}\tr k&=\R_{iTjT}+\R_{ij}.
    \end{align*}
\end{lem}
\begin{proof}
See (1.0.3a)--(1.0.3c) in \cite{Ch-Kl}.
\end{proof}
Taking the trace of the formulae in Lemma \ref{nullk}, we obtain the following corollary.
\begin{cor}\label{cornullk}
If $(\M,\g)$ is a solution of Einstein equations, we have the following equations:
    \begin{align*}
        R-|k|^2+(\tr k)^2&=0,\\
        \nab^jk_{ji}-\nab^i\tr k&=0,\\
        \pr_t \tr k&=-\De\phi+\phi|k|^2.
    \end{align*}
\end{cor}
Combining Lemma \ref{nullk} and Corollary \ref{cornullk}, we obtain the following proposition.
\begin{prop}\label{maximal}
We impose a maximal foliation on $\Si_t$ in the sense of Definition \ref{def6.1}. Then, we have:
\begin{itemize}
    \item Maximal constraint equations:
    \begin{align}
    \begin{split}\label{eqtrk=0}
        \tr k&=0,\\
        \nab^jk_{ji}&=0,\\
       R&=|k|^2.
    \end{split}
    \end{align}
\item Evolution equations:
\begin{align}
\begin{split}\label{eqprtgk}
    \pr_t g_{ij}&=-2\phi k_{ij},\\
    \pr_t k_{ij}&=-\nab_i\nab_j\phi+\phi(\sRic_{ij}-2k_{il}{k^l}_j).
\end{split}
\end{align}
\item Lapse equation:
\begin{align}\label{lapseequation}
     \De\phi=\phi|k|^2.
\end{align}
\end{itemize}
\end{prop}
\begin{proof}
See (1.0.11)--(1.0.13) in \cite{Ch-Kl}.
\end{proof}
\subsubsection{Maximal-null foliation}\label{nullfoliation}
Let $O\in\Si_0$. We define $\varPhi$ as the integral curve of $O$ along the vectorfield $T$, called the \emph{symmetry axis}. Starting from $O$, we construct an outgoing null cone $C_0$ in the future of $\Si_0$. The spacetime $\M$ is then divided into
\begin{align*}
    \M=I^+(O)\cup C_0\cup D^+(\Si_0\setminus O),
\end{align*}
where $I^+(O)$ denotes the \emph{future domain of influence} of $O$ and $D^+(\Si_0\setminus O)$ denotes the \emph{future domain of dependence} of $\Si_0\setminus O$.\\ \\
We now construct an outgoing optical function $u$ in $\MM$ as follows. First, to allow ourselves some room, we assume that a spacetime slab $\bigcup_{t\in[-1,0]}\Si_t$ in the past of $\Si_0$ has been constructed. We denote
\begin{align*}
    O_{-1}:=\Vphi\cap \Si_{-1}.
\end{align*}
We define $u$ on $\Vphi$ by
\begin{equation*}
    T(u)=1 \quad \mbox{on }\Vphi,\qquad\quad u(O)=0.
\end{equation*}
Then, for any $p\in\Vphi\cap I^+(O_{-1})$, we construct, emanating from $p$, an outgoing null cone $C_{u(p)}$. Thus, the spacetime $I^+(O_{-1})$ is foliated by the outgoing null cones $C_u$ for $u>u(O_{-1})$. We then extend $u$ to $I^+(O_{-1})$ by assuming that $C_u$ are level sets of $u$ and we denote 
\begin{align*}
    L:=-\grad u\quad\mbox{ in }\; I^+(O_{-1}).
\end{align*}
Note that
\begin{align*}
    I^+(O)\cup C_0\subseteq I^+(O_{-1}).
\end{align*}
It thus remains to define $u$ in $D^+(\Si_0\setminus O)$.\\ \\
For $u(O_{-1})<u\leq 0$, the $2$--spheres $C_u\cap \Si_0$ define a radial foliation on $\Si_0\cap I^+(O_{-1})$ centered at $O$. We define a scalar function $w$ on $\Si_0$ by assuming
\begin{align*}
    w:=-u ,\qquad \mbox{ on }\Si_0\cap I^+(O_{-1}),
\end{align*}
and extending it smoothly to $\Si_0$. We thus fixed a radial foliation on the initial hypersurface $\Si_0$ centered at $O$ by the level sets of $w$. The corresponding leaves are denoted by
\begin{equation}
    S_{(0)}(w_1)=\{p\in\Si_0\big/\, w(p)=w_1\},
\end{equation}
where $w_1\geq 0$ and $w(O)=0$. Then, we define the outgoing optical function $u$ in $D^+(\Si_0\setminus O)$ by the following \emph{Eikonal equation}:
\begin{equation}\label{uSi0}
    \g^{\mu\nu}\pr_\mu u\pr_\nu u=0,\qquad\quad u|_{\Si_0\setminus O}=-w.
\end{equation}
Denoting
\begin{align*}
    L:=-\grad u\qquad \mbox{ in }D^+(\Si_0\setminus K),
\end{align*}
we have the following geodesic equation:
\begin{align*}
    \D_L L=0.
\end{align*}
Thus, $u$ and $L$ are well-defined everywhere in $\MM$. Moreover, $u$ is smooth in $\MM$ by construction.\\ \\
We consider the points on $\Vphi$:
\begin{align}\label{dfucttcu}
    S(t,u_c(t)):=\Si_t\cap\Vphi,\qquad S(t_c(u),u):=C_u\cap\Vphi,
\end{align}
as spheres of radius $0$. Thus, the spacetime $\M$ is foliated by $2$--spheres:
\begin{align*}
    S(t,u):=\Si_t\cap C_u,\qquad u\in\mathbb{R},\quad t\geq 0.
\end{align*}
We define the \emph{initial layer region}:
\begin{equation}\label{dfinitiallayer}
    \kk_{(0)}:=\M\cap\{0\leq t\leq 2\}=J^-(\Si_2)\cap J^+(\Si_0),
\end{equation}
where $J^\pm(\Si)$ denotes the \emph{causal future} (resp. \emph{past}) of $\Si$. Note that we have
\begin{align*}
    u>0\;\;\mbox{ in }I^+(O),\qquad u=0\;\;\mbox{ on }C_0,\qquad u<0 \;\;\mbox{ in }D^+(\Si_0\setminus O),
\end{align*}
see Figure \ref{foliation} for a picture.
\begin{figure}
    \centering
    \includegraphics[width=15.2cm]{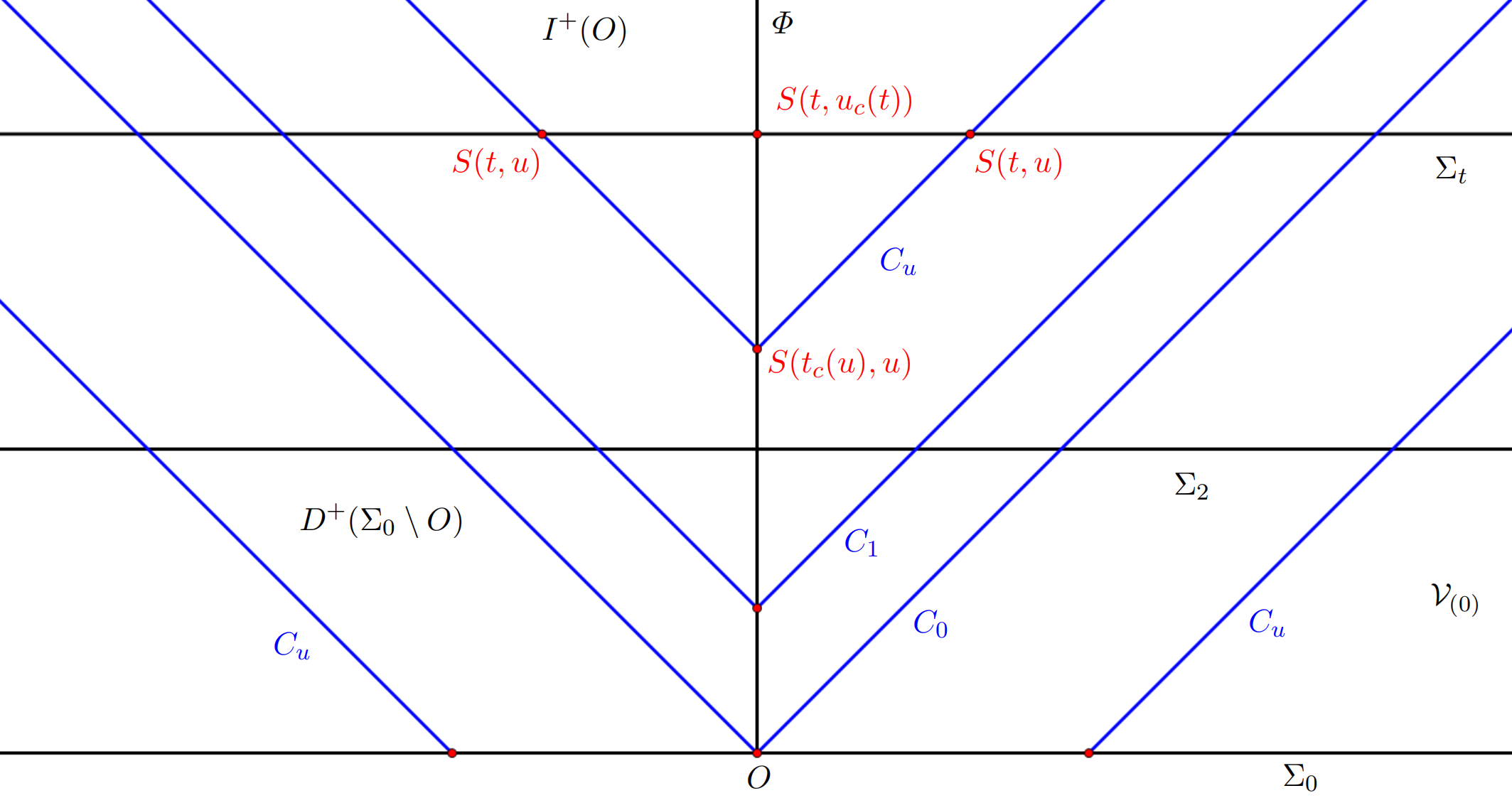}
    \caption{Maximal-null foliation of the spacetime $\M$}
    \label{foliation}
\end{figure}
We also define for $t\geq 2$
\begin{equation*}
    V_t:=J^-(\Si_t)\cap J^+(\Si_2).
\end{equation*}
For any $t\geq 2$, we define the \emph{exterior region}:
\begin{align*}
    \Ve_t:=\{p\in V_t\big/\, u\leq 1\},
\end{align*}
and the \emph{interior region}:
\begin{align*}
    \Vi_t:=\{p\in V_t\big/\, u\geq 1\}.
\end{align*}
We then define the null vectorfield $e_4$ and the \emph{null lapse function} $a$ by
\begin{equation*}
    a:=|\nab u|^{-1},\qquad\quad e_4:=aL.
\end{equation*}
By definition, we have
\begin{align*}
    0=\g(\D u,\D u)=|\nab u|^2-|T(u)|^2=a^{-2}-|T(u)|^2,
\end{align*}
which implies
\begin{align*}
    a=T(u)^{-1}.
\end{align*}
Thus, we have in particular $a=1$ on $\Vphi$. Next, we define
\begin{equation*}
    N:=e_4-T,
\end{equation*}
which is a vectorfield on $\MM\setminus\Vphi$. By a direct computation, we have
\begin{align*}
    \g(T,N)&=\g(T,e_4-T)=a\g(T,L)-\g(T,T)=-a\g(T,\D u)+1=-aT(u)+1=0,\\
    \g(N,N)&=\g(e_4-T,e_4-T)=-2\g(e_4,T)+\g(T,T)=-2a\g(L,T)-1=2aT(u)-1=1.
\end{align*}
Thus, $N$ denotes the unit vectorfield tangent to $\Si_t$, oriented towards infinity and orthogonal to the leaves $S(t,u)$.\\ \\
We denote by $\th$ the second fundamental form of the surfaces $S(t,u)$ relative to $\Si_t$:
\begin{equation*}
    \th_{AB}:=\g(\D_A N, e_B).
\end{equation*}
We then define the null vectorfield $e_3$ by
\begin{equation*}
   e_3:=T-N.
\end{equation*}
By a direct computation, we have
\begin{align*}
    \g(e_3,e_3)=\g(T-N,T-N)=0,\qquad\quad \g(e_3,e_4)=\g(T-N,T+N)=0.
\end{align*}
On a given $2$--sphere $S(t,u)$, we choose a local orthonormal frame $(e_1,e_2)$ and we call $(e_1,e_2,e_3,e_4)$ a null frame.\footnote{Note that the null frame $(e_1,e_2,e_3,e_4)$ is well-defined on $\M\setminus\Vphi$.}
\begin{rk}
As a convention, throughout the paper, we use capital Latin letters $A,B,C,...$ to denote an index from $1$ to $2$, lower case Latin letters $i,j,k,...$ to denote an index from $1$ to $3$ and Greek letters $\a,\b,\ga,...$ to denote an index from $1$ to $4$, e.g. $e_A$ denotes either $e_1$ or $e_2$.
\end{rk}
The spacetime metric $\bg$ induces a Riemannian metric $\slg$ on $S(t,u)$. We use $\nabs$ to denote the Levi-Civita connection of $\slg$ on $S(t,u)$. We recall the null decomposition of the Ricci coefficients and curvature components of the null frame $(e_1,e_2,e_3,e_4)$ as follows:
\begin{align}
\begin{split}\label{defga}
\chib_{AB}&=\g(\D_A e_3, e_B),\qquad\quad \chi_{AB}=\g(\D_A e_4, e_B),\\
\xib_A&=\frac 1 2 \g(\D_3 e_3,e_A),\qquad\quad\,\,  \xi_A=\frac 1 2 \g(\D_4 e_4, e_A),\\
\omb&=\frac 1 4 \g(\D_3e_3 ,e_4),\qquad\quad \,\,\,\;\, \om=\frac 1 4 \g(\D_4 e_4, e_3), \\
\etab_A&=\frac 1 2 \g(\D_4 e_3, e_A),\qquad\quad\;  \eta_A=\frac 1 2 \g(\D_3 e_4, e_A),\\
\ze_A&=\frac 1 2 \g(\D_{e_A}e_4, e_3),
\end{split}
\end{align}
and
\begin{align}
\begin{split}\label{defr}
\a_{AB} &=\R(e_A, e_4, e_B, e_4), \qquad \,\,\,\aa_{AB} =\R(e_A, e_3, e_B, e_3), \\
\b_{A} &=\frac 1 2 \R(e_A, e_4, e_3, e_4), \qquad\,\, \;\bb_{A}=\frac 1 2 \R(e_A, e_3, e_3, e_4),\\
\rho&= \frac 1 4 \R(e_3, e_4, e_3, e_4), \qquad\,\;\;\,\,\,\, \si =\frac{1}{4}{^*\R}(e_3,e_4,e_3,e_4),
\end{split}
\end{align}
where $^*\R$ denotes the Hodge dual of $\R$.\\ \\
The null second fundamental forms $\chi, \chib$ are further decomposed in their traces $\trch$ and $\trchb$, and traceless parts $\hch$ and $\hchb$:
\begin{align*}
\trch&:=\slg^{AB}\chi_{AB},\qquad\quad \,\hch_{AB}:=\chi_{AB}-\frac{1}{2}\slg_{AB}\trch,\\
\trchb&:=\slg^{AB}\chib_{AB},\qquad\quad \, \hchb_{AB}:=\chib_{AB}-\frac{1}{2}\slg_{AB}\trchb.
\end{align*}
We define the horizontal covariant operator $\nabs$ as follows:
\begin{equation*}
\nabs_X Y:=\D_X Y-\frac{1}{2}\chib(X,Y)e_4-\frac{1}{2}\chi(X,Y)e_3.
\end{equation*}
We also define $\nabs_4X$ and $\nabs_3X$ to be the horizontal projections of $\D_4X$ and $\D_3 X$:
\begin{align*}
\nabs_4X&:=\D_4X-\frac{1}{2}\g(X,\D_4e_3)e_4-\frac{1}{2}\g(X,\D_4e_4)e_3,\\
\nabs_3X&:=\D_3X-\frac{1}{2}\g(X,\D_3e_3)e_4-\frac{1}{2}\g(X,\D_3e_4)e_3.
\end{align*}
A tensor field $\psi$ defined on $\M$ is called tangent to $S(t,u)$ if it is a priori defined on the spacetime $\M$ and all the possible contractions of $\psi$ with either $e_3$ or $e_4$ are zero. We use $\nabs_3\psi$ and $\nabs_4\psi$ to denote the projection to $S(t,u)$ of usual derivatives $\D_3\psi$ and $\D_4\psi$.\\ \\
As a direct consequence of \eqref{defga}, we have the Ricci formulae:
\begin{align}
\begin{split}\label{ricciformulas}
    \D_A e_B&=\nabs_A e_B+\frac{1}{2}\chi_{AB} e_3+\frac{1}{2}\chib_{AB}e_4,\\
    \D_A e_3&=\chib_{AB}e_B+\ze_A e_3,\\
    \D_A e_4&=\chi_{AB}e_B-\ze_A e_4,\\
    \D_3 e_A&=\nabs_3 e_A+\eta_A e_3+\xib_A e_4,\\
    \D_4 e_A&=\nabs_4 e_A+\etab_A e_4+\xi_A e_4,\\
    \D_3 e_3&=-2\omb e_3+2\xib_B e_B,\\ 
    \D_3 e_4&=2\omb e_4+2\eta_B e_B,\\
    \D_4 e_4&=-2\om e_4+2\xi_B e_B,\\
    \D_4 e_3&=2\om e_3+2\etab_B e_B.
\end{split}
\end{align}
Next, we decompose $k$ as follows:
\begin{align}\label{boundarydec}
    \de:=k_{NN},\qquad \eps_A:=k_{AN},\qquad \ka_{AB}:=k_{AB}.
\end{align}
The following identities hold for a maximal-null foliation:
\begin{align}
    \begin{split}\label{6.6}
    \chi_{AB}&=\th_{AB}-\ka_{AB},\qquad\qquad\,\,\chib_{AB}=-\th_{AB}-\ka_{AB},\\
    \xi_A&=0,\qquad\qquad\qquad\qquad\quad\;\, \xib_A=\nabs_A\log\phi-\nabs_A\log a,\\
    \eta_A&=\nabs_A\log a+\eps_A,\qquad\qquad \etab_A=\nabs_A\log\phi-\eps_A,\\
    2\om&=-\nab_N\log\phi+\de,\qquad\quad\;\;\, 2\omb=\nab_N\log\phi+\de,\\
    \ze_A&=\eps_A,
    \end{split}
\end{align}
see (7.5.2b) in \cite{Ch-Kl}.\\ \\
Note that we have the following types of manifolds in this paper: the spacetime region $\M$, the spacelike hypersurfaces $\Si_t$, and the spheres $S(t,u)$. Every type of manifold has its metric, Levi-Civita connection and curvature tensor:
\begin{align*}
 (\M,\g,\D,\R),\qquad (\Si_t,g,\nab,R),\qquad (S(t,u),\slg,\nabs,\slr).
\end{align*}
\subsection{Integral formulae}\label{ssecave}
We define the $S$-average of scalar functions.
\begin{df}\label{average}
Given any scalar function $f$, we denote its average and its average free part by
\begin{equation*}
    \ov{f}:=\frac{1}{|S|}\int_{S}f\,d\slg,\qquad \fc:=f-\overline{f},
\end{equation*}
where $|S|$ denotes the area of $S$.
\end{df}
The following lemma follows immediately from the definition.
\begin{lem}\label{chav}
For any two scalar functions $u$ and $v$, we have
\begin{equation*}
\overline{uv}=\overline{u}\,\overline{v}+\overline{\widecheck{u}\widecheck{v}},
\end{equation*}
and
\begin{equation*}
uv-\overline{uv}=\widecheck{u}\overline{v}+\overline{u}\widecheck{v}+\left(\widecheck{u}\widecheck{v}-\overline{\widecheck{u}\widecheck{v}}\right).
\end{equation*}
\end{lem}
We recall the following integral formulae.
\begin{lem}\label{dint}
For any scalar function $f$, the following equations hold:
\begin{align*}
    \phi e_4\left(\int_{S(t,u)} f d\slg\right) &=\int_{S(t,u)} \left(\phi e_4(f) + \phi\trch f\right) d\slg,\\
    aN\left(\int_{S(t,u)} f d\slg\right) &= \int_{S(t,u)} \left(aN(f) + a\tr\th f\right) d\slg.
\end{align*}
Taking $f=1$, we obtain
\begin{equation}
    e_4(r)=\frac{\overline{\phi\tr\chi}}{2\phi}r,\qquad N(r)=\frac{\ov{a\tr\th}}{2a}r, \label{e3e4r}
\end{equation}
where $r$ is the \emph{area radius} defined by
\begin{equation*}
    r(t,u):=\sqrt{\frac{|S(t,u)|}{4\pi}}.
\end{equation*}
\end{lem}
\begin{proof}
See Lemma 2.26 in \cite{kl-ro}.
\end{proof}
\subsection{Hodge systems in \texorpdfstring{$2D$}{}--geometry}\label{ssec7.2}
\begin{df}\label{tensorfields}
For tensor fields defined on a $2$--sphere $S$, we denote by $\sfr_0:=\sfr_0(S)$ the set of pairs of scalar functions, $\sfr_1:=\sfr_1(S)$ the set of $1$--forms and $\sfr_2:=\sfr_2(S)$ the set of symmetric traceless $2$--tensors.
\end{df}
\begin{df}\label{def7.2}
Given $\xi\in\sfr_1$, we define its Hodge dual $^*\xi$ by
\begin{equation*}
    {^*\xi}_A := \in_{AB}\xi^B,
\end{equation*}
where $\in_{AB}$ denotes the \emph{volume element} on $S$. Clearly $^*\xi \in \sfr_1$ and
\begin{equation*}
    ^*(^*\xi)=-\xi.
\end{equation*}
Given $U \in \sfr_2$, we define its Hodge dual $^*U$ by
\begin{equation*}
{^*U}_{AB}:= \in_{AC} {U^C}_B.
\end{equation*}
Observe that $^*U\in\sfr_2$ and
\begin{equation*}
     ^*(^*U)=-U.
\end{equation*}
\end{df}
\begin{df}
    Given $\xi,\eta\in\sfr_1$, we denote
\begin{align*}
    \xi\cdot \eta&:= \de^{AB}\xi_A \eta_B,\\
    \xi\wedge \eta&:= \in^{AB} \xi_A \eta_B,\\
    (\xi\hot \eta)_{AB}&:=\xi_A \eta_B +\xi_B \eta_A -\de_{AB}\xi\cdot \eta.
\end{align*}
Given $\xi\in \sfr_1$, $U\in \sfr_2$, we denote
\begin{align*}
    (\xi \cdot U)_A:= \de^{BC} \xi_B U_{AC}.
\end{align*}
Given $U,V\in \sfr_2$, we denote
\begin{align*}
    (U\wedge V)_{AB}:=\in^{AB}U_{AC}V_{CB}.
\end{align*}
\end{df}
\begin{df}
    For a given $\xi\in\sfr_1$, we define the following differential operators:
    \begin{align*}
        \sdivs\xi&:= \de^{AB} \nabs_A\xi_B,\\
        \curls\xi&:= \in^{AB} \nabs_A \xi_B,\\
        (\nabs\hot\xi)_{AB}&:=\nabs_A \xi_B+\nabs_B \xi_A-\slg_{AB}(\sdivs\xi).
    \end{align*}
\end{df}
\begin{df}\label{hodgeop}
We define the following Hodge type operators, as introduced in Section 2.2 in \cite{Ch-Kl}:
    \begin{itemize}
        \item $\sld_1$ takes $\sfr_1$ into pairs of $\sfr_0$ and is given by:
        \begin{equation*}
            \sld_1\xi :=(\sdivs\xi,\curls \xi),
        \end{equation*}
        \item $\sld_2$ takes $\sfr_2$ into $\sfr_1$ and is given by:
        \begin{equation*}
            (\sld_2 U)_A := \nabs^{B} U_{AB}, 
        \end{equation*}
        \item $\sld_1^*$ takes pairs of $\sfr_0$ into $\sfr_1$ and is given by:
        \begin{align*}
            \sld_1^*(f,f_*)_{A}:=-\nabs_A f +{\in_A}^{B}\nabs_B f_*,
        \end{align*}
        \item $\sld_2^*$ takes $\sfr_1$ into $\sfr_2$ and is given by:
        \begin{align*}
            \sld_2^* \xi := -\frac{1}{2}\nabs\hot\xi.
        \end{align*}
    \end{itemize}
\end{df}
\begin{prop}\label{Hodge2Diden}
We have the following identities:
\begin{align}
    \begin{split}\label{dddd}
        \sld_1^*\sld_1&=-\Des_1+\mathbf{K},\qquad\qquad  \sld_1\sld_1^*=-\Des_{{0}},\\
        \sld_2^*\sld_2&=-\frac{1}{2}\Des_2+\mathbf{K},\qquad\quad\; \sld_2\sld_2^*=-\frac{1}{2}(\Des_1+\mathbf{K}),
    \end{split}
\end{align}  
where $\mathbf{K}$ denotes the Gauss curvature on $S$ and $\Des_k$ denotes the Laplacian on $\sk_k$. 
\end{prop}
\begin{proof}
    See (2.2.2) in \cite{Ch-Kl}. 
\end{proof}
\begin{df}\label{dfdkb}
We define the weighted angular derivatives $\dkb$ as follows:
\begin{align*}
    \dkb U &:= r\sld_2 U,\qquad\quad\;\;\quad\quad\,\, \forall U\in \sk_2,\\
    \dkb \xi&:= r\sld_1 \xi,\qquad\qquad\quad\quad\,\, \forall \xi\in \sk_1,\\
    \dkb (f,f_*)&:= r\sld_1^*(f,f_*),\qquad\forall (f,f_*)\in \sk_0\times\sk_0.
\end{align*}
We denote for any tensor $h\in\sk_k$, $k=0,1,2$,
\begin{equation*}
    h^{(0)}:=h,\qquad h^{(1)}:=(h,\dkb h),\qquad h^{(2)}:=(h,\dkb h,\dkb^2 h).
\end{equation*}
\end{df}
\begin{prop}\label{commdkb}
We have the following commutator identities:
\begin{itemize}
    \item For any $U\in\sk_2$
    \begin{align*}
        \dkb\sld_2U=\sld_1\dkb U.
    \end{align*}
    \item For any $\xi\in\sk_1$
    \begin{align*}
        \dkb\sld_1\xi&=\sld_1^*\dkb\xi,\\
        \dkb(2\sld_2^*\xi)&=\sld_1^*\dkb\xi-2r\mathbf{K}\xi.
    \end{align*}
    \item For any $(f,f_*)\in\sk_0$
    \begin{align*}
        \dkb\sld_1^*(f,f_*)=\sld_1\dkb(f,f_*).
    \end{align*}
\end{itemize}
\end{prop}
\begin{proof}
    We recall from \eqref{dddd} that for any $\xi\in\sk_1$
    \begin{align*}
        2\dkb \sld_2^* \xi =2r\sld_2\sld_2^*\xi=-r(\Des_1+\mathbf{K})\xi=r(\sld_1^*\sld_1-2\mathbf{K})\xi=\sld_1^*\dkb\xi-2r\mathbf{K}\xi,
    \end{align*}
    which implies the third identity. The other identities follow directly from Definition \ref{dfdkb}. This concludes the proof of Proposition \ref{commdkb}.
\end{proof}
\begin{df}\label{Lpnorms}
For a tensor field $f$ on a $2$--sphere $S$, we denote its $L^p$--norm as follows:
\begin{equation}
    |f|_{p,S}:= \left(\int_S |f|^p \right)^\frac{1}{p}.
\end{equation}
\end{df}
\subsection{Null structure equations}
We recall the null structure equations, see Proposition 7.4.1 in \cite{Ch-Kl}.
\begin{prop}\label{nulles}
We have the null structure equations:
\begin{align*}
\nabs_4\hch+\trch\,\hch&=\nabs\hot\xi-2\om\hch+(\eta+\etab+2\ze)\hot\xi-\a,\\
\nabs_4\trch+\frac{1}{2}(\tr\chi)^2&=2\sdivs\xi-2\om\trch+2\xi\c(\eta+\etab+2\ze)-|\hch|^2,\\
\nabs_3\hchb+\trchb\,\hchb&=\nabs\hot\xib-2\omb\hchb+(\eta+\etab-2\ze)\hot\xib-\aa,\\
\nabs_3\trchb+\frac{1}{2}(\trchb)^2&=2\sdivs\xib-2\omb\tr\unchi+2\xib\c(\eta+\etab-2\ze)-|\widehat{\unchi}|^2,\\
\nabs_4\hchb+\frac{1}{2}\trch\,\hchb&=\nabs\hot\etab+2\om\hchb-\frac{1}{2}\trchb\,\hch+\xib\hot\xib+\etab\hot\etab,\\
\nabs_4\trchb+\frac{1}{2}\trch\trchb&=2\om\trchb-\hch\c\hchb+2\sdivs\etab+2(\xib\c\xi+|\etab|^2)+2\rho,\\
\nabs_3\hch+\frac{1}{2}\trchb\,\hch&=\nabs\hot\eta+2\omb\widehat{\chi}-\frac{1}{2}\trch\,\hchb+\xi\hot\xi+\eta\hot\eta,\\
\nabs_3\trch+\frac{1}{2}\trchb\trch&=2\omb\trch-\hch\c\hchb+2\sdivs\eta+2(\xi\c\xib+|\eta|^2)+2\rho,\\
\nabs_4\eta-\nabs_3\xi&=-4\omb\xi-\chi\cdot(\eta-\etab)-\b,\\
\nabs_3\etab-\nabs_4\xib&=-4\om\xib-\chib\cdot(\etab-\eta)+\bb, \\
\nabs_3\ze&=-2\nabs\omb-\chib\c(\ze+\eta)+2\omb(\ze-\eta)+\chi\c\xib+2\om\xib-\bb,\\
\nabs_4\ze&=2\nabs\om+\chi\c(-\ze+\etab)+2\om(\ze+\etab)-\chib\c\xi-2\omb\xi-\b,\\
\nabs_3\om+\nabs_4\omb&=\xi\c\xib+\ze\c(\eta-\etab)-\eta\c\etab+4\om\omb+\rho.
\end{align*}
Also, we have the Codazzi equations:
\begin{align}
\begin{split}\label{codazzi}
\sdivs\hch=\frac{1}{2}\nabs\tr\chi-\zeta\cdot\left(\hch-\frac{1}{2}\trch\right)-\b,\\ 
\sdivs\hchb=\frac{1}{2}\nabs\tr\unchi+\zeta\cdot\left(\hchb-\frac{1}{2}\trchb\right)+\bb,
\end{split}
\end{align}
the torsion equation:
\begin{align}
\begin{split}\label{torsion}
\curls\eta&=\si+\frac{1}{2}\widehat\unchi\wedge\widehat\chi-\xi\wedge\xib,\\
\curls\etab&=-\si-\frac{1}{2}\widehat\unchi\wedge\widehat\chi+\xi\wedge\xib,\\
\curls\xi&=\xi\wedge(\eta+\etab+2\ze),\\
\curls\xib&=\xib\wedge(\eta+\etab-2\ze),
\end{split}
\end{align}
and the Gauss equation:
\begin{equation}
    \mathbf{K}=-\frac{1}{4}\tr\chi\tr\unchi+\frac{1}{2}\widehat\unchi\cdot\widehat\chi-\rho. \label{gauss}
\end{equation}
\end{prop}
\subsection{Bianchi equations}\label{ssec7.4}
We recall the Bianchi equations, see Proposition 7.3.2 in \cite{Ch-Kl}.
\begin{prop}\label{Bianchieq}
The Bianchi equations take the following form
\begin{align*}
\nabs_3\a+\frac{1}{2}\trchb\,\a&=\nabs\hot\b+4\omb\a-3(\hch\rho+{^*\hch}\si)+(\ze+4\eta)\hot\b,\\
\nabs_4\b+2\trch\,\b&=\sdivs\a-2\om\b+(2\ze+\etab)\c\a+3(\xi\rho+{^*\xi}\si),\\
\nabs_3\b+\trchb\,\b&=\nabs\rho+{^*\nabs}\si+2\omb\b+\xib\c\a+2\hch\cdot\bb+3(\eta\rho+{^*\eta}\si),\\
\nabs_4\rho+\frac{3}{2}\trch\,\rho&=\sdivs\b-\frac{1}{2}\hchb\c\a+\ze\c\b+2(\etab\c\b-\xi\c\bb),\\
\nabs_3\rho+\frac{3}{2}\trchb\,\rho&=-\sdivs\bb-\frac{1}{2}\hch\cdot\aa+\ze\cdot\bb+2(\xib\c\b-\eta\cdot\bb),\\
\nabs_4\si+\frac{3}{2}\trch\,\si&=-\curls\b+\frac{1}{2}\hchb\c{^*\a}-\ze\c{^*\b}-2(\etab\cdot{^*\beta}+2\xi\c{^*\bb}),\\
\nabs_3\si+\frac{3}{2}\trchb\,\si&=-\curls\bb-\frac{1}{2}\hch\c{^*\aa}+\ze\c{^*\bb}-2(\eta\c{^*\bb}+\xib\c{^*\b}),\\
\nabs_4\bb+\trch\,\bb&=-\nabs\rho+{^*\nabs\si}+2\om\bb+2\hchb\c\b-\xi\c\aa-3(\etab\rho-{^*\etab}\si),\\
\nabs_3\bb+2\trchb\,\bb&=-\sdivs\aa-2\omb\bb+(2\ze-\eta)\cdot\aa+3(-\xib\rho+{^*\xib}\si),\\
\nabs_4\aa+\frac{1}{2}\trch\,\aa&=-\nabs\hot\bb+4\om\aa-3(\hchb\rho-{^*\hchb}\si)+(\ze-4\etab)\hot\bb.
\end{align*}
\end{prop}
\subsection{Commutation identities}
We recall the following commutation lemma.
\begin{lem}\label{comm}
Let $U_{A_1...A_k}$ be an $S$-tangent $k$-covariant tensor on $(\M,\bg)$. Then
\begin{align*}
    [\nabs_4,\nabs_B]U_{A_1...A_k}&=-\chi_{BC}\nabs_C U_{A_1...A_k}+F_{4BA_1...A_k},\\
    F_{4BA_1...A_k}&:=\xi_B\nabs_3U_{A_1,...,A_k}+(\zeta_B+\etab_B)\nabs_4 U_{A_1...A_k}\\
    &+\sum_{i=1}^k(\chib_{A_iB}\,\xi_C-\chib_{BC}\,\xi_{A_i}+\chi_{A_iB}\,\etab_C-\chi_{BC}\,\etab_{A_i}+\in_{A_iC}{^*\beta}_B)U_{A_1...C...A_k},\\
    [\nabs_3,\nabs_B]U_{A_1...A_k}&=-\chib_{BC}\nabs_C U_{A_1...A_k}+F_{3BA_1...A_k},\\
    F_{3BA_1...A_k}&:=\xib_B\nabs_4U_{A_1,...,A_k}+(\eta_B-\ze_B)\nabs_3U_{A_1...A_k}\\
    &+\sum_{i=1}^k(\chi_{A_iB}\xib_C-\chi_{BC}\xib_{A_i}+\chib_{A_iB}\,\eta_C-\chib_{BC}\,\eta_{A_i}+\in_{A_iC}{^*\bb}_B)U_{A_1...C...A_k},\\
    [\nabs_3,\nabs_4]U_{A_1...A_k}&=F_{34A_1...A_k},\\
    F_{34A_1...A_k}&:=-2\om\nabs_3U+2\omb\nabs_4 U+(\eta_B-\etab_B)\nabs_B U_{A_1...A_k}\\
    &+2\sum_{i=1}^k(\xi_{A_i}\,\xib_{C}-\xi_{C}\,\xib_{A_i}+\etab_{A_i}\,\eta_C-\etab_{A_i}\,\eta_C+\in_{A_iC}\si)U_{A_1...C...A_k}.
\end{align*}
\end{lem}
\begin{proof}
See Lemma 7.3.3 in \cite{Ch-Kl}.
\end{proof}
\subsection{The electric-magnetic decomposition of curvature}
\begin{prop}\label{EHiden}
We define the electric-magnetic decomposition of $\R$:
\begin{align}\label{dfEH}
    E_{ij}:=\R(T,e_i,T,e_j),\qquad H_{ij}:={^*\R}(T,e_i,T,e_j), \qquad i,j=1,2,3.
\end{align}
Then, the following identities hold for a maximal foliation:
\begin{align}
\begin{split}\label{EHdecomposition}
    E_{AB}&=\frac{1}{4}\a_{AB}+\frac{1}{4}\aa_{AB}-\frac{1}{2}\rho\,\de_{AB},\qquad\; H_{AB}=-\frac{1}{4}{^*\a}_{AB}+\frac{1}{4}{^*\aa}_{AB}-\frac{1}{2}\si\,\de_{AB},\\
    E_{AN}&=\frac{1}{2}\bb_A+\frac{1}{2}\b_A,\qquad\qquad\qquad\quad\;\;\; H_{AN}=\frac{1}{2}{^*\bb}_A-\frac{1}{2}{^*\b}_{A},\\
    E_{NN}&=\rho,\qquad\qquad\qquad\qquad\qquad\qquad\; H_{NN}=\si.
\end{split}
\end{align}
\end{prop}
\begin{proof}
    See (7.3.3e) in \cite{Ch-Kl}.
\end{proof}
\begin{prop}\label{EHidentity}
The following equations hold for a maximal foliation:
    \begin{align*}
        \nab_ik_{jm}-\nab_jk_{im}&={\in_{ij}}^lH_{lm},\\
        \sRic_{ij}-k_{il}{k^l}_j&=E_{ij}.
    \end{align*}
\end{prop}
\begin{proof}
    See (1.0.8a)--(1.0.8b) in \cite{Ch-Kl}.
\end{proof}
\begin{df}\label{dfdivcurl}
Let $\Si$ be a $3$--dimensional manifold diffeomorphic to $\mathbb{R}^3$. We define the following differential operators on $\Si$:
    \begin{itemize}
    \item For any scalar function $\phi$ on $\Si$, we define
    \begin{align*}
        \De\phi=\nab^i\nab_i\phi.
    \end{align*}
        \item For any $1$--form $\xi$ on $\Si$, we define
        \begin{align*}
            \sdiv\xi&:=\nab^i\xi_i,\\
            (\curl\xi)_i&:={\in_{i}}^{lm}\nab_l\xi_m.
        \end{align*}
        \item For any traceless symmetric $2$--tensor $U$ on $\Si$, we define
        \begin{align*}
            (\sdiv U)_i&:=\nab^jU_{ij},\\
            (\curl U)_{ij}&:=\frac{1}{2}\left({\in_{i}}^{lm}\nab_{l}U_{mj}+{\in_{j}}^{lm}\nab_{l}U_{mi}\right).
        \end{align*}
    \end{itemize}
\end{df}
\begin{prop}\label{EHmaxwell}
Let $E$, $H$ be the electric-magnetic decomposition defined in \eqref{dfEH}. Then, the following identities hold:
    \begin{align*}
        \sdiv E&=k\wedge H,\\
        \sdiv H&=-k\wedge E,\\
        -\Lieh_T H+\curl E&=-\nab\log\phi\wedge E-\frac{1}{2}k\times H,\\
        \Lieh_T E+\curl H&=-\nab\log\phi\wedge H+\frac{1}{2}k\times E,
    \end{align*}
    where 
    \begin{align}
    \begin{split}\label{dfLieh}
        (\Lieh_T E)_{ij}&:=(\Lie_T E)_{ij}-\frac{2}{3}(k\c E)g_{ij},\\
        (\Lieh_T H)_{ij}&:=(\Lie_T H)_{ij}-\frac{2}{3}(k\c H)g_{ij},
    \end{split}
    \end{align}
    and where $\Lie_X U$ denotes the Lie derivative of a tensor $U$ relative to a vectorfield $X$.
\end{prop}
\begin{proof}
    See Proposition 7.2.1 in \cite{Ch-Kl}.
\end{proof}
\section{Main theorem}\label{sec8}
\subsection{Preliminary definitions}
\begin{df}\label{dfBEF}
For any $t\geq 2$ and $-\infty\leq u_1<u_2\leq +\infty$, we denote
\begin{align*}
    V_t(u_1,u_2):=J^+(\Si_2)\cap J^-(\Si_t)\cap\{u_1\leq u\leq u_2\},
\end{align*}
and for $V=V_t(u_1,u_2)$ and $u_1\leq u\leq u_2$
\begin{align*}
    \cuv:=C_u\cap V,\qquad\qquad\quad \ucuv:=\Si_t\cap V.
\end{align*}
We then define for any tensorfield $\Psi$:
    \begin{align*}
        B_p^q[\Psi](u_1,u_2)&:=\int_V r^{p-1}|\Psi^{(q)}|^2,\\
        E_p^q[\Psi](u)&:=\int_{\cuv} r^p|\Psi^{(q)}|^2,\\
        F_p^q[\Psi](u_1,u_2)&:=\int_{\ucuv} r^p|\Psi^{(q)}|^2,\\
        D_\ell^q[\Psi](u)&:=\sup_{S\subseteq\cuv}\sup_{p\in[2,4]}|r^{\frac{\ell+3}{2}-\frac{2}{p}}\Psi^{(q)}|_{p,S}^2,\\
        \Dinf_\ell[\Psi](u)&:=\sup_{\cuv}|r^\frac{\ell+3}{2}\Psi|^2,
    \end{align*}
where $q=0,1,2,3$ and where the notation $\Psi^{(q)}$ has been introduced in Definition \ref{dfdkb}. We also use the following shorthand notations:
\begin{align*}
    V_t(u):=V_t(-\infty, u),\qquad X_p^q[\Psi](u):=X_p^q[\Psi](-\infty,u),\quad \mbox{ where } \; X=B,F.
\end{align*}
We also denote
\begin{align*}
    X_p[\Psi](u_1,u_2):=X_p^0[\Psi](u_1,u_2),\quad\mbox{ where }\; X=B,E,F,D.
\end{align*}
Moreover, we denote
\begin{align*}
XY_p^q[\Psi](u_1,u_2):=X_p^q[\Psi](u_1,u_2)+Y_p^q[\Psi](u_1,u_2)\quad\mbox{ where }\; X,Y\in\{B,E,F\},
\end{align*}
and
\begin{equation*}
    E_p^q[\Psi](u_1,u_2):=\sup_{u\in[u_1,u_2]}E_p^q[\Psi](u).
\end{equation*}
\end{df}
\begin{df}\label{dfdecaynorms}
Let $2<t_*<+\infty$ and $q=0,1,2,3$. We define:
\begin{align*}
\BBe_{p,d}^q[\Psi]&:=\sup_{2\leq t\leq t_*}\sup_{u\leq 1}\ujp^dB_p^q[\Psi](u),\\
\EEe_{p,d}^q[\Psi]&:=\sup_{2\leq t\leq t_*}\sup_{u\leq 1}\ujp^dE_p^q[\Psi](u),\\
\FFe_{p,d}^q[\Psi]&:=\sup_{2\leq t\leq t_*}\sup_{u\leq 1}\ujp^dF_p^q[\Psi](u),\\
\DDe_{p,d}^q[\Psi]&:=\sup_{2\leq t\leq t_*}\sup_{u\leq 1}\ujp^dD_p^q[\Psi](u),\\
^e\DDinf_{p,d}[\Psi]&:=\sup_{2\leq t\leq t_*}\sup_{u\leq 1}\ujp^d\Dinf_{p}[\Psi](u),
\end{align*}
where
\begin{equation}\label{japenese}
    \ujp:=\sqrt{1+|u|^2}.
\end{equation}
We also define:
\begin{align*}
\BBi_{p,d}^q[\Psi]&:=\sup_{2\leq t\leq t_*}\sup_{1\leq u_1\leq u_2\leq u_c(t)}u_1^dB_p^q[\Psi](u_1,u_2),\\
\EEi_{p,d}^q[\Psi]&:=\sup_{2\leq t\leq t_*}\sup_{1\leq u\leq u_c(t)}u^dE_p^q[\Psi](u),\\
\FFi_{p,d}^q[\Psi]&:=\sup_{2\leq t\leq t_*}\sup_{1\leq u_1\leq u_2\leq u_c(t)}u_1^dF_p^q[\Psi](u_1,u_2),\\
\DDi_{p,d}^q[\Psi]&:=\sup_{2\leq t\leq t_*}\sup_{1\leq u\leq u_c(t)}u^dD_p^q[\Psi](u),\\
^i\DDinf_{p,d}[\Psi]&:=\sup_{2\leq t\leq t_*}\sup_{1\leq u\leq u_c(t)}u^d\Dinf_{p,d}[\Psi](u).
\end{align*}
Moreover, we denote
\begin{align*}
    \XX_{p,d}^q[\Psi]:={^e\XX}_{p,d}^q[\Psi]+{^i\XX}_{p,d}^q[\Psi],\quad\mbox{ where }\; \XX\in\{\BB,\,\EE,\,\FF,\,\DD\}.
\end{align*}
We also use the following shorthand notations:
\begin{align*}
    \XX_{p,d}[\Psi]:=\XX_{p,d}^0[\Psi],\qquad \XX_{p}^q[\Psi]:=\XX_{p,0}^q[\Psi],\qquad \XX_p[\Psi]:=\XX_{p,0}^0[\Psi],\quad\mbox{ where }\; \XX\in\{\BB,\,\EE,\,\FF,\,\DD\}.
\end{align*}
\end{df}
\begin{df}\label{renorr}
We define the renormalized quantity $\slu$ and mass aspect function $\mu$:\footnote{The quantities $\slu$ and $\mu$ appeared first in \cite{Ch-Kl}.}
\begin{equation}\label{dfmu}
\slu:=\nabs\trch+\trch\,\ze,\qquad \mu:=-\sdivs\eta-\rho+\frac{1}{2}\hch\c\hchb.
\end{equation}
We also define the following renormalized curvature components:\footnote{This renormalization appeared first in \cite{kl-ro}.}
\begin{equation}\label{renorq}
\rhoc:=\rho-\frac{1}{2}\hch\c\hchb,\qquad \sic:=\si-\frac{1}{2}\hch\wedge\hchb.
\end{equation}
\end{df}
\begin{rk}
The purpose of defining the renormalized quantities $\rhoc$ and $\sic$ is to get rid of the terms $\hch\c\aa$ and $\hch\wedge\aa$ in the right hand side of the Bianchi equations for $\nabs_3\rho$ and $\nabs_3\si$, as they are the most dangerous among all the nonlinear terms appearing in the Bianchi equations. See Lemma \ref{renorequation} for the Bianchi equations satisfied by $(\rhoc,\sic)$. Note that the terms $\hch\c\aa$ and $\hch\wedge\aa$ are emphasized as being borderline for the case $s=2$ in \cite{Bieri}.
\end{rk}
\subsection{Main norms}\label{ssec8.1}
The norms in Sections \ref{secRnorms}--\ref{sssec8.1.3} are defined in the bootstrap region $V_{t_*}$ while the norms in Section \ref{initialO0} are defined in the initial layer region $\kk_{(0)}$. Also, throughout Sections \ref{secRnorms}--\ref{sssec8.1.3}, the norms $\mre$, $\mke$ and $\moe$ are defined in the whole region $V_{t_*}=\Ve\cup\Vi$ while the norms $\mri$, $\mrt$, $\mki$, $\mkt$, $\moi$ and $\mot$ are defined in the interior region $\Vi$ only.
\subsubsection{\texorpdfstring{$\mr$}{} norms (curvature components)}\label{secRnorms}
We define
\begin{align*}
\mre:=\mre[\a]+\mre[\b]+\mre[\rhoc,\sic]+\mre[\bb]+\mre[\aa]+\mre[\a_4]+\mre[\aa_3],
\end{align*}
where
\begin{align*}
\big(\mre[\a]\big)^2&:=\BB_s^1[\a]+\EE_s^1[\a]+\FF_s^1[\a]+\DD_s[\a],\\
\big(\mre[\b]\big)^2&:=\BB_s^1[\b]+\EE_s^1[\b]+\FF_s^1[\b]+\DD_s[\b],\\
\big(\mre[\rhoc,\sic]\big)^2&:=\BB_s^1[\rhoc,\sic]+\EE_s^1[\rhoc,\sic]+\FF_s^1[\rhoc,\sic]+\DD_s[\rhoc,\sic],\\
\big(\mre[\bb]\big)^2&:=\EE_{0,s}^1[\bb]+\FF_s^1[\bb]+\DD_{s-1,1}[\bb],\\
\big(\mre[\aa]\big)^2&:=\FF_{0,s}^1[\aa]+\DD_{-1,s+1}[\aa],\\
\big(\mre[\a_4]\big)^2&:=\BB_s[\a_4]+\EE_s[\a_4]+\FF_s[\a_4],\\
\big(\mre[\aa_3]\big)^2&:=\FF_{-2,s+2}[\aa_3],
\end{align*}
with
\begin{equation*}
    \a_4:=r^{-4}\nabs_4(r^5\a),\qquad\aa_3:=r^{-4}\nabs_3(r^5\aa).
\end{equation*}
Next, we define
\begin{align*}
    \mri:=\mri[\a]+\mri[\b]+\mri[\rhoc,\sic]+\mri[\bb]+\mri[\aa]+\mri[\a_4],
\end{align*}
where
\begin{align*}
    \big(\mri[R]\big)^2&:=\EEi^1_{0,s}[R]+\FFi_{-2,s+2}^1[R]+\DDi_{-2,s+2}[R],\qquad\forall\; R\in\{\a,\b,\rhoc,\sic,\bb,\aa\},\\
    \big(\mri[\a_4]\big)^2&:=\EEi_{0,s}[\a_4]+\FFi_{-2,s+2}[\a_4]+\DDi_{-2,s+2}[\a_4].
\end{align*}
We also define the following time derivative norms:
\begin{align*}
\mrt:=\mrt[\a]+\mrt[\b]+\mrt[\rho,\si]+\mrt[\bb]+\mrt[\aa],
\end{align*}
where
\begin{align*}
    \big(\mrt[R]\big)^2&:=\EEi_{0,s+2}[\nabs_TR]+\FFi_{0,s+2}[\nabs_TR], \qquad\forall\; R\in\{\a,\b,\rho,\si,\bb\},\\
    \big(\mrt[\aa]\big)^2&:=\FFi_{0,s+2}[\nabs_T\aa].
\end{align*}
Finally, we denote
\begin{align*}
    \mr:=\mre+\mri+\mrt.
\end{align*}
\subsubsection{\texorpdfstring{$\mk$}{} norms (maximal connection coefficients)}
We define
\begin{align*}
    \mke:=\mke[\de]+\mke[\eps]+\mke[\kah]+\mke[\phi],
\end{align*}
where
\begin{align*}
\big(\mke[\de]\big)^2&:=\FF_{s-2}^2[\de]+\FF_{s}^1[\nabs_N\de]+\DD_{s-2}^1[\de]+\DD_{s-1,1}[\nabs_N\de],\\
\big(\mke[\eps]\big)^2&:=\FF_{s-2}^2[\eps]+\FF_{s}^1[\nabs_N\eps]+\DD_{s-2}^1[\eps]+\DD_{s-1,1}[\nabs_N\eps],\\
\big(\mke[\kah]\big)^2&:=\FF_{s-2}^2[\kah]+\FF_s^1[\nabs_N\kah]+\DD^1_{s-3,1}[\kah]+\DD_{-1,s-1}[\kah]+\DD_{-1,s+1}[\nabs_N\kah],\\
\big(\mke[\phi]\big)^2&:=\FF_{s-4}^3[\log\phi]+\FF_{s-2}^2[\nab_N\phi]+\FF_s^1[\nab_N\nab_N\phi]\\
&\;+\DD_{s-4}^2[\log\phi]+\DD_{s-2}^1[\nab_N\phi]+\DD_{s-1,1}[\nab_N\nab_N\phi].
\end{align*}
Next, we define
\begin{align*}
    \mki:=\mki[k]+\mki[\phi],
\end{align*}
where
\begin{align*}
    \big(\mki[k]\big)^2:=&\FFi_{-2,s}[k]+\FFi_{0,s}[\nab k]+\FFi_{0,s+2}[\nab^2 k]+{^i\Dinf_{-3,s+1}}[k]+\DDi_{-2,s+2}[\nab k],\\
    \big(\mki[\phi]\big)^2:=&\FFi_{-3,s-1}[\log\phi]+\FFi_{-2,s}[\nab\phi]+\FFi_{0,s}[\nab^2\phi]+\FFi_{0,s+2}[\nab^3\phi]\\
    +&{^i\Dinf_{-3,s-1}}[\log\phi]+{^i\Dinf_{-3,s+1}}[\nab\phi]+\DDi_{-2,s+2}[\nab^2\phi].
\end{align*}
We also define the following time derivative norms:
\begin{align*}
    \mkt:=\mkt[\de]+\mkt[\eps]+\mkt[\kah]+\mkt[\phi],
\end{align*}
where
\begin{align*}
    \big(\mkt[\de]\big)^2&:=\sup_{p\in[-2,s-1]}\DDi_{p,s-p}[\nabs_T\de],\\
    \big(\mkt[\eps]\big)^2&:=\sup_{p\in[-2,s-1]}\DDi_{p,s-p}[\nabs_T\eps],\\
    \big(\mkt[\kah]\big)^2&:=\sup_{p\in[-2,-1]}\DDi_{p,s-p}[\nabs_T\kah],\\
    \big(\mkt[\phi]\big)^2&:=\sup_{p\in[-2,s-1]}\DDi_{p,s-p}[\nabs_T\nab\phi].
\end{align*}
Finally, we denote
\begin{equation*}
    \mk:=\mke+\mki+\mkt.
\end{equation*}
\subsubsection{\texorpdfstring{$\mo$}{} norms (null connection coefficients)}\label{sssec8.1.3}
In the rest of this paper, we always denote by $\db$ a fixed constant satisfying
\begin{align}\label{dfdb}
    0<\db\ll s-1.
\end{align}
We define
\begin{align*}
    \moe:=\moe[\slu]+\moe[\mu]+\moe[\trch]+\moe[\hch]+\moe[\eta],
\end{align*}
where
\begin{align*}
    \big(\moe[\slu]\big)^2:=&\DD_s[\slu]+\sup_{S\subseteq V}|r^\frac{s+3}{2}\nabs\slu|_{2,S}+\EE_{s-\db,\db}^1[\slu]+\FF_{s-\db,\db}^1[\slu],\\
    \big(\moe[\mu]\big)^2:=&\DD_{1,s-1}[\mu]+\FF_{s-\db,\db}^1[\mu],\\
    \big(\moe[\trch]\big)^2:=&\DD_{s-2}^1[\trchc]+\FF_{s-2-\db,\db}^2[\trchc]+\DDinf_{-1,s-1}\left[\trch-\frac{2}{r}\right],\\
    \big(\moe[\hch]\big)^2:=&\DD_{s-2}^1[\hch]+\FF_{s-2-\db,\db}^2[\hch],\\
    \big(\moe[\eta]\big)^2:=&\DD_{-1,s-1}^1[\eta]+\FF_{s-2-\db,\db}^2[\eta].
\end{align*}
Next, we define
\begin{align*}
    \moi:=\moi[\slu]+\moi[\mu]+\moi[\trch]+\moi[\hch]+\moi[\eta],
\end{align*}
where
\begin{align*}
    \big(\moi[\slu]\big)^2:=&\DDi_{-2,s+2}[\slu]+\sup_{S\subseteq\Vi}|ru^\frac{s+1}{2}\nabs\slu|_{2,S}+\EEi_{0,s}^1[\slu]+\FFi_{0,s}^1[\slu],\\
    \big(\moi[\mu]\big)^2:=&\DDi_{-2,s+2}[\mu]+\FFi_{0,s}^1[\mu],\\
    \big(\moi[\trch]\big)^2:=&\DDi_{-3,s+1}^1[\trchc]+\FFi_{-2,s}^2[\trchc]+{^i\DDinf_{-3,s+1}}\left[\trch-\frac{2}{r}\right],\\
    \big(\moi[\hch]\big)^2:=&\DDi_{-3,s+1}^1[\hch]+\FFi_{-2,s}^2[\hch],\\
    \big(\moi[\eta]\big)^2:=&\DDi^1_{-3,s+1}[\eta]+\FFi_{-2,s}^2[\eta].
\end{align*}
We also define
\begin{align*}
    \big(\mot\big)^2:=\big(\mot[\eta]\big)^2:=\sup_{p\in[-2,-1]}\DDi_{p,s-p}[\nabs_T\eta].
\end{align*}
Finally, we denote
\begin{equation*}
    \mo:=\moe+\moi+\mot.
\end{equation*}
\begin{rk}
By definition, the norms $\mo$ and $\mk$ control $\de$, $\eps$, $\ka$, $\nab\log\phi$, $\eta$ and $\chi$. Moreover, we have from \eqref{6.6}
\begin{align*}
    \chib_{AB}&=-\chi_{AB}-2k_{AB},\qquad\qquad\; \xib_A=\nabs_A\log\phi+\eps_A-\eta_A,\\
    2\om&=-\nab_N\log\phi+\de,\qquad\qquad\, 2\omb=\nab_N\log\phi+\de,\\
    \nabs_A\log a&=\eta_A-\eps_A,\qquad\qquad\qquad\;\;\;\,\ze_A=\eps_A,\\
    \xi&=0,\qquad\qquad\qquad\qquad\quad\;\;\, \etab_A=\nabs_A\log\phi-\eps_A.
\end{align*}
Thus, the norms $\mo$ and $\mk$ control all the connection coefficients.
\end{rk}
\subsubsection{\texorpdfstring{$\mo_{(0)}$}{}, \texorpdfstring{$\mk_{(0)}$}{} and \texorpdfstring{$\Rfk_{(0)}$}{} norms (initial data)}\label{initialO0}
We define the curvature flux in the initial layer region $\kk_{(0)}$ introduced in \eqref{dfinitiallayer}:
\begin{align*}
(\Rfk_{(0)})^2:=\sup_{t\in[0,2]}\int_{\Si_t}r^s\left|\dk^{\leq 1}\left(\a,\b,\rho,\si,\bb,\aa\right)\right|^2,\qquad\dk:=\{r\nabs_3,r\nabs_4,r\nabs\}.
\end{align*}
Next, we define the norms of null connection coefficients in the initial layer $\kk_{(0)}$:
\begin{align*}
\mo_{(0)}:=\sum_{\Ga\in\{\eta,\trchc,\hch\}}\mo_{(0)}[\Ga]+\sup_{S\subseteq\kk_{(0)}}\left|r^\frac{s+1}{2}\left(\trch-\frac{2}{r}\right)\right|_{\infty,S}+\sup_{\Si_0}|r+u|,
\end{align*}
where
\begin{align*}
(\mo_{(0)}[\Ga])^2:=\sup_{t\in[0,2]}\int_{\Si_t} r^{s-2}|\dk^{\leq 2}\Ga|^2+\sup_{S\subseteq\kk_{(0)}}|r^\frac{s+1}{2}\dk^{\leq 2}\Ga|_{\infty,S}^2.
\end{align*}
We also define the norms of maximal connection coefficients in $\kk_{(0)}$:
\begin{align*}
    \mk_{(0)}:=\sum_{k\in\{\eps,\de,\kah,\nab\phi\}}\mk_{(0)}[k]+\sup_{S\subseteq\kk_{(0)}}|r^\frac{s-1}{2}(\phi-1)|_{\infty,S},
\end{align*}
where
\begin{align*}
    \big(\mk_{(0)}[k]\big)^2:=\sup_{t\in[0,2]}\int_{\Si_t} r^{s-2}|\dk^{\leq 2}k|^2.
\end{align*}
\subsection{Main theorem}\label{ssec8.3}
We are ready to state a precise version of our main theorem.
\begin{thm}[Main Theorem, version 2]\label{th8.1}
Consider an initial data set $(\Si_0,g,k)$, $(s,2)$--asymptotically flat in the sense of Definition \ref{def6.3} with $1<s\leq 2$. Assume that we have the following control in the initial layer region $\kk_{(0)}$:
\begin{equation}\label{(0)ass}
    \mo_{(0)}\leq \ep_0,\qquad\mk_{(0)}\leq\ep_0,\qquad \Rfk_{(0)}\leq\ep_0,
\end{equation}
where $\mo_{(0)}$, $\mk_{(0)}$ and $\mathfrak{R}_{(0)}$ are defined in Section \ref{initialO0} and $\ep_0>0$ is a small enough constant.\\ \\
Then, the initial layer $\kk_{(0)}$ has a unique development $(\M,\bg)$ in its future domain of dependence with the following properties:
\begin{enumerate}
    \item $(\M,\bg)$ can be foliated by a maximal-null foliation $(t,u)$. Moreover, the outgoing cones $C_u$ are complete for all $u\in\mathbb{R}$.
    \item The norms $\mo$, $\mk$ and $\mr$ defined in Sections \ref{secRnorms}--\ref{sssec8.1.3} satisfy
\begin{equation}\label{finalesti}
    \mo\les\ep_0,\qquad \mk\les\ep_0,\qquad \mr\les\ep_0.
\end{equation}
\end{enumerate}
\end{thm}
\begin{rk}\label{restrictions}
The lower bound $s>1$ is almost sharp in the following sense. Denoting $\Ga$ the connection coefficients, $R$ the curvature components and $\mu=3,4$, we have from Propositions \ref{nulles} and \ref{Bianchieq}, schematically,
\begin{align}
    \begin{split}
    \nabs_\mu\Ga&=\nabs\Ga+R+\Ga\c\Ga,\label{Gashc}\\
    \nabs_\mu R&=\nabs R+\Ga\c R.
    \end{split}
\end{align}
Also, we have from \eqref{gks1} that the order of decay of $\Ga$ is $-\frac{s+1}{2}$ while that of $R$ is $-\frac{s+3}{2}$. Thus, since the nonlinear terms should have better order of decay than the linear terms, we should have in view of \eqref{Gashc}
\begin{align*}
    \frac{s+1}{2}+\frac{s+1}{2}>\frac{s+3}{2},\qquad \frac{s+1}{2}+\frac{s+3}{2}>\frac{s+5}{2},
\end{align*}
where both restrictions are equivalent to $s>1$.
\end{rk}
The proof of Theorem \ref{th8.1} is given in Section \ref{ssec8.6}. It hinges on three theorems stated in Section \ref{ssec8.4}, concerning estimates for $\mo$, $\mk$ and $\mr$ norms. \\ \\
We choose $\ep_0$ small enough such that
\begin{equation*}
    \ep:=\ep_0^{\frac{2}{3}},\qquad\quad 0<\ep\ll 1,
\end{equation*}
where $\ep$ will be the smallness constant involved in bootstrap assumptions. Here, $A\ll B$ means that $CA<B$ where $C$ is the largest universal constant among all the constants involved in the proof via $\lesssim$.
\subsection{Main intermediate results}\label{ssec8.4}
The following three theorems are the main intermediate results in the proof of Theorem \ref{th8.1}.
\begin{thm}\label{M1}
Assume that
\begin{equation}
    \Rfk_{(0)}\leq\ep_0,\qquad \mo\leq\ep,\qquad \mk\leq\ep,\qquad \RR\leq\ep.
\end{equation}
Then, we have
\begin{equation}
    \mr\les \ep_0.
\end{equation}
\end{thm}
Theorem \ref{M1} is proved in Section \ref{sec9}. The proof is based on the $r^p$--weighted estimate method introduced by Dafermos and Rodnianski in \cite{Da-Ro} and applied to the Bianchi equations.
\begin{thm}\label{M2}
Assume that
\begin{align}
    \mo\leq\ep,\qquad\mk\leq\ep,\qquad\RR\les\ep_0.
\end{align}
Then, we have
\begin{equation}
    \mk\les \ep_0.
\end{equation}
\end{thm}
Theorem \ref{M2} is proved in Section \ref{seck}. The proof is done by applying elliptic estimates on maximal hypersurfaces $\Si_t$.
\begin{thm}\label{M3}
Assume that
\begin{align}
    \mo_{(0)}\leq\ep_0,\qquad\mo\leq\ep,\qquad \mk\les\ep_0,\qquad \RR\les\ep_0.
\end{align}
Then, we have
\begin{equation}
    \mo\les\ep_0.
\end{equation}
\end{thm}
Theorem \ref{M3} is proved in Section \ref{sec10}. The proof is done by integrating transport equations along the outgoing null cones and applying elliptic estimates on $2$--spheres of the maximal-null foliation of the spacetime $\M$.
\subsection{Proof of the main theorem}\label{ssec8.6}
We now use Theorems \ref{M1}--\ref{M3} to prove Theorem \ref{th8.1}.
\begin{df}\label{defboot}
We denote $\mathcal{U}$ the set of values $t_*$ such that the spacetime
\begin{align*}
    \kk_{t_*}:=V_{t_*}\cup \kk_{(0)}=J^-(\Si_{t_*})\cap J^+(\Si_0)
\end{align*}
associated with the maxmimal-null foliation $(t,u)$ defined in Section \ref{ssec7.1} satisfies the following bounds: 
\begin{align}
    \mo\leq\ep,\qquad \mk\leq \ep,\qquad\mr\leq\ep.\label{B2}
\end{align}
\end{df}
The assumption \eqref{(0)ass} implies that \eqref{B2} holds if $t_*\in[0,2]$. So, we have $\mathcal{U}\ne\emptyset$. Next, we define $t_*$ to be the supremum of the set $\mathcal{U}$. We want to prove $t_*=+\infty$. We assume by contradiction that $t_*<\infty$. In particular we may assume $t_*\in\mathcal{U}$. We have from \eqref{(0)ass}
\begin{equation}
    \mo_{(0)}\leq\ep_0,\qquad \mk_{(0)}\leq\ep_0,\qquad \mathfrak{R}_{(0)}\leq\ep_0.
\end{equation}
Applying Theorems \ref{M1}, \ref{M2} and \ref{M3} one by one in that order, we obtain
\begin{equation}\label{mrest}
    \mr\les\ep_0,\qquad \mk\les\ep_0,\qquad \mo\les\ep_0.
\end{equation}
By local existence, we can extend $V_{t_*}$ to $\wideparen{V}:={V}_{t_*+\nu}$ for a $\nu>0$ sufficiently small. We denote $\wideparen{\mo}$, $\wideparen{\mk}$ and $\wideparen{\mr}$ the norms in the extended region $\wideparen{V}$. We have
\begin{equation*}
    \wideparen{\mo}\les\ep_0,\qquad\wideparen{\mk}\les\ep_0,\qquad \wideparen{\mr}\les\ep_0,
\end{equation*}
as a consequence of \eqref{mrest} and local existence in $t_*\leq t\leq t_*+\nu$. We deduce that $\wideparen{V}$ satisfies \eqref{B2}, which is a contradiction. Thus, we have $t_*=+\infty$, which implies property 1 of Theorem \ref{th8.1}. Moreover, we have
\begin{equation}
    \mo\les \ep_0,\qquad\mk\les\ep_0,\qquad  \mr\les\ep_0,\quad\mbox{ in } \; \MM=\kk_{\infty},
\end{equation}
which implies property 2 of Theorem \ref{th8.1}. This concludes the proof of Theorem \ref{th8.1}.
\section{Bootstrap assumptions and first consequences}\label{firstboot}
In the rest of the paper, we always assume the following bootstrap assumptions:
\begin{align}
    \mo\leq\ep,\qquad \mk\leq \ep,\qquad\mr\leq\ep.\label{B1}
\end{align}
In this section, we derive first consequences of \eqref{B1} in the region $V_{t_*}$. In the sequel, the results of this section will be used frequently without explicitly mentioning them.
\subsection{Schematic notation \texorpdfstring{$\Gag$}{}, \texorpdfstring{$\Gab$}{} and \texorpdfstring{$\Gaw$}{}}
We introduce the following schematic notations.
\begin{df}\label{gammag}
We divide the Ricci coefficients into three parts:
\begin{align*}
    \Gag&:=\left\{\widecheck{\trch},\,\widecheck{\trchb},\,\hch,\,\om,\,\omb,\,\eps,\,\de,\,\nab\log\phi,\,\ze,\,\etab\,\right\},\\
    \Gab&:=\left\{\eta,\,\nabs\log a,\,\xib,\,\trch-\frac{2}{r}\right\},\\
    \Gaw&:=\{\hchb,\,\kah,\,\thh\}.
\end{align*}
We also denote:
\begin{align*}
    \Gag^{(1)}&:=(r\nabs)^{\leq 1}\Gag\cup\{r\a,r\b,r\rho,r\si,r\slu\},\\
    \Gab^{(1)}&:=(r\nabs)^{\leq 1}\Gab\cup\{r\mu\},\\
    \Gaw^{(1)}&:=(r\nabs)^{\leq 1}\Gaw\cup\{r\bb\},
\end{align*}
and
\begin{align*}
    \Ga_i^{(2)}:=(r\nabs)^{\leq 1}\Ga_i^{(1)},\quad\mbox{ where }\; i=g,b,w.
\end{align*}
\end{df}
\begin{lem}\label{decayGagGabGaa}
We have the following estimates:
\begin{align*}
    \DD_{s-2}^1[\Gag]\les\ep^2,\qquad\DD_{-1,s-1}^1[\Gab]\les\ep^2,\qquad\DD_{-1,s-1}[\Gaw]\les\ep^2,\qquad \DD_{s-3,1}^1[\Gaw]\les\ep^2,
\end{align*}
and
\begin{align*}
    ^i\Dinf_{-3,s+1}[\Gag,\Gab,\Gaw]\les\ep^2,\qquad \DDi^1_{-2,s}[\Gag,\Gab,\Gaw]\les\ep^2.
\end{align*}
Moreover, we have
\begin{equation}\label{estGa2}
    \FF_{s-2-\db,\db}^2[\Gag,\Gab,\Gaw]\les\ep^2,
\end{equation}
where $\db$ is introduced in \eqref{dfdb}.
\end{lem}
\begin{proof}
    It follows directly from \eqref{6.6}, \eqref{B1} and Definitions \ref{dfdecaynorms} and \ref{gammag}.
\end{proof}
\begin{lem}\label{Rdecay}
    For two quantities $X_{(1)}$ and $X_{(2)}$, we denote
    \begin{align*}
        X_{(1)}\preceq X_{(2)}
    \end{align*}
    if $X_{(1)}^{(q)}$ decays better than $X_{(2)}^{(q)}$ for any $q=0,1,2$. Then, throughout this paper, we have
    \begin{align*}
        \a\preceq\b\preceq(\rho,\si)\preceq\bb\preceq\aa,\qquad \quad\Gag\preceq\Gab\preceq\Gaw.
    \end{align*}
\end{lem}
\begin{proof}
    It is a direct consequence of \eqref{B1}.
\end{proof}
\begin{rk}\label{Gawrk}
In the sequel, we choose the following conventions:
\begin{itemize}
    \item For a quantity $h$ satisfying the same or even better decay and regularity as $\Ga_{i}$, for $i=g,b,w$, we write
    \begin{equation*}
        h\in\Ga_i,\quad i=g,b,w.
    \end{equation*}
    \item For a sum of schematic notations, we ignore the terms which have same or even better decay and regularity. For example, we write
    \begin{equation*}
        \Gag+\Gab=\Gab,\qquad \Gab+\Gaw=\Gaw,\qquad \left(\frac{\ujp}{\ujp+r}\right)^\frac{s-1}{2}\Gab=\Gag.
    \end{equation*}
    \item Since $\Gaw$ has better decay and regularity than $\Gab^{(1)}$, we write
    \begin{align*}
        \Gaw=\Gab^{(1)}.
    \end{align*}
    \item For a quantity $h$ satisfying the same or even better decay and regularity than the curvature component $R^{(q)}$, we write
    \begin{equation*}
        h=R^{(q)}.
    \end{equation*}
    For example, we have
    \begin{equation*}
        r^{-1}\Gaw^{(1)}=\aa^{(0)},\qquad \Gab\c\rho=\Gab\c\bb,\qquad \Gag\c\b+\Gag\c\rho=\Gag\c\rho.
    \end{equation*}
\end{itemize} 
\end{rk}
\subsection{Commutation identities in schematic form}
\begin{prop}\label{commutation}
We have the following simple schematic consequences of the commutator identities
\begin{align*}
[r\nabs,\nabs_4]&=\Gag\c r\nabs+\Gag\c r\nabs_4+\Gag^{(1)},\\
[r\nabs,\nabs_3]&=\Gaw\c r\nabs+\Gab\c r\nabs_4+\Gab\c r\nabs_3+\Gaw^{(1)},\\
[\nabs_3,\nabs_4]&=\Gab\c\nabs+\Gag\c\nabs_3+\Gag\c\nabs_4+r^{-1}\Gag^{(1)}.
\end{align*}
\end{prop}
\begin{proof}
    It follows directly from Lemma \ref{comm} and \eqref{6.6}.
\end{proof}
\begin{cor}\label{Tcomm}
We have the following commutator identities:
\begin{align*}
[\nabs_T,\nabs_4]&=\Gab\c\nabs+\Gag\c\nabs_T+\Gag\c\nabs_4+r^{-1}\Gag^{(1)},\\
[\nabs_T,\nabs_3]&=\Gab\c\nabs+\Gag\c\nabs_T+\Gag\c\nabs_4+r^{-1}\Gag^{(1)},\\
[\nabs_T,\nabs]&=\Gaw\c\nabs+\Gab\c\nabs_T+\Gab\c\nabs_4+r^{-1}\Gaw^{(1)}.
\end{align*}
\end{cor}
\begin{proof}
It follows directly from Proposition \ref{commutation} and the fact that $2T=e_4+e_3$.
\end{proof}
\subsection{Main equations in schematic form}
\begin{prop}\label{nullschematic}
We have the null structure equations:
\begin{align*}
\nabs_4\eta&=r^{-1}\Gab^{(1)},\\
\nabs_4\hch&=r^{-1}\Gag^{(1)},\qquad\qquad\qquad\qquad\qquad\;\,\,\nabs_3\hch=r^{-1}\Gab^{(1)},\\
\nabs_4\trch+\frac{1}{2}(\tr\chi)^2&=r^{-1}\Gag,\qquad\qquad\;\,\,\nabs_3\trch+\frac{1}{2}\trchb\,\trch=r^{-1}\Gab^{(1)},\\
\nabs_3\hchb&=-\aa+r^{-1}\Gab^{(1)},\qquad\qquad\qquad\qquad\nabs_4\hchb=r^{-1}\Gaw^{(1)},\\
\nabs_3\trchb+\frac{1}{2}(\trchb)^2&=r^{-1}\Gaw^{(1)},\qquad\qquad\nabs_4\trchb+\frac{1}{2}\trch\,\trchb=r^{-1}\Gag^{(1)},\\
\nabs_3\ze&=r^{-1}\Gaw^{(1)},\qquad\qquad\qquad\qquad\qquad\;\;\,\nabs_4\ze=r^{-1}\Gag^{(1)},\\
\sdivs\hch&=\frac{1}{2}\slu-\b+\Gag\c\Gag,\qquad\qquad\qquad\,\sdivs\hchb=r^{-1}\Gaw^{(1)}.
\end{align*}
\end{prop}
\begin{proof}
    It follows directly from Proposition \ref{nulles}.
\end{proof}
\begin{cor}\label{LieTGa}
We have the following identities:
\begin{align}
    \begin{split}
        \nabs_T\hch&=r^{-1}\Gab^{(1)},\qquad\;\,\, \nabs_T\ze=r^{-1}\Gaw^{(1)},\qquad \;\,\,
        \nabs_T\hchb=\aa^{(0)},\\
\nabs_T\trch&=r^{-1}\Gab^{(1)},\quad\;\,\nabs_T\trchb=r^{-1}\Gaw^{(1)}.
    \end{split}
\end{align}
\end{cor}
\begin{proof}
    It follows directly from Proposition \ref{nullschematic} and the fact that $2T=e_4+e_3$.
\end{proof}
\begin{lem}\label{renorequation}
We have the following Bianchi equations for $((\rhoc,\sic),\bb)$:
\begin{align*}
\nabs_4\bb+\trch\,\bb&=\sld_1^*(\rhoc,\sic)+\Gaw\c(\slu,\b)+r^{-1}\Gag\c\Gaw^{(1)},\\
\nabs_3(\rhoc,\sic)+\frac{3}{2}\trchb(\rhoc,\sic)&=-\sld_1\bb+r^{-1}(\Gab\c\Gaw)^{(1)}.
\end{align*}
\end{lem}
\begin{proof}
We have from Proposition \ref{nullschematic}
\begin{align*}
\nabs_3\hchb=-\aa+r^{-1}\Gab^{(1)},\qquad \nabs_3\hch=r^{-1}\Gab^{(1)}.
\end{align*}
Hence, we have from Proposition \ref{Bianchieq} and Remark \ref{Gawrk}
\begin{align*}
    \nabs_3\rhoc+\frac{3}{2}\trchb\,\rhoc&=\nabs_3\rho+\frac{3}{2}\trchb\,\rho-\frac{1}{2}\nabs_3(\hch\c\hchb)+r^{-1}\Gag\c\Gaw\\
    &=-\sdivs\bb-\frac{1}{2}\hch\c\aa+\Gab\cdot(\b,\bb)-\frac{1}{2}\hch\c\nabs_3\hchb-\frac{1}{2}\nabs_3\hch\c\hchb+r^{-1}\Gag\c\Gaw\\
    &=-\sdivs\bb-\frac{1}{2}\hch\c\aa+r^{-1}\Gab\c\Gaw^{(1)}-\frac{1}{2}\hch\c(-\aa+r^{-1}\Gab^{(1)})+r^{-1}\Gab^{(1)}\c\Gaw\\
    &=-\sdivs\bb+r^{-1}(\Gab\c\Gaw)^{(1)}
\end{align*}
as stated. The derivation of the equation of $\sic$ is similar and left to the reader.\\ \\
Next, we have from \eqref{codazzi}, \eqref{dfmu}, \eqref{renorq} and Proposition \ref{Bianchieq}
\begin{align*}
\nabs_4\bb+\trch\,\bb&=d_1^*(\rho,\si)+\Gaw\c\b+\Gag\c(\rho,\si,\bb)\\
&=d_1^*(\rhoc,\sic)+\frac{1}{2}d_1^*(\hch\c\hchb,\hch\wedge\hchb)+\Gaw\c\b+\Gag\c(\rho,\si,\bb)\\
&=d_1^*(\rhoc,\sic)+\Gaw\c(\slu,\b)+r^{-1}\Gag\c\Gaw^{(1)}
\end{align*}
as stated. This concludes the proof of Proposition \ref{renorequation}.
\end{proof}
\begin{prop}\label{bianchischematic}
The Bianchi equations take the following form:
\begin{align*}
\nabs_3\a+\frac{1}{2}\trchb\,\a&=-2\sld_2^*\b+\Gag\c\rho+\Gab\c\b,\\
\nabs_4\b+2\trch\,\b&=\sld_2\a+\Gag\c\rho,\\
\nabs_3\b+\trchb\,\b&=-\sld_1^*(\rho,-\si)+\Gag\c\bb+\Gab\c\rho,\\
\nabs_4(\rho,-\si)+\frac{3}{2}\trch(\rho,-\si)&=\sld_1\b+\Gaw\c\b,\\
\nabs_3(\rhoc,\sic)+\frac{3}{2}\trchb(\rhoc,\sic)&=-\sld_1\bb+r^{-1}(\Gab\c\Gaw)^{(1)},\\
\nabs_4\bb+\trch\,\bb&=\sld_1^*(\rhoc,\sic)+\Gaw\c(\slu,\b)+r^{-1}\Gag\c\Gaw^{(1)},\\
\nabs_3\bb+2\trchb\,\bb&=-\sld_2\aa+\Gab\c\aa,\\
\nabs_4\aa+\frac{1}{2}\trch\,\aa&=2\sld_2^*\bb+\Gag\c\aa+\Gaw\c\rho.
\end{align*}
\end{prop}
\begin{proof}
    It follows directly from Proposition \ref{Bianchieq} and Lemma \ref{renorequation}.
\end{proof}
\begin{cor}\label{LieTR}
We have the following identities:
\begin{align*}
\nabs_T\a&=r^{-1}\b^{(1)},\qquad\qquad\qquad\;\;\nabs_T\b=r^{-1}\rho^{(1)}+\Gag\c\bb,\\
\nabs_T(\rho,\si)&=r^{-1}\bb^{(1)}+\Gag\c\aa,\qquad\quad\nabs_T\bb=r^{-1}\aa^{(1)}.
\end{align*}
\end{cor}
\begin{proof}
    It follows directly from Proposition \ref{bianchischematic} and the fact that $2T=e_4+e_3$.
\end{proof}
\subsection{Elliptic estimates in \texorpdfstring{$2D$}{}--geometry}\label{2D}
Throughout Section \ref{2D}, we assume that $t\geq 0$ and $u\in\mathbb{R}$ are fixed constants and we denote $S:=S(t,u)$ the leaf of the maximal-null foliation constructed in Section \ref{nullfoliation}.
\begin{prop}[$L^p$ estimates for Hodge systems]\label{Lpestimates}
The following statements hold for all $p\in(1,+\infty)$:
\begin{enumerate}
\item Let $\phi\in\sfr_0$ be a solution to $\Des\phi=f$. Then we have
\begin{align*}
    |\nabs^2\phi|_{p,S}+r^{-1}|\nabs\phi|_{p,S}+r^{-2} |\phic|_{p,S}\les |f|_{p,S}.
\end{align*}
\item Let $\xi\in\sfr_1$ be a solution to $\sld_1\xi=(f,f_*)$. Then we have
\begin{align*}
    |\nabs\xi|_{p,S}+r^{-1}|\xi|_{p,S}\les |(f,f_*)|_{p,S}.
\end{align*}
\item Let $U\in\sfr_2$ be a solution to $\sld_2 U=f$. Then we have
\begin{align*}
    |\nabs U|_{p,S}+r^{-1}|U|_{p,S}\les |f|_{p,S}.
\end{align*}
\end{enumerate}
\end{prop}
\begin{proof}
See Corollary 2.3.1.1 in \cite{Ch-Kl}.
\end{proof}
\begin{prop}\label{standardsobolev}
Let $F$ be a tensor field on $S$. Then, we have
\begin{equation*}
    |r^{\frac{1}{2}}F|_{\infty,S}\les |F|_{4,S}+|r\nabs F|_{4,S}.
\end{equation*}
\end{prop}
\begin{proof}
See Lemma 4.1.3 in \cite{Kl-Ni}.
\end{proof}
\subsection{Elliptic estimates in \texorpdfstring{$3D$}{}--geometry}\label{3D}
Throughout Section \ref{3D}, we assume that $t\geq 0$ is a fixed constant and we denote $\Si:=\Si_t$ the maximal hypersurface constructed in Section \ref{maximalfoliation}.
\begin{df}\label{L2flux}
For a tensor field $h$ on $\Si$, we define its $L^p$--norm as follows:
\begin{equation}
    \|h\|_{p,\Si}:= \left(\int_{\Si}|h|^p \right)^\frac{1}{p}.
\end{equation}
\end{df}
\begin{prop}[Hardy]\label{Hardy}
We have for any tensor field $F$ on $\Si$
\begin{align*}
    \|r^{-1}F\|_{2,\Si}\leq c_H\|\nab F\|_{2,\Si},
\end{align*}
where
\begin{equation*}
    c_H=2+O(\ep).
\end{equation*}
\end{prop}
\begin{proof}
We have
\begin{align}\label{divNfor}
\sdiv N=g^{NN}g(D_N N,N)+g^{AB}g(D_A N, e_B)=\slg^{AB}\th_{AB}=\tr\th=\frac{2}{r}+\Gab,
\end{align}
which implies from $r\Gab=O(\ep)$
\begin{align*}
    \sdiv\left(\frac{N}{r}\right)=\frac{\sdiv N}{r}-\frac{N(r)}{r^2}=\frac{1}{r^2}+r^{-1}\Gab=\frac{1+O(\ep)}{r^2}.
\end{align*}
Thus, we have
\begin{align*}
    (1+O(\ep))\int_{\Si}\frac{|F|^2}{r^2}&=\int_{\Si}\sdiv\left(\frac{N}{r}\right)|F|^2=-\int_{\Si}\frac{2}{r}F\c\nab_N F\leq 2\left(\int_{\Si} \frac{|F|^2}{r^2}\right)^\frac{1}{2}\left(\int_{\Si}|\nab_N F|^2\right)^\frac{1}{2},
\end{align*}
which implies
    \begin{align*}
        \|r^{-1}F\|_{2,\Si}\leq (2+O(\ep))\|\nab F\|_{2,\Si}.
    \end{align*}
This concludes the proof of Proposition \ref{Hardy}.
\end{proof}
\begin{prop}[Sobolev]\label{fluxsobolev}
Let $F$ be a tensor field, tangent to $S$ at every point. We have the following inequalities:
\begin{align*}
    \|F\|_{6,\Si}&\les\|\nab F\|_{2,\Si},\\
\sup_{S\subseteq\Si}|F|_{4,S}&\les \|\nab F\|_{2,\Si},\\
\sup_{S\subseteq\Si}|r^\frac{1}{2}\ujp^\frac{1}{2}F|_{4,S}&\les\|F\|_{2,\Si}+\|r\nabs F\|_{2,\Si}+\|\ujp\nabs_NF\|_{2,\Si}.
\end{align*}
\end{prop}
\begin{proof}
See Corollaries 3.2.1.1 and 3.2.1.2 in \cite{Ch-Kl}.
\end{proof}
\begin{prop}[Gagliardo-Nirenberg]\label{GN}
Given $F$ a tensor field on $\Si$, we have the following inequality:
\begin{align*}
\|F\|_{\infty,\Si}\les\|F\|_{6,\Si}^\frac{1}{2}\|\nab F\|_{6,\Si}^\frac{1}{2}.
\end{align*}
\end{prop}
\begin{proof}
    See Page 308 in \cite{Ch-Kl}.
\end{proof}
\begin{prop}\label{hodgerank1}
    Let $\xi$ be a $1$--form on $\Si$ satisfying
    \begin{align*}
        \sdiv\xi&=D(\xi),\\
        \curl\xi&=A(\xi).
    \end{align*}
    Then, the following integral identity holds:
    \begin{align*}
    \int_{\Si}|\nab\xi|^2=\int_\Si|A(\xi)|^2+|D(\xi)|^2-\sRic^{ij}\xi_i\xi_j.
    \end{align*}
\end{prop}
\begin{proof}
See Lemma 4.4.1 in \cite{Ch-Kl}.
\end{proof}
\begin{prop}\label{hodgerank2}
Let $U$ be a $2$--traceless symmetric tensor on $\Si$ satisfying
\begin{align*}
    \sdiv U&=D(U),\\
    \curl U&=A(U).
\end{align*}
Then, the following estimates hold:
\begin{align*}
    \int_\Si |\nab U|^2&\les\int_\Si|A(U)|^2+|D(U)|^2+|\sRic||U|^2,\\
    \int_\Si |\nab^2 U|^2&\les\int_\Si|\nab A(U)|^2+|\nab D(U)|^2+|\sRic||\nab U|^2+|\sRic|^2|U|^2.
\end{align*}
\end{prop}
\begin{proof}
    See Propositions 4.4.1 and 4.4.2 in \cite{Ch-Kl}.
\end{proof}
\begin{prop}\label{hodgerank0}
    Let $\phi$ be a scalar function on $\Si$ satisfying
    \begin{align*}
        \De\phi=f.
    \end{align*}
    Then, the following estimates hold:
    \begin{align*}
    \int_{\Si}|\nab\phi|^2+r^2|\nab^2\phi|^2&\les\int_\Si r^2|f|^2,\\
    \int_\Si r^4|\nabs^3\phi|^2+r^4|\nabs^2\nab_N\phi|^2+r^4|\nabs\nab_N^2\phi|^2&\les\int_\Si r^2|f|^2+r^4|\nabs f|^2.
    \end{align*}
\end{prop}
\begin{proof}
    See Propositions 4.2.2 and 4.2.3 in \cite{Ch-Kl}.
\end{proof}
\subsection{Comparison of \texorpdfstring{$t$}{}, \texorpdfstring{$u$}{} and \texorpdfstring{$r$}{}}
\begin{lem}\label{axecomparison}
We have the following estimates on the symmetry axis:
\begin{align}
    |t(p)-u(p)|\les\ep t(p),\qquad \forall p\in\Vphi. 
\end{align}
In particular, we have
\begin{align*}
    |u_c(t)-t|\les\ep t,\qquad\quad |t_c(u)-u|\les\ep t.
\end{align*}
\end{lem}
\begin{proof}
Recalling that $\pr_t=\phi T$ and 
\begin{align*}
    T(u)=1,\quad \mbox{ on }\Vphi,
\end{align*}
we infer from \eqref{B1}
\begin{align*}
    |\pr_t(t-u)|=|1-\phi T(u)|=|1-\phi|\les\ep,\quad\mbox{ on }\Vphi.
\end{align*}
Integrating it along $\Vphi$, we deduce
\begin{align*}
    |t(p)-u(p)|\leq |t(O)-u(O)|+\int_0^t|\pr_t(t-u)|dt\les 0+\int_0^t\ep dt\les\ep t,\qquad \forall p\in\Vphi.
\end{align*}
This concludes the proof of Lemma \ref{axecomparison}.
\end{proof}
\begin{lem}\label{tur}
    We have the following estimates:
    \begin{align*}
        |r-(t-u)|&\les \ep t.
    \end{align*}
\end{lem}
\begin{proof}
    We have from Lemma \ref{dint}
    \begin{align*}
        e_4(r)=\frac{\ov{\phi\trch}}{2\phi}r.
    \end{align*}
    Recalling that
    \begin{align*}
        \phi e_4(t)=\phi T(t)=1,\qquad \phi e_4(u)=0,
    \end{align*}
    we infer from \eqref{B1}
    \begin{align*}
        \left|\phi e_4(r-(t-u))\right|=\left|\frac{\ov{\phi\trch}}{2}r-1\right|\les\ep.
    \end{align*}
    Hence, we obtain in $\Ve$
    \begin{align*}
        |r(t,u)-(t-u)|\les |r(0,u)+u|+\int_0^t \left|\phi e_4(r-(t-u))\right|dt\les \ep t,
    \end{align*}
    where we used the fact that $|r+u|\leq\mo_{(0)}\leq\ep_0$ on $\Si_0$.\\ \\
Applying Lemma \ref{axecomparison}, we have in $\Vi$
\begin{align*}
    |r(t,u)-(t-u)|\les |t_c(u)-u|+\int_{t_c(u)}^t\left|\phi e_4(r-(t-u))\right|dt\les \ep t.
\end{align*}
This concludes the proof of Lemma \ref{tur}.
\end{proof}
We define
\begin{align}\label{dfVieVii}
    \Vie_{t_*}:=\left\{p\in\Vi_{t_*}\Big/ r\geq \frac{t}{2}\right\},\qquad\quad\Vii_{t_*}:=\left\{p\in\Vi_{t_*}\Big/ r\leq \frac{t}{2}\right\}.
\end{align}
\begin{cor}
We have the following equivalences on the symmetry axis:
    \begin{align*}
        u_c(t)\simeq t,\qquad t_c(u)\simeq u,\quad \mbox{ on }\;\Vphi.
    \end{align*}
    We also have the following comparison results:
    \begin{align*}
        t+\ujp\les r\quad\mbox{ in }\;\Ve_{t_*},\qquad 1\leq u,\langle r\rangle\les t\quad\mbox{ in }\;\Vi_{t_*}.
    \end{align*}
Moreover, we have
\begin{align*}
    1\leq u\les t\simeq r,\quad\mbox{ in }\;\Vie_{t_*},\qquad  r\les t\simeq u\quad\mbox{ in }\;\Vii_{t_*}.
\end{align*}
\end{cor}
\begin{proof}
    It follows directly from Lemmas \ref{axecomparison} and \ref{tur}, and from \eqref{dfVieVii}.
\end{proof}
\subsection{Evolution lemma}
We recall the following evolution lemma, which will be used in Sections \ref{seck} and \ref{sec10}.
\begin{lem}\label{evolution}
Under the assumption \eqref{B1}, the following holds:
\begin{enumerate}
\item Let $U,F$ be $k$-covariant $S$-tangent tensor fields satisfying the outgoing evolution equation
\begin{equation*}
    \nabs_4U+\la_0\trch\,U=F.
\end{equation*}
Denoting $\la_1:=2(\lambda_0-\frac{1}{p})$, we have for all $p\in(1,\infty)$:
\begin{align*}
|r^{\la_1}U|_{p,S}(t,u)&\les |r^{\la_1}U|_{p,S}(t_c(u),u)+\int_{t_c(u)}^{t}  |r^{\la_1}F|_{p,S}(t',u)dt'\qquad\forall u\geq 1,\\
|r^{\la_1}U|_{p,S}(t,u)&\les |r^{\la_1}U|_{p,S}(2,u)+\int_{\ujp}^{r}|r^{\la_1}F|_{p,S}(t',u)dr'\qquad\qquad\forall u\leq 1.
\end{align*}
\item Let $V,\uf$ be k-covariant and S-tangent tensor fields satisfying the normal evolution equation
\begin{align*}
\nabs_NV+\la_0\tr\th\,V=\uf.
\end{align*}
Denoting $\la_1=2(\la_0-\frac{1}{p})$, we have for all $p\in(1,\infty)$:
\begin{align*}
    |r^{\la_1}V|_{p,S}(t,u) &\les \lim_{u\to-\infty}|r^{\la_1}V|_{p,S}(t,u)+\int_{-\infty}^{u} |r^{\lambda_1}\underline{F}|_{p,S}(t,u')du',\\
    |r^{\la_1}V|_{p,S}(t,u) &\les \lim_{u\to u_c(t)}|r^{\la_1}V|_{p,S}(t,u)+\int_{u}^{u_c(t)}|r^{\lambda_1}\underline{F}|_{p,S}(t,u')du'.
\end{align*}
\end{enumerate}
\end{lem}
\begin{proof}
See Lemma 13.1.1 in \cite{Ch-Kl} or Lemma 7 in \cite{Bieri}.
\end{proof}
\section{Curvature estimates (Theorem \ref{M1})}\label{sec9}
In this section, we prove Theorem \ref{M1} by the $r^p$--weighted estimate method introduced in \cite{Da-Ro} and applied to Bianchi equations in \cite{holzegel,KS}, see also \cite{ShenMink}.\\ \\
Throughout Section \ref{sec9}, we denote
\begin{align*}
\Ve&:=V_t(-\infty,u),\qquad\qquad\quad\, \Vi:=V_t(u_1,u_2),\\
\Vie&:=\Vi\cap\left\{r\geq \frac{t}{2}\right\},\qquad\quad\; \Vii:=\Vi\cap\left\{r\leq\frac{t}{2}\right\}.
\end{align*}
\subsection{Preliminaries}
\subsubsection{Schematic notation \texorpdfstring{$\F_q$}{}, \texorpdfstring{$\Fb_q$}{} and \texorpdfstring{$\O_\ell$}{}}
\begin{df}\label{wonderfuldef}
For a quantity $X$, we denote
\begin{equation*}
    X\in \F_q(s),
\end{equation*}
if $X$ satisfies the following estimates:
\begin{align*}
    \EE_{q,s-q}^1[X]+\DD_{q,s-q}[X]+\EEi_{0,s}^1[X]+\DDi_{-2,s+2}[X]\les\ep^2.
\end{align*}
Similarly, we denote
\begin{equation*}
    X\in\Fb_q(s),
\end{equation*}
if $X$ satisfies the following estimates:
\begin{align*}
    \FF_{q,s-q}^1[X]+\DD_{q-1,s-q+1}[X]+\FFi_{0,s}^1[X]+\DDi_{-2,s+2}[X]\les\ep^2.
\end{align*}
Moreover, we denote
\begin{align*}
    X\in\O_{\ell}(s),
\end{align*}
if $X$ satisfies the following estimates:
\begin{align*}
    \DD^1_{\ell-2,s-\ell}[X]+\DDi^1_{-3,s+1}[X]\les\ep^2.
\end{align*}
We also denote
\begin{align*}
    X\in\F_q^{(1)}(s)\qquad\,\mbox{ if }\quad\quad\qquad\;\;\; \EE_{q,s-q}[X]+\EEi_{0,s}[X]&\les\ep^2,\\
    X\in\Fb_q^{(1)}(s)\qquad \mbox{ if }\quad\,\quad\qquad\; \FF_{q,s-q}[X]+\FFi_{0,s}[X]&\les\ep^2,\\
    X\in\O_\ell^{(1)}(s)\qquad \mbox { if }\quad\quad \DD_{\ell-2,s-\ell}[X]+
    \DDi_{-3,s+1}[X]&\les\ep^2.
\end{align*}
We also use the following shorthand notation:
\begin{align*}
\mathbb{X}_q:=\mathbb{X}_q(s),\qquad\forall\; \mathbb{X}\in\{\F,\,\Fb,\,\O\}.
\end{align*}
\end{df}
\begin{rk}\label{wonderfulrk}
Under the assumptions \eqref{B1}, we have
\begin{align*}
    (\a,\b,\rhoc,\sic)&\in\F_s\cap\Fb_s,\qquad\qquad \bb\in\F_0\cap\Fb_s,\qquad\qquad\aa\in\Fb_0,\\
    \a_4&\in\F_{s}^{(1)},\qquad\qquad\quad\aa_3\in\Fb_{-2}^{(1)}.
\end{align*}
We also have\footnote{Recall that $\db$ is defined in \eqref{dfdb}.}
\begin{align*}
    (\rho,\si,\slu)\in\F_{s-\db},\qquad r^{-1}\Ga_i^{(1)}\in\Fb_{s-\db},\qquad i=g,b,w.
\end{align*}
Moreover, we have from Lemma \ref{decayGagGabGaa}
\begin{align*}
    \Gag\in\O_s,\qquad \Gab\in\O_1,\qquad \Gaw\in\O_{s-1}.
\end{align*}
We also have from Corollary \ref{LieTR}
\begin{align*}
\nabs_T\a&\in\F_{s+2}^{(1)}(s+2),\qquad\,\nabs_T\b\in\F^{(1)}_{2s-\db}(s+2),\qquad\,\nabs_T(\rho,\si)\in\Fb_{s+1}^{(1)}(s+2),\\
\nabs_T\bb&\in\Fb_{2}^{(1)}(s+2),\qquad\;\; \nabs_T\aa\in\Fb_0^{(1)}(s+2).
\end{align*}
\end{rk}
\subsubsection{Bianchi equations in schematic form}
\begin{prop}\label{bianchi}
We have the following Bianchi equations:
\begin{align}
\begin{split}\label{bianchi0}
\nabs_3\a+\frac{1}{2}\trchb\,\a&=-2\sld_2^*\b+\O_1\c\F_s,\\
\nabs_4\b+2\trch\,\b&=\sld_2\a+\O_s\c\F_s,\\
\nabs_3\b+\trchb\,\b&=-\sld_1^*(\rho,-\si)+\O_s\c\Fb_s+\O_1\c\F_{s-\db},\\
\nabs_4(\rho,-\si)+\frac{3}{2}\trch(\rho,-\si)&=\sld_1\b+\O_{s-1}\c\F_s,\\
\nabs_3(\rhoc,\sic)+\frac{3}{2}\trchb(\rhoc,\sic)&=-\sld_1\bb+\O_{s-1}\c\Fb_{s-\db},\\
\nabs_4\bb+\trch\,\bb&=\sld_1^*(\rhoc,\sic)+\O_{s-1}\c\F_{s-\db}+\O_s\c\Fb_{s-\db},\\
\nabs_3\bb+2\trchb\,\bb&=-\sld_2\aa+\O_1\c\Fb_0,\\
\nabs_4\aa+\frac{1}{2}\trch\,\aa&=2\sld_2^*\bb+\O_s\c\Fb_0.
\end{split}
\end{align}
\end{prop}
\begin{proof}
It follows directly from Remark \ref{wonderfulrk}, Lemma \ref{renorequation} and Proposition \ref{bianchischematic}.
\end{proof}
Commuting \eqref{bianchi0} with $\dkb$, applying Propositions \ref{commdkb} and \ref{commutation} and ignoring the terms that decay better, we obtain the following corollary.\footnote{The linear terms $r^{-1}O(\b)$ and $r^{-1}O(\bb)$ come from the fact that $[\dkb,2\sld_2^*]=O(r^{-1})$, see Proposition \ref{commdkb}.}
\begin{cor}\label{bianchifirst}
We have the following equations:
\begin{align*}
\nabs_3(\dkb\a)+\frac{1}{2}\trchb(\dkb\a)&=-\sld_1^*(\dkb\b)+r^{-1}O(\b)+(\O_{s-1}\c\F_s)^{(1)},\\
\nabs_4(\dkb\b)+2\trch(\dkb\b)&=\sld_1(\dkb\a)+(\O_s\c\F_s)^{(1)},\\
\nabs_3(\dkb\b)+\trchb(\dkb\b)&=-\sld_1(\dkb\rho,-\dkb\si)+(\O_s\c\Fb_s)^{(1)}+(\O_{s-1}\c\F_{s-\db})^{(1)},\\
\nabs_4(\dkb\rho,-\dkb\si)+\frac{3}{2}\trch(\dkb\rho,-\dkb\si)&=\sld_1^*(\dkb\b)+(\O_{s-1}\c\F_s)^{(1)},\\
\nabs_3(\dkb\rhoc,\dkb\sic)+\frac{3}{2}\trchb(\dkb\rhoc,\dkb\sic)&=-\sld_1^*(\dkb\bb)+(\O_{s-1}\c\Fb_{s-\db})^{(1)},\\
\nabs_4(\dkb\bb)+\trch(\dkb\bb)&=\sld_1(\dkb\rhoc,\dkb\sic)+(\O_{s-1}\c\F_{s-\db})^{(1)}+(\O_s\c\Fb_{s-\db})^{(1)},\\
\nabs_3(\dkb\bb)+2\trchb(\dkb\bb)&=-\sld_1(\dkb\aa)+(\O_1\c\Fb_0)^{(1)},\\
\nabs_4(\dkb\aa)+\frac{1}{2}\trch(\dkb\aa)&=\sld_1^*(\dkb\bb)+r^{-1}O(\bb)+(\O_s\c\Fb_0)^{(1)}.
\end{align*}
\end{cor}
\begin{lem}\label{teulm}
We define the following quantities:
\begin{align}
\begin{split}
\a_4&:=\frac{1}{r^4}\nabs_4(r^5\a)\in \sk_2,\qquad\quad \as:=r\sld_2\a\in\sk_1,\\
\aa_3&:=\frac{1}{r^4}\nabs_3(r^5\aa)\in \sk_2,\qquad\quad \aas:=r\sld_2\aa\in\sk_1.
\end{split}
\end{align}
Then, we have the following equations:
\begin{align}
\begin{split}\label{teu}
\nabs_3\a_4&=-2\sld_2^*\as+\frac{4\a}{r}+(\O_s\c\F_s)^{(1)},\\
\nabs_4\as+\frac{5}{2}\trch \,\as&=\sld_2\ac+(\O_s\c\F_s)^{(1)},
\end{split}
\end{align}
and
\begin{align}
\begin{split}\label{teuaa}
\nabs_4\aac&=-2\sld_2^*\aas+\frac{4\aa}{r}+(\O_{s-1}\c\Fb_0)^{(1)},\\
\nabs_3\aas+\frac{5}{2}\trchb\,\aas&=\sld_2\aac+(\O_{s-1}\c\Fb_0)^{(1)}.
\end{split}
\end{align}
\end{lem}
\begin{proof}
See Lemma 2.26 in \cite{ShenMink}.
\end{proof}
\subsubsection{General Bianchi pairs}
The following lemma provides the general structure of Bianchi pairs. It will be used repeatedly in Sections \ref{exteriorest}--\ref{ssecTime}.
\begin{lem}\label{keypoint}
Let $k=1,2$ and $a_{(1)}$, $a_{(2)}$ real numbers. Then, we have the following properties.
\begin{enumerate}
    \item If $\psi_{(1)},h_{(1)}\in\sk_k$ and $\psi_{(2)},h_{(2)}\in \sk_{k-1}$ satisfying
    \begin{align}
    \begin{split}\label{bianchi1}
    \nabs_3(\psi_{(1)})+a_{(1)}\trchb\,\psi_{(1)}&=-k\sld_k^*(\psi_{(2)})+h_{(1)},\\
    \nabs_4(\psi_{(2)})+a_{(2)}\trch\,\psi_{(2)}&=\sld_k(\psi_{(1)})+h_{(2)}.
    \end{split}
    \end{align}
Then, the pair $(\psi_{(1)},\psi_{(2)})$ satisfies for any real number $p$
\begin{align}
\begin{split}\label{div}
       &\bdiv(r^p |\psi_{(1)}|^2e_3)+k\bdiv(r^p|\psi_{(2)}|^2e_4)\\
       +&\left(2a_{(1)}-1-\frac{p}{2}\right)r^{p}\trchb|\psi_{(1)}|^2+k\left(2a_{(2)}-1-\frac{p}{2}\right)r^{p}\trch|\psi_{(2)}|^2\\
       =&2k r^p \sdivs(\psi_{(1)}\cdot\psi_{(2)})
       +2r^p\psi_{(1)}\cdot h_{(1)}+2kr^p\psi_{(2)}\cdot h_{(2)}-2r^p\omb|\psi_{(1)}|^2\\
       -&2kr^p\om|\psi_{(2)}|^2+pr^{p-1}\left(e_3(r)-\frac{r}{2}\trchb\right)|\psi_{(1)}|^2+kpr^{p-1}\left(e_4(r)-\frac{r}{2}\tr\chi\right)|\psi_{(2)}|^2.
\end{split}
\end{align}
    \item If $\psi_{(1)},h_{(1)}\in\sk_{k-1}$ and $\psi_{(2)},h_{(2)}\in\sk_k$ satisfying
\begin{align}
    \begin{split}\label{bianchi2}
    \nabs_3(\psi_{(1)})+a_{(1)}\trchb\,\psi_{(1)}&=\sld_k(\psi_{(2)})+h_{(1)},\\
        \nabs_4(\psi_{(2)})+a_{(2)}\trch\,\psi_{(2)}&=-k\sld_k^*(\psi_{(1)})+h_{(2)}.
    \end{split}
\end{align}
Then, the pair $(\psi_{(1)},\psi_{(2)})$ satisfies for any real number $p$
\begin{align}
\begin{split}\label{div2}
       &k\bdiv(r^p |\psi_{(1)}|^2e_3)+\bdiv(r^p|\psi_{(2)}|^2e_4)\\
       +&k\left(2a_{(1)}-1-\frac{p}{2}\right)r^{p}\trchb|\psi_{(1)}|^2+\left(2a_{(2)}-1-\frac{p}{2}\right)r^{p}\trch|\psi_{(2)}|^2\\
       =&2 r^p\sdivs(\psi_{(1)}\cdot\psi_{(2)})
       +2kr^p\psi_{(1)}\cdot h_{(1)}+2r^p\psi_{(2)}\cdot h_{(2)}-2k r^p\omb|\psi_{(1)}|^2\\
       -&2r^p\om|\psi_{(2)}|^2+k pr^{p-1}\left(e_3(r)-\frac{r}{2}\trchb\right)|\psi_{(1)}|^2+pr^{p-1}\left(e_4(r)-\frac{r}{2}\trch\right)|\psi_{(2)}|^2.
\end{split}
\end{align}
\end{enumerate}
\end{lem}
\begin{proof}
    See Lemma 4.2 in \cite{ShenMink}.
\end{proof}
\begin{rk}
    Note that the Bianchi equations in \eqref{bianchi0}, \eqref{teu} and \eqref{teuaa} can be written as systems of equations of the type \eqref{bianchi1} and \eqref{bianchi2}. In particular
    \begin{itemize}
    \item the Bianchi pair $(\a_4,\as)$ satisfies \eqref{bianchi1} with $k=2$, $a_{(1)}=0$, $a_{(2)}=\frac{5}{2}$,
    \item the Bianchi pair $(\a,\b)$ satisfies \eqref{bianchi1} with $k=2$, $a_{(1)}=\frac{1}{2}$, $a_{(2)}=2$,
    \item the Bianchi pair $(\b,(\rho,-\si))$ satisfies \eqref{bianchi1} with $k=1$, $a_{(1)}=1$, $a_{(2)}=\frac{3}{2}$,
    \item the Bianchi pair $((\rhoc,\sic),\bb)$ satisfies \eqref{bianchi2} with $k=1$, $a_{(1)}=\frac{3}{2}$, $a_{(2)}=1$,
    \item the Bianchi pair $(\bb,\aa)$ satisfies \eqref{bianchi2} with $k=2$, $a_{(1)}=2$, $a_{(2)}=\frac{1}{2}$,
    \item the Bianchi pair $(\aas,\aa_3)$ satisfies \eqref{bianchi2} with $k=2$, $a_{(1)}=\frac{5}{2}$, $a_{(2)}=0$.
    \end{itemize}
\end{rk}
\subsubsection{Estimates for nonlinear error terms}
The following theorem provides a unified treatment of all the nonlinear error terms in curvature estimates.
\begin{thm}\label{wonderfulrp}
We define\footnote{In order to simplify the notations, $V$ can be ignored if it is clear in the context.}
\begin{equation}\label{deferr}
    \ee_p^q\left[\psi_{(1)},\psi_{(2)}\right](V):=\int_V r^p \left|\psi_{(1)}^{(q)}\c\psi_{(2)}^{(q)}\right|.
\end{equation}
Let $p_1,p_2,p,\ell\leq s$ and let $s>1$. Then, we have the following properties:
\begin{enumerate}
    \item In the case $2p<p_1+p_2+\ell+1$, we have
\begin{align*}
    \ee^1_p\left[\F_{p_1},\O_\ell\c\F_{p_2}\right](\Ve)&\les\frac{\ep^3}{\ujp^{s-p}}.
\end{align*}
    \item In the case $2p<p_1+p_2+\ell$ and $2p+1<p_1+p_2+s$, we have
\begin{align*}
    \ee^1_p\left[\F_{p_1},\O_\ell\c\Fb_{p_2}\right](\Ve)&\les\frac{\ep^3}{\ujp^{s-p}},\\
    \ee^1_p\left[\Fb_{p_1},\O_\ell\c\F_{p_2}\right](\Ve)&\les\frac{\ep^3}{\ujp^{s-p}}.
\end{align*}
    \item In the case $2p<p_1+p_2+\ell-1$, we have
\begin{align*}
\ee^1_p\left[\Fb_{p_1},\O_\ell\c\Fb_{p_2}\right](\Ve)&\les\frac{\ep^3}{\ujp^{s-p}}.
\end{align*}
\end{enumerate}
If in addition we assume that $p_1,p_2,p\geq 0$. Then, we have the following properties:
\begin{enumerate}
    \item In the case $2p<p_1+p_2+\ell+1$, we have
\begin{align*}
\ee^1_p\left[\F_{p_1},\O_\ell\c\F_{p_2}\right](\Vi)&\les\frac{\ep^3}{u_1^{s-p}}.
\end{align*}
    \item In the case $2p<p_1+p_2+\ell$ and $2p+1<p_1+p_2+s$, we have
\begin{align*}
    \ee^1_p\left[\F_{p_1},\O_\ell\c\Fb_{p_2}\right](\Vi)&\les\frac{\ep^3}{u_1^{s-p}},\\
    \ee^1_p\left[\Fb_{p_1},\O_\ell\c\F_{p_2}\right](\Vi)&\les\frac{\ep^3}{u_1^{s-p}}.
\end{align*}
    \item In the case $2p<p_1+p_2+\ell-1$, we have
\begin{align*}
\ee^1_p\left[\Fb_{p_1},\O_\ell\c\Fb_{p_2}\right](\Vi)&\les\frac{\ep^3}{u_1^{s-p}}.
\end{align*}
\end{enumerate}
\end{thm}
\begin{proof}
   See Appendix \ref{secA}.
\end{proof}
\subsubsection{Strategy of the proof of Theorem \ref{M1}}
The proof of Theorem \ref{M1} proceeds in 4 steps which we summarize below for convenience:
\begin{itemize}
\item In Section \ref{exteriorest}, we apply $r^p$--weighted estimates with $-2\leq p\leq s$ to the Bianchi equations to control the curvature components in the exterior region $\Ve$.
\item In Section \ref{intsec}, we apply $r^p$--weighted estimates with $0\leq p\leq s$\footnote{Notice that the restriction $p\geq 0$ is necessary according to the fact that $r^p$ is unbounded near $\Vphi$ for $p<0$.} to the Bianchi equations to control the curvature components in the interior region $\Vi=\Vie\cup\Vii$. In $\Vie$, the desired decay of curvature components follows directly from the fact that $\ujp\les r$ holds in $\Vie$. On the other hand, in $\Vii$, by the mean value method introduced in \cite{Da-Ro}, we can only deduce $r^{-\frac{3}{2}}u^{-\frac{s}{2}}$--decay for curvature components. The goal of the last two steps is then to recover the expected $r^{-\frac{1}{2}}u^{-\frac{s+2}{2}}$--decay for curvature components in $\Vii$.
\item In Section \ref{ssecTime}, we first commute the Bianchi equations with $\nabs_T$. We then apply $r^{p}$--weighted estimates with $0\leq p\leq s+2$ to the commuted Bianchi equations to control $\nabs_T\R$. By the mean value method, we deduce $r^{-\frac{3}{2}}u^{-\frac{s+2}{2}}$--decay for $\nabs_T\R$ in $\Vii$.
\item Finally, in Section \ref{endM1}, we apply the elliptic estimates of Proposition \ref{hodgerank2} to the Maxwell type systems in Proposition \ref{EHmaxwell} to deduce the desired $r^{-\frac{1}{2}}u^{-\frac{s+2}{2}}$--decay for $\R$ from the estimate of $\nabs_T \R$ in $\Vii$.
\end{itemize}
\subsection{Exterior region estimates}\label{exteriorest}
Throughout Section \ref{exteriorest}, we always assume that $u\leq 1$ and we denote
\begin{equation*}
    V:=V_t(u)=V_t(-\infty,u).
\end{equation*}
\begin{prop}\label{keyintegral}
For $\psi_{(1)}$, $\psi_{(2)}$ and $h_{(1)}$, $h_{(2)}$ satisfying \eqref{bianchi1} or \eqref{bianchi2}, we have the following properties:
\begin{itemize}
\item In the case $2+p-4a_{(1)}>0$ and $4a_{(2)}-2-p>0$, we have
\begin{align}
\begin{split}\label{caseone}
&BEF_p[\psi_{(1)}](u)+BF_p[\psi_{(2)}](u)\\ 
\les&F_p[\psi_{(1)},\psi_{(2)}](\Si_2\cap V)+\ee_p^0[\psi_{(1)},h_{(1)}]+\ee_p^0[\psi_{(2)},h_{(2)}].
\end{split}
\end{align}
\item In the case $2+p-4a_{(1)}\leq 0$ and $4a_{(2)}-2-p>0$, we have
\begin{align}
\begin{split}\label{casetwobis}
&EF_p[\psi_{(1)}](u)+BF_p[\psi_{(2)}](u) \\ 
\les&F_p[\psi_{(1)},\psi_{(2)}](\Si_2\cap V)+B_p[\psi_{(1)}](u)+\ee_p^0[\psi_{(1)},h_{(1)}]+\ee_p^0[\psi_{(2)},h_{(2)}].
\end{split}
\end{align}
\item In the case $2+p-4a_{(1)}\leq 0$ and $4a_{(2)}-2-p\geq 0$, we have
\begin{align}
\begin{split}\label{casethree}
&EF_p[\psi_{(1)}](u)+F_p[\psi_{(2)}](u) \\ 
\les&F_p[\psi_{(1)},\psi_{(2)}](\Si_2\cap V)+B_p[\psi_{(1)}](u)+\ee_p^0[\psi_{(1)},h_{(1)}]+\ee_p^0[\psi_{(2)},h_{(2)}].
\end{split}
\end{align}
\end{itemize}
\end{prop}
\begin{proof}
We have from Lemma \ref{dint}
\begin{align*}
e_4(r)-\frac{r}{2}\trch=\frac{\ov{\phi\trch}}{2\phi}r-\frac{r}{2}\trch=\frac{r}{2\phi}(\ov{\phi\trch}-\phi\trch)\in r\Gag,
\end{align*}
and similarly
\begin{align*}
e_3(r)-\frac{r}{2}\trchb=e_4(r)-2N(r)-\frac{r}{2}(\trch-2\tr\th)=r\Gag-2\left(N(r)-\frac{r}{2}\tr\th\right)\in r\Gab.
\end{align*}
Applying Lemma \ref{decayGagGabGaa}, we infer
\begin{align}
\begin{split}\label{Gaapsi}
&\;\;\;\,\,\,\int_V r^p\left(|\Gab||\psi_{(1)}|^2+|\Gag||\psi_{(2)}|^2\right)\\
&\les\int_V \frac{\ep}{\ujp^\frac{s+1}{2}}r^p|\psi_{(1)}|^2+\int_V\frac{\ep}{t^\frac{s+1}{2}}r^p|\psi_{(2)}|^2\\
&\les\ep \int_{-\infty}^u\left(\ujp^{-\frac{s+1}{2}}\int_\cuv r^{p}|\psi_{(1)}|^2\right) du+\ep\int_{2}^t\left(t^{-\frac{s+1}{2}}\int_\ucuv r^{p}|\psi_{(2)}|^2\right)dt\\
&\les\ep\sup_u\left(\int_\cuv r^{p}|\psi_{(1)}|^2\right)\int_{-\infty}^u\ujp^{-\frac{s+1}{2}}du+\ep\sup_{t}\left(\int_\ucuv r^{p}|\psi_{(2)}|^2\right)\int_{2}^{t}t^{-\frac{s+1}{2}}dt\\
&\les\ep\sup_u\left(\int_\cuv r^{p}|\psi_{(1)}|^2\right)+\ep\sup_t\left(\int_\ucuv r^{p}|\psi_{(2)}|^2\right).
\end{split}
\end{align}
Integrating \eqref{div} or \eqref{div2} in $V$\footnote{See Figure \ref{Vt} for a geometric description of $V$.}, applying Stokes formula and reminding that $\trch\simeq\frac{2}{r}$, $\trchb\simeq -\frac{2}{r}$, $\om,\omb\in\Gag$, we obtain
\begin{align*}
&\int_{\cuv}r^p |\psi_{(1)}|^2+\int_\ucuv r^p(|\psi_{(1)}|^2+|\psi_{(2)}|^2)+\int_{\II_*\cap V} r^p|\psi_{(2)}|^2\\
+&\int_{V} (2+p-4a_{(1)})r^{p-1}|\psi_{(1)}|^2+(4a_{(2)}-2-p)r^{p-1}|\psi_{(2)}|^2 \\ 
\les &\int_{\Si_2\cap V}r^p|\psi_{(1)}|^2+r^p|\psi_{(2)}|^2+\int_{V} r^p|\psi_{(1)}||h_{(1)}|+r^p|\psi_{(2)}||h_{(2)}|+r^p|\Gab||\psi_{(1)}|^2+r^p|\Gag||\psi_{(2)}|^2,
\end{align*}
where $\II_*$ denotes the future null infinity. Taking the supremum of $u$ and $t$ and applying \eqref{Gaapsi}, we obtain for $\ep$ small enough
\begin{align}
\begin{split}\label{psipsipsi}
&\sup_u\int_{\cuv}r^p |\psi_{(1)}|^2+\sup_t\int_\ucuv r^p(|\psi_{(1)}|^2+|\psi_{(2)}|^2)\\ +&\int_{V}(2+p-4a_{(1)})r^{p-1}|\psi_{(1)}|^2+(4a_{(2)}-2-p)r^{p-1}|\psi_{(2)}|^2 \\ 
\les &\int_{\Si_2\cap V} r^p |\psi_{(1)}|^2+r^p|\psi_{(2)}|^2+\int_{V}r^p|\psi_{(1)}||h_{(1)}|+r^p|\psi_{(2)}||h_{(2)}|,
\end{split}
\end{align}
which implies 
\begin{align*}
&EF_p[\psi_{(1)}](u)+F_p[\psi_{(2)}](u)+(2+p-4a_{(1)})B_p[\psi_{(1)}](u)+(4a_{(2)}-2-p)B_p[\psi_{(2)}](u)\\
\les \; &F_p[\psi_{(1)},\psi_{(2)}](\Si_2\cap V)+B_p[\psi_{(2)}](u)+\ee_p^0[\psi_{(1)},h_{(1)}]+\ee_p^0[\psi_{(2)},h_{(2)}].
\end{align*}
Hence, \eqref{caseone}--\eqref{casethree} hold in the corresponding range of parameters. This concludes the proof of Proposition \ref{keyintegral}.
\end{proof}
\begin{figure}
    \centering
    \includegraphics[width=16.5cm]{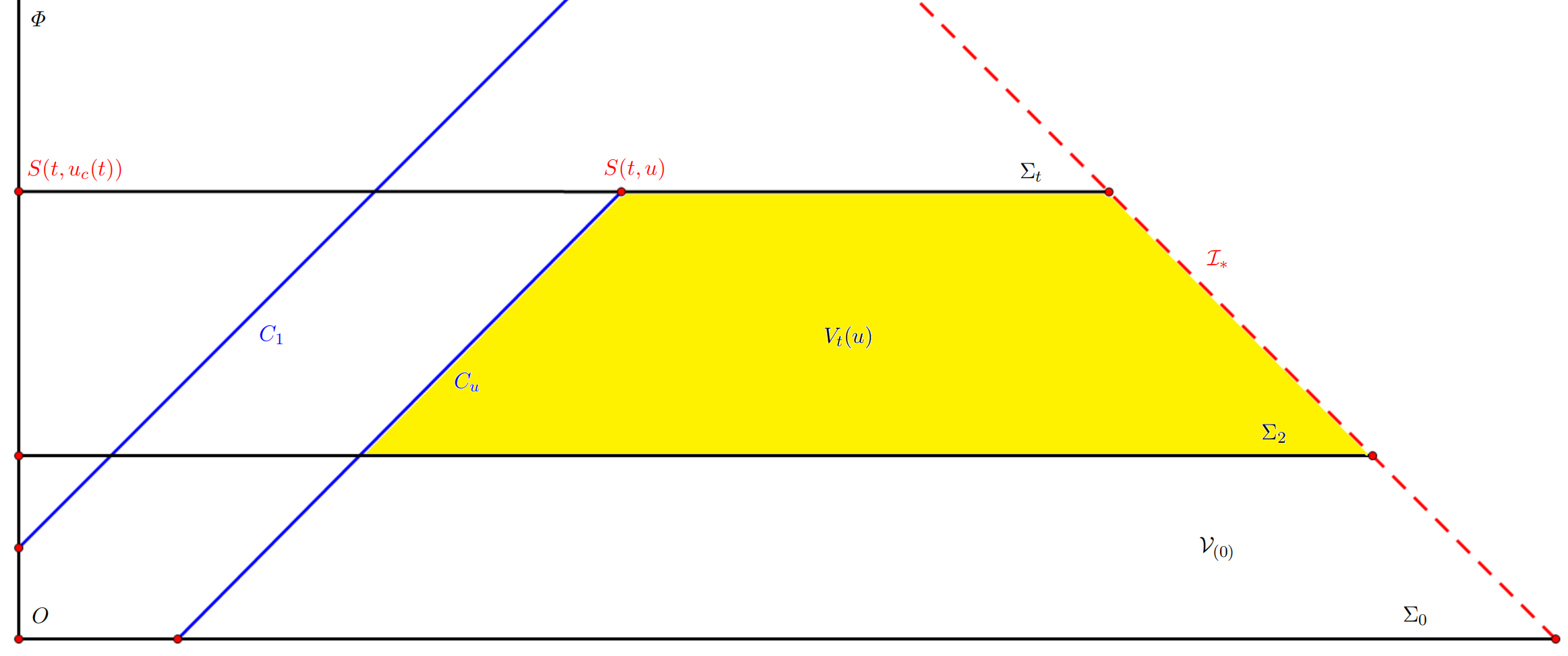}
    \caption{Domain of integration $V_t(u)$}
    \label{Vt}
\end{figure}
The following lemma allows us to obtain $\ujp$--decay of curvature along $\Si_2\cap V$. 
\begin{lem}\label{gainu}
We have the following estimate for $p\leq s$ and $u\leq 1$:
\begin{align*}
\int_{\Si_2\cap V}r^p\left(|\a^{(1)}|^2+|\b^{(1)}|^2+|(\rhoc^{(1)},\sic^{(1)})|^2+|\bb^{(1)}|^2+|\aa^{(1)}|^2\right)&\les \frac{\Rk}{\ujp^{s-p}}.
\end{align*}
\end{lem}
\begin{proof}
It follows directly from the fact that $\ujp\les r$ on $\Si_2\cap V$.
\end{proof}
\subsubsection{Estimates for the Bianchi pair \texorpdfstring{$(\a,\b)$}{}}\label{ssec9.1}
\begin{prop}\label{estab}
We have the following estimate:
\begin{equation}
BEF_s^1[\a](u)+BF_s^1[\b](u)\les\ep_0^2.\label{abs}
\end{equation}
\end{prop}
\begin{proof}
We recall from Proposition \ref{bianchi}
\begin{align*}
    \nabs_3\a+\frac{1}{2}\trchb\,\a&=-2\sld_2^*\b+\O_1\c\F_s,\\
\nabs_4\b+2\trch\,\b&=\sld_2\a+\O_s\c\F_s.
\end{align*}
Applying \eqref{caseone} with $\psi_{(1)}=\a$, $\psi_{(2)}=\b$, $a_{(1)}=\frac{1}{2}$, $a_{(2)}=2$, $h_{(1)}=\O_1\c\F_s$, $h_{(2)}=\O_s\c\F_s$, $p=s$ and noticing that
\begin{equation*}
    2+p-4a_{(1)}=s>0,\qquad 4a_{(2)}-2-p=6-s>0,
\end{equation*}
we obtain from \eqref{deferr} and Lemma \ref{gainu}
\begin{align*}
BEF_s[\a](u)+BF_s[\b](u)\les\ep_0^2+\ee_s^0[\F_s,\O_1\c\F_s]+\ee_s^0[\F_s,\O_s\c\F_s].
\end{align*}
Applying Theorem \ref{wonderfulrp}, we obtain
\begin{align*}
    \ee_s^0[\F_s,\O_1\c\F_s]+\ee_s^0[\F_s,\O_s\c\F_s]\les\ep_0^2.
\end{align*}
Hence, we have
\begin{equation}\label{bulkterm0}
BEF_s[\a](u)+BF_s[\b](u)\les\ep_0^2.
\end{equation}
Next, we recall from Corollary \ref{bianchifirst}
\begin{align*}
\nabs_3(\dkb\a)+\frac{1}{2}\trchb(\dkb\a)&=-\sld_1^*(\dkb\b)+r^{-1}O(\b)+(\O_{s-1}\c\F_s)^{(1)},\\
\nabs_4(\dkb\b)+2\trch(\dkb\b)&=\sld_1(\dkb\a)+(\O_s\c\F_s)^{(1)}.
\end{align*}
Applying \eqref{caseone} with $\psi_{(1)}=\dkb\a$, $\psi_{(2)}=\dkb\b$, $a_{(1)}=\frac{1}{2}$, $a_{(2)}=2$, $h_{(1)}=r^{-1}O(\b)+(\O_{s-1}\c\F_s)^{(1)}$, $h_{(2)}=(\O_s\c\F_s)^{(1)}$ and $p=s$, we infer
\begin{align*}
    BEF_s^1[\a](u)+BF_s^1[\b](u)\les\ep_0^2+\ee_{s-1}^0[\a^{(1)},\b]+\ee_s^1[\F_s,\O_{s-1}\c\F_s]+\ee_s^1[\F_s,\O_s\c\F_s].
\end{align*}
We have from Cauchy-Schwarz inequality that for all $\de_0>0$
\begin{align*}
    \ee_{s-1}^0[\a^{(1)},\b]=\int_V r^{s-1}|\a^{(1)}||\b|\leq \de_0 B_s^1[\a](u)+(4\de_0)^{-1} B_s[\b](u).
\end{align*}
Moreover, applying Theorem \ref{wonderfulrp}, we have
\begin{align*}
    \ee_s^1[\F_s,\O_{s-1}\c\F_s]+\ee_s^1[\F_s,\O_s\c\F_s]\les\ep_0^2.
\end{align*}
Hence, combining with \eqref{bulkterm0}, we have for $\de_0$ small enough
\begin{equation}\label{bulkterm}
BEF_s^1[\a](u)+BF_s^1[\b](u)\les\ep_0^2.
\end{equation}
This concludes the proof of Proposition \ref{estab}.
\end{proof}
\subsubsection{Estimates for the Bianchi pair \texorpdfstring{$(\b,(\rho,-\si))$}{}}\label{ssec9.2}
\begin{prop}\label{estbr}
We have the following estimate:
\begin{equation}\label{eqbr}
BEF_s^1[\b](u)+BF_s^1[\rhoc,\sic](u)\les\ep_0^2.
\end{equation}
\end{prop}
\begin{proof}
We recall from Corollary \ref{bianchifirst}
\begin{align*}
\nabs_3(\dkb\b)+\trchb(\dkb\b)&=-\sld_1(\dkb\rho,-\dkb\si)+(\O_s\c\Fb_s)^{(1)}+(\O_{s-1}\c\F_s)^{(1)},\\
\nabs_4(\dkb\rho,-\dkb\si)+\frac{3}{2}\trch(\dkb\rho,-\dkb\si)&=\sld_1^*(\dkb\b)+(\O_{s-1}\c\F_s)^{(1)}.
\end{align*}
Applying \eqref{casetwobis} with $\psi_{(1)}=\dkb\b$, $\psi_{(2)}=(\dkb\rho,-\dkb\si)$, $a_{(1)}=1$, $a_{(2)}=\frac{3}{2}$, $h_{(1)}=(\O_s\c\Fb_s)^{(1)}+(\O_{s-1}\c\F_s)^{(1)}$, $h_{(2)}=(\O_{s-1}\c\F_s)^{(1)}$ and $p=s$, we obtain
\begin{align*}
&BEF_s^1[\b](u)+BF_s^1[\rho,\si](u)\\
\les&\;F_s^1[\b,\rho,\si](\Si_2\cap V)+B_s^1[\b](u)+\ee_s^1[\F_s,\O_s\c\Fb_s]+\ee_s^1[\F_s,\O_{s-1}\c\F_s]+\ee_s^1[\F_{s-\db},\O_{s-1}\c\F_s]\\
\les&\; \ep_0^2+\ee_s^1[\F_s,\O_s\c\Fb_s]+\ee_s^1[\F_s,\O_{s-1}\c\F_s]+\ee_s^1[\F_{s-\db},\O_{s-1}\c\F_s],
\end{align*}
where we used \eqref{bulkterm} and Lemma \ref{gainu} in the last step. Next, applying Theorem \ref{wonderfulrp} in corresponding cases, we infer
\begin{align*}
\ee_s^1[\F_s,\O_s\c\Fb_s]+\ee_s^1[\F_s,\O_{s-1}\c\F_s]+\ee_s^1[\F_{s-\db},\O_{s-1}\c\F_s]\les\ep_0^2.
\end{align*}
Hence, we deduce
\begin{equation}\label{brbulk}
BEF_s^1[\b](u)+BF_s^1[\rho,\si](u)\les\ep_0^2.
\end{equation}
Recalling \eqref{renorq}, we infer
\begin{align*}
    F_s^1[\rhoc,\sic](u)&\les F_s^1[\rho,\si](u)+\int_{-\infty}^u r^{s+2}|r^{-\frac{2}{4}}\Gag|^2_{4,S}|r^{-\frac{2}{4}}\Gaw|^2_{4,S}\\
    &\les\ep_0^2+\int_{-\infty}^u r^{s+2}\frac{\ep^2}{r^{s+1}}\frac{\ep^2}{r^2\ujp^{s-1}}du\\
    &\les\ep_0^2,
\end{align*}
and
\begin{align*}
    B_s^1[\rhoc,\sic](u)&\les B_s^1[\rho,\si](u)+\int_{2}^{t}dt\int_{-\infty}^u r^{s+1}|r^{-\frac{2}{4}}\Gag|_{4,S}^2|r^{-\frac{2}{4}}\Gaw|^2_{4,S}du\\
    &\les \ep_0^2+\int_{2}^{t}dt\int_{-\infty}^u r^{s+1}\frac{\ep^2}{r^{s+1}}\frac{\ep^2}{r^2\ujp^{s-1}}du\\
    &\les\ep_0^2.
\end{align*}
This concludes the proof of Proposition \ref{estbr}.
\end{proof}
\subsubsection{Estimates for the Bianchi pair \texorpdfstring{$((\rhoc,\sic),\bb)$}{}}\label{ssec9.3}
\begin{prop}\label{estrb}
We have the following estimate:
\begin{equation}
BEF_s^1[\rhoc,\sic](u)+F_s^1[\bb](u)\les \ep_0^2.\label{rb00}
\end{equation}
\end{prop}
\begin{proof}
We recall from Corollary \ref{bianchifirst}
\begin{align*}
\nabs_3(\dkb\rhoc,\dkb\sic)+\frac{3}{2}\trchb(\dkb\rhoc,\dkb\sic)&=-\sld_1^*(\dkb\bb)+(\O_{s-1}\c\Fb_{s-\db})^{(1)},\\
\nabs_4(\dkb\bb)+\trch(\dkb\bb)&=\sld_1(\dkb\rhoc,\dkb\sic)+(\O_{s-1}\c\F_{s-\db})^{(1)}+(\O_s\c\Fb_{s-\db})^{(1)}.
\end{align*}
Applying \eqref{casethree} with $\psi_{(1)}=(\dkb\rhoc,\dkb\sic)$, $\psi_{(2)}=\dkb\bb$, $a_{(1)}=1$, $a_{(2)}=\frac{3}{2}$, $h_{(1)}=(\O_{s-1}\c\F_s)^{(1)}$, $h_{(2)}=(\O_{s-1}\c\F_{s})^{(1)}+(\O_s\c\Fb_s)^{(1)}$ and $p=s$, we obtain from \eqref{eqbr} and Lemma \ref{gainu}
\begin{align*}
&BEF_s^1[\rhoc,\sic](u)+F_s^1[\bb](u)\\
\les &\;\ep_0^2+\ee_s^1[\F_s,\O_{s-1}\c\Fb_{s-\db}]+\ee_s^1[\Fb_s,\O_{s-1}\c\F_{s-\db}]+\ee_s^1[\Fb_s,\O_s\c\Fb_{s-\db}].
\end{align*}
Applying Theorem \ref{wonderfulrp} in corresponding cases, we deduce
\begin{align}\label{first}
\ee_s^1[\F_s,\O_{s-1}\c\Fb_{s-\db}]+\ee_s^1[\Fb_s,\O_{s-1}\c\F_{s-\db}]+\ee_s^1[\Fb_s,\O_s\c\Fb_{s-\db}]\les\ep_0^2.
\end{align}
Thus, we obtain
\begin{align*}
BEF_s^1[\rhoc,\sic](u)+F_s^1[\bb](u)\les\ep_0^2.
\end{align*}
This concludes the proof of Proposition \ref{estrb}.
\end{proof}
\subsubsection{Estimates for the Bianchi pair \texorpdfstring{$(\bb,\aa)$}{}}\label{ssec9.4}
\begin{prop}\label{estba}
We have the following estimate:
\begin{equation}\label{bbaa00}
BEF_0^1[\bb](u)+F_0^1[\aa](u)\les\frac{\ep_0^2}{\ujp^s}.
\end{equation}
\end{prop}
\begin{proof}
We have from Corollary \ref{bianchifirst}
\begin{align*}
\nabs_3(\dkb\bb)+2\trchb(\dkb\bb)&=-\sld_1(\dkb\aa)+(\O_1\c\Fb_0)^{(1)},\\
\nabs_4(\dkb\aa)+\frac{1}{2}\trch(\dkb\aa)&=\sld_1^*(\dkb\bb)+r^{-1}O(\bb)+(\O_s\c\Fb_0)^{(1)}.
\end{align*}
Applying \eqref{casethree} with $\psi_{(1)}=\dkb\bb$, $\psi_{(2)}=\dkb\aa$, $a_{(1)}=2$, $a_{(2)}=\frac{1}{2}$, $h_{(1)}=(\O_1\c\Fb_0)^{(1)}$, $h_{(2)}=(\O_s\c\Fb_0)^{(1)}$ and $p=0$, we obtain from Lemma \ref{gainu}
\begin{align}
\begin{split}\label{EFbbFa}
& EF_0^1[\bb](u)+F_0^1[\aa](u)\\
\les&\; \frac{\ep_0^2}{\ujp^s}+B_0^1[\bb](u)+\ee_{-1}^0[\aa^{(1)},\bb]+\ee_0^1[\Fb_s,\O_1\c\Fb_0]+\ee^1_0[\Fb_0,\O_s\c\Fb_0].
\end{split}
\end{align}
We first have
\begin{align*}
    \ee_{-1}^0[\aa^{(1)},\bb]&=\int_2^t \frac{dt}{(t+|u|)^\frac{s+2}{2}} \left(\int_\ucuv |\aa^{(1)}|^2\right)^\frac{1}{2}\left(\int_\ucuv r^s|\bb|^2\right)^\frac{1}{2}\\
    &\les \int_2^t \frac{dt}{(t+|u|)^\frac{s+2}{2}} F_0^1[\aa](u)^\frac{1}{2}F_s[\bb](u)^\frac{1}{2}\\
    &\les \frac{\ep_0}{\ujp^\frac{s}{2}} \sup_{2\leq t\leq t_*} F_0^1[\aa](u)^\frac{1}{2},
\end{align*}
which implies that for all $\de_0>0$
\begin{align}\label{linearbb}
    \ee_{-1}^0[\aa^{(1)},\bb]\leq \de_0\sup_{2\leq t\leq t_*} F_0^1[\aa](u)+C\de_0^{-1}\frac{\ep_0^2}{\ujp^s},
\end{align}
where $C$ is a positive constant independent of $\de_0$. Next, applying \eqref{casethree} to the Bianchi pair $((\dkb\rhoc,\dkb\sic),\dkb\bb)$ with $p=0$ and proceeding as in Proposition \ref{estrb}, we obtain from \eqref{eqbr}, Lemma \ref{gainu} and Theorem \ref{wonderfulrp}
\begin{align}\label{B0bb}
\begin{split}
B_0^1[\bb](u)&\les\frac{\ep_0^2}{\ujp^s}+\ee_0^1[\F_s,\O_{s-1}\c\Fb_{s-\db}]+\ee_0^1[\Fb_s,\O_{s-1}\c\F_{s-\db}]+\ee_0^1[\Fb_{s},\O_{s}\c\F_{s-\db}]\\
&\les\frac{\ep_0^2}{\ujp^s}.
\end{split}
\end{align}
Finally, applying Theorem \ref{wonderfulrp} once again, we deduce
\begin{align}
    \ee_0^1[\Fb_s,\O_1\c\Fb_0]+\ee^1_0[\Fb_0,\O_s\c\Fb_0]\les\frac{\ep_0^2}{\ujp^s}.\label{ee01ba}
\end{align}
Injecting \eqref{linearbb}, \eqref{B0bb} and \eqref{ee01ba} into \eqref{EFbbFa}, we have for $\de_0$ small enough
\begin{align*}
BEF_0^1[\bb](u)+F_0^1[\aa](u)\les\frac{\ep_0^2}{\ujp^s}.
\end{align*}
This concludes the proof of Proposition \ref{estba}.
\end{proof}
\subsubsection{Estimate for \texorpdfstring{$\a_4$}{}}
\begin{prop}\label{estaaa}
We have the following estimate:
\begin{equation}
    BEF_s[\a_4](u)\les\ep_0^2.
\end{equation}
\end{prop}
\begin{proof}
We have from \eqref{teu}
\begin{align*}
\nabs_3\a_4&=-2\sld_2^*\as+\frac{4\a}{r}+(\O_s\cdot\F_s)^{(1)},\\
\nabs_4\as+\frac{5}{2}\trch\,\as&=\sld_2\ac+(\O_s\cdot\F_s)^{(1)}.
\end{align*}
Applying \eqref{caseone} with $\psi_{(1)}=\a_4$, $\psi_{(2)}=\as$, $a_{(1)}=0$, $a_{(2)}=\frac{5}{2}$, $h_{(1)}=\frac{4\a}{r}+(\O_s\c\F_s)^{(1)}$, $h_{(2)}=(\O_s\c\F_s)^{(1)}$ and $p=s$, we deduce
\begin{align*}
    BEF_s[\a_4](u)+BF_s^1[\a](u)\les\ep_0^2+\ee_{s-1}^0[\a_4,\a]+\ee_s^1[\F_s,\O_s\c\F_s].
\end{align*}
Applying Cauchy-Schwarz inequality and \eqref{bulkterm}, we obtain for any $\de_0>0$
\begin{align*}
    \ee_{s-1}^0[\a_4,\a]\leq\de_0\int_V r^{s-1}|\a_4|^2+\frac{1}{4\de_0}\int_V r^{s-1}|\a|^2\leq \de_0 B_s[\a_4](u)+(4\de_0)^{-1} B_s[\a].
\end{align*}
Moreover, applying Theorem \ref{wonderfulrp}, we have
\begin{align*}
    \ee_s^1[\F_s,\O_s\c\F_s]\les\ep_0^2.
\end{align*}
Thus, we obtain for $\de_0$ small enough
\begin{align*}
    BEF_s[\a_4](u)+BF_s^1[\a](u)\les\ep_0^2.
\end{align*}
This concludes the proof of Proposition \ref{estaaa}.
\end{proof}
\subsubsection{Estimate for \texorpdfstring{$\aa_3$}{}}
\begin{prop}\label{estaa3}
We have the following estimate:
\begin{equation}\label{aa3}
     EF_{-2}^1[\aa](u)+F_{-2}[\aa_3](u)\les\frac{\ep_0^2}{\ujp^{s+2}}.
\end{equation}
\end{prop}
\begin{proof}
We recall from \eqref{teuaa}
\begin{align*}
\nabs_4\aac&=-2\sld_2^*\aas+\frac{4\aa}{r}+(\O_{s-1}\c\Fb_0)^{(1)},\\
\nabs_3\aas+\frac{5}{2}\trchb\,\aas&=\sld_2\aac+(\O_{s-1}\c\Fb_0)^{(1)}.
\end{align*}
Applying \eqref{casethree} with $\psi_{(1)}=\aas$, $\psi_{(2)}=\aac$, $a_{(1)}=\frac{5}{2}$, $a_{(2)}=0$, $h_{(1)}=(\O_{s-1}\c\Fb_0)^{(1)}$, $h_{(2)}=\frac{4\aa}{r}+(\O_{s-1}\c\Fb_0)^{(1)}$ and $p<-2$, we obtain from Lemma \ref{gainu}
\begin{align*}
EF_{p}^1[\aa](u)+BF_{p}[\aa_3](u)\les\frac{\ep_0^2}{\ujp^{s-p}}+B_p^1[\aa](u)+\ee_{p-1}^0[\aa_3,\aa]+\ee_p^0\left[\Fb_{-2}^{(1)},(\O_{s-1}\c\Fb_0)^{(1)}\right].
\end{align*}
Applying \eqref{casethree} to the Bianchi pair $(\bb,\aa)$ with $p<0$ and proceeding as in Proposition \ref{estba}, we obtain from Lemma \ref{gainu} and Theorem \ref{wonderfulrp}
\begin{align}\label{Bpaa}
B_{p}^1[\aa](u)\les\frac{\ep_0^2}{\ujp^{s-p}},\qquad\forall p<0.
\end{align}
We have from Cauchy-Schwarz inequality that for any $\de_0>0$:
\begin{align*}
    \ee_{p-1}^0[\aa_3,\aa]\leq \int_V r^{p-1}|\aa_3||\aa|\leq \de_0 B_{p}[\aa_3](u)+(4\de_0)^{-1}B_{p}[\aa](u).
\end{align*}
We also have from Theorem \ref{wonderfulrp}
\begin{equation*}
    \ee_p^0\left[\Fb_{-2}^{(1)},(\O_{s-1}\c\Fb_0)^{(1)}\right]\les \ee_p^1\left[\Fb_{-2},\O_{s-1}\c\Fb_0\right]\les\frac{\ep_0^2}{\ujp^{s-p}}.
\end{equation*}
Combining the above estimates, we deduce for $\de_0$ small enough
    \begin{align}\label{Bpaa3}
        B_{p}[\aa_3](u)\les\frac{\ep_0^2}{\ujp^{s-p}},\qquad \forall p<-2.
    \end{align}
Next, applying \eqref{casethree} with $\psi_{(1)}=\aas$, $\psi_{(2)}=\aac$, $a_{(1)}=\frac{5}{2}$, $a_{(2)}=0$, $h_{(1)}=(\O_{s-1}\c\Fb_0)^{(1)}$, $h_{(2)}=\frac{4\aa}{r}+(\O_{s-1}\c\Fb_0)^{(1)}$ and $p=-2$, we obtain from Lemma \ref{gainu} and \eqref{Bpaa}
\begin{align*}
    EF_{-2}^1[\aa](u)+F_{-2}[\aa_3](u)\les\frac{\ep_0^2}{\ujp^{s+2}}+\ee_{-3}^0[\aa_3,\aa]+\ee_{-2}^0\left[\Fb_{-2}^{(1)},(\O_{s-1}\c\Fb_0)^{(1)}\right].
\end{align*}
We then have from \eqref{Bpaa} and \eqref{Bpaa3}
\begin{align*}
    \ee_{-3}^0[\aa_3,\aa]\les B_{-3}[\aa_3](u)^\frac{1}{2}B_{-1}[\aa](u)^\frac{1}{2}\les \frac{\ep_0^2}{\ujp^{s+2}}.
\end{align*}
We also have from Theorem \ref{wonderfulrp}
\begin{equation*}
    \ee_{-2}^0\left[\Fb_{-2}^{(1)},(\O_{s-1}\c\Fb_0)^{(1)}\right]\les\frac{\ep_0^2}{\ujp^{s+2}}.
\end{equation*}
Combining the above estimates, we deduce
\begin{align*}
    EF_{-2}^1[\aa](u)+F_{-2}[\aa_3](u)\les\frac{\ep_0^2}{\ujp^{s+2}}.
\end{align*}
This concludes the proof of Proposition \ref{estaa3}.
\end{proof}
\subsection{Interior region estimates}\label{intsec}
Throughout Section \ref{intsec}, we always assume that $1\leq u_1\leq u_2\leq u_c(t)$ and we denote
\begin{align*}
V:=V_t(u_1,u_2).
\end{align*}
Also, we recall that
\begin{align*}
    V=\Vii\cup\Vie,
\end{align*}
where
\begin{equation*}
    \Vii:=V\cap\left\{r\leq \frac{t}{2}\right\},\qquad\quad \Vie:=V\cap\left\{r\geq \frac{t}{2}\right\}.
\end{equation*}
\subsubsection{Domain of integration}
\begin{lem}\label{axislimitR}
Let $\C_\de$ be the cylinder of radius $\de>0$ and centered on $\Vphi$, i.e.
\begin{align*}
    \C_\de:=\{q\in\M\big/\, r(q)\leq \de\}.
\end{align*}
Then, for any $p\geq -1$, there exists a sequence $\de_n\to 0$ satisfying 
\begin{align*}
    \lim_{n\to\infty}\int_{\pr\C_{\de_n}\cap V}r^p|\Psi^{(1)}|^2=0,\qquad \Psi\in\{\a,\b,\rho,\si,\rhoc,\sic,\bb,\aa\},
\end{align*}
where
\begin{align*}
    \pr\C_\de=\{q\in \M\big/\, r(q)=\de\},
\end{align*}
denotes the boundary of $\C_\de$.
\end{lem}
\begin{proof}
    We have from \eqref{B1}
    \begin{align*}
        \FFi_{-2,s+2}^1[\Psi]\les\ep^2,
    \end{align*}
    which implies in particular
    \begin{align*}
        \int_0^1d\de\int_{\pr\C_{\de}\cap V}r^{-2}|\Psi^{(1)}|^2\les\ep^2 t.
    \end{align*}
    Noticing that $r=\de$ on $\pr\C_{\de}$, we infer
    \begin{align*}
        \int_0^1\frac{d\de}{\de^{p+2}}\int_{\pr\C_\de\cap V}r^p|\Psi^{(1)}|^2\les\ep^2 t<+\infty.
    \end{align*}
    Hence, for $p\geq -1$, we deduce that there exists a sequence $\de_n\to 0$ satisfying
    \begin{align*}
        \lim_{n\to\infty}\int_{\pr\C_{\de_n}\cap V}r^p|\Psi^{(1)}|^2=0.
    \end{align*}
    This concludes the proof of Lemma \ref{axislimitR}.
\end{proof}
\begin{figure}
    \centering
    \includegraphics[width=15cm]{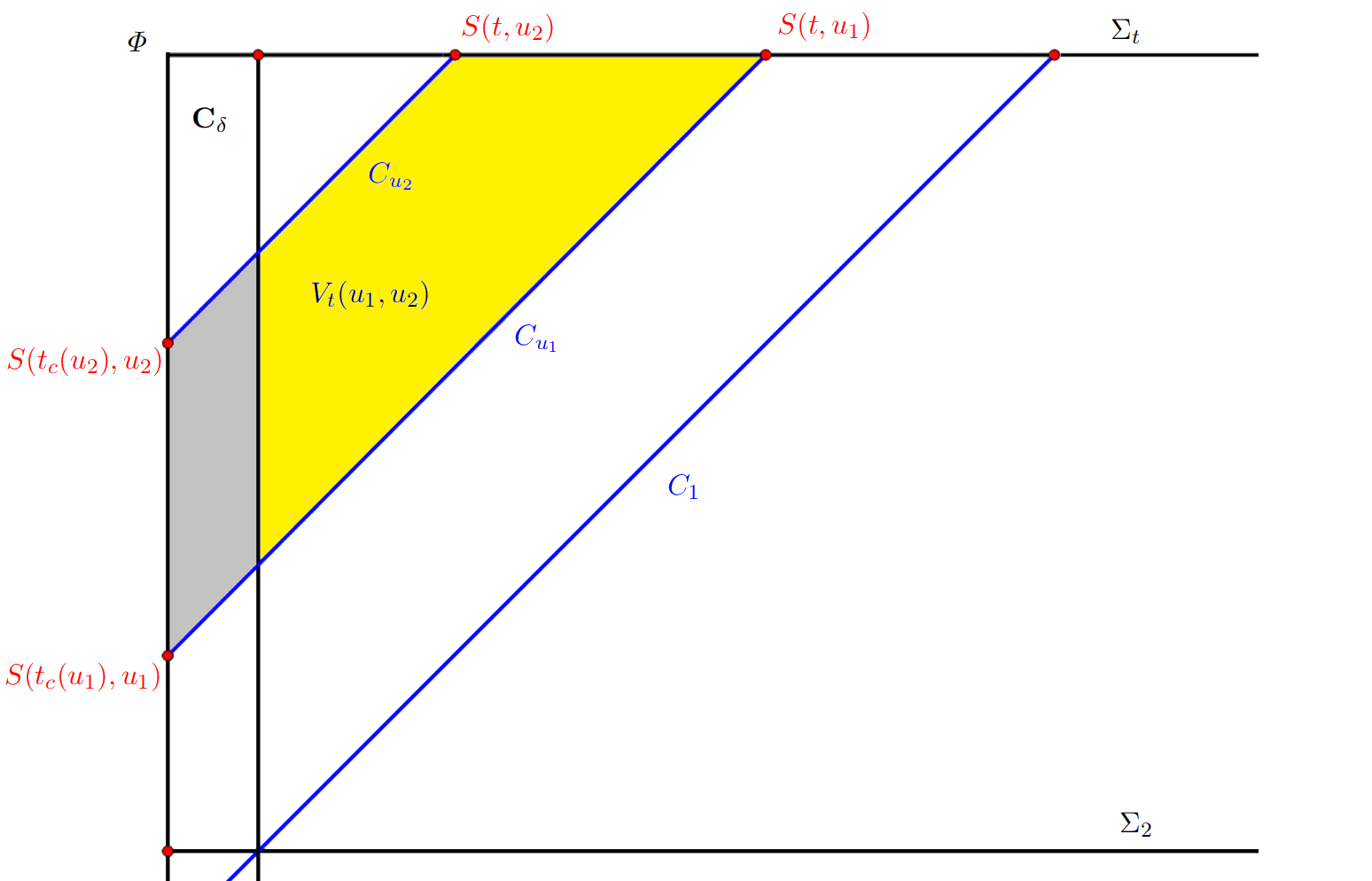}
    \caption{Domain of integration $V_t(u_1,u_2)\setminus\C_\de$}
    \label{Vtu1u2}
\end{figure}
We have the following analog of Proposition \ref{keyintegral}.
\begin{prop}\label{keyintegralint}
Let $p\geq -1$ and let $\psi_{(1)}$ and $\psi_{(2)}$ be curvature components, i.e.
\begin{align*}
    \psi_{(1)},\psi_{(2)}=R^{(q)},\quad\mbox{ where }\;R\in \{\a,\b,(\rho,-\si),(\rhoc,\sic),\bb,\aa\}.
\end{align*}
Assume that $\psi_{(1)}$ and $\psi_{(2)}$ satisfy \eqref{bianchi1} or \eqref{bianchi2}. Then, we have 
\begin{align*}
&EF_p[\psi_{(1)}](u_2)+F_p[\psi_{(2)}](u_1,u_2)\\
+&(2+p-4a_{(1)})B_p[\psi_{(1)}](u_1,u_2)+(4a_{(2)}-2-p)B_p[\psi_{(2)}](u_1,u_2)\\
\les&\;E_p[\psi_{(1)}](u_1)+\ee_p^0[\psi_{(1)},h_{(1)}]+\ee_p^0[\psi_{(2)},h_{(2)}].
\end{align*}
\end{prop}
\begin{proof}
Integrating \eqref{div} or \eqref{div2} in $V\setminus\C_{\de}$ described in Figure \ref{Vtu1u2} and proceeding as in Proposition \ref{keyintegral}, we obtain for all $\de>0$
\begin{align*}
&\int_{C_{u_2}^V\setminus\C_\de}r^p |\psi_{(1)}|^2+\int_{\ucuv\setminus\C_\de} r^p(|\psi_{(1)}|^2+|\psi_{(2)}|^2)\\
+&\int_{V} (2+p-4a_{(1)})r^{p-1}|\psi_{(1)}|^2+(4a_{(2)}-2-p)r^{p-1}|\psi_{(2)}|^2 \\ 
\les&\int_{C_{u_1}^V}r^p|\psi_{(1)}|^2+\int_{V\cap\pr\C_\de(u_1,u_2)}|\psi_{(1)}|^2+|\psi_{(2)}|^2+\int_{V} r^p|\psi_{(1)}||h_{(1)}|+r^p|\psi_{(2)}||h_{(2)}|.
\end{align*}
Taking $\de=\de_n$ defined in Lemma \ref{axislimitR} and letting $n\to\infty$, we obtain
\begin{align*}
&\int_\cuv r^p|\psi_{(1)}|^2+\int_\ucuv r^p(|\psi_{(1)}|^2+|\psi_{(2)}|^2)\\
+&\int_{V} (2+p-4a_{(1)})r^{p-1}|\psi_{(1)}|^2+(4a_{(2)}-2-p)r^{p-1}|\psi_{(2)}|^2 \\ 
\les&\int_{C_{u_1}^V}r^p|\psi_{(1)}|^2+\int_{V} r^p|\psi_{(1)}||h_{(1)}|+r^p|\psi_{(2)}||h_{(2)}|.
\end{align*}
This concludes the proof of Proposition \ref{keyintegralint}.
\end{proof}
\subsubsection{Estimates for Bianchi pairs}
The following lemmas allows us to obtain $u$--decay of curvature in the interior region.
\begin{lem}\label{meanvalue}
Let $s>1$, $q=0,1$ and let $1\leq u_1\leq u_2\leq u_c(t)$. We assume that
\begin{align}\label{BEFp}
        pB_p^q[\Psi](u_1,u_2)+EF_p^q[\Psi](u_1,u_2)&\les E_p^q[\Psi](u_1)+\frac{\ep_0^2}{u^{s-p}},\qquad \forall p\in [0,s],
\end{align}
and
\begin{equation*}
    E_s^q[\Psi](1)\les\ep_0^2.
\end{equation*}
Then, we have
\begin{align*}
    EF_p^q[\Psi](u_1,u_2)\les\frac{\ep_0^2}{u_1^{s-p}},\qquad \forall p\in[0,s].
\end{align*}
\end{lem}
\begin{proof}
We first have for all $u\geq 1$
\begin{align*}
    BEF_s^q[\Psi](1,u)\les E_s^q[\Psi](1)\les\ep_0^2,
\end{align*}
which implies
\begin{equation}\label{BEFinitial}
    E_s^q[\Psi](u)\les \ep_0^2,\qquad \forall\; 1\leq u\leq u_c(t).
\end{equation}
Next, taking $u_1=2^j$ and $u_2=2^{j+1}$ for all $j\in\mathbb{N}$ in \eqref{BEFp}, we infer
\begin{equation*}
    B_s^q[\Psi](2^j,2^{j+1})\les E_s^q[\Psi](2^j)\les\ep_0^2.
\end{equation*}
By definition, we have
\begin{align*}
    \int_{2^j}^{2^{j+1}}E_{s-1}^q[\Psi](u) du\les\ep_0^2,
\end{align*}
which implies that there exists a sequence $u^{(j)}\in[2^j,2^{j+1}]$ such that
\begin{align*}
    E_{s-1}^q[\Psi](u^{(j)})\les\frac{\ep_0^2}{u^{(j)}}.
\end{align*}
Interpolating with \eqref{BEFinitial}, we obtain
\begin{align*}
    E_1^q[\Psi](u^{(j)})\les \frac{\ep_0^2}{\left(u^{(j)}\right)^{s-1}}.
\end{align*}
Thus, applying \eqref{BEFp} with $p=1$, $u_1=u^{(j-1)}$ and $u_2=u\in[2^j,2^{j+1}]$, we have
\begin{align*}
    E_1^q[\Psi](u)\les\frac{\ep_0^2}{\left(u^{(j-1)}\right)^{s-1}}\les\frac{\ep_0^2}{(2^j)^{s-1}}\les\frac{\ep_0^2}{u^{s-1}}.
\end{align*}
Next, we apply once again \eqref{BEFp} with $p=1$ to deduce
\begin{equation*}
    BEF_1^q[\Psi](u_1,u_2)\les\frac{\ep_0^2}{u_1^{s-1}}.
\end{equation*}
Then, we obtain for $j\in\mathbb{N}$
\begin{align*}
    B_1^q[\Psi](2^j,2^{j+1})\les \frac{\ep_0^2}{(2^j)^{s-1}},
\end{align*}
which implies
\begin{align*}
    \int_{2^j}^{2^{j+1}}E_0^q[\Psi](u)du\les\frac{\ep_0^2}{(2^j)^{s-1}}.
\end{align*}
Then, there exists $v^{(j)}\in[2^j,2^{j+1}]$ such that
\begin{align*}
    E_0^q[\Psi](v^{(j)})\les\frac{\ep_0^2}{(2^j)^{s}}\les\frac{\ep_0^2}{(v^{(j)})^s}.
\end{align*}
Thus, applying \eqref{BEFp} with $p=0$, $u_1=v^{(j-1)}$ and $u_2=u\in[2^j,2^{j+1}]$, we have
\begin{align*}
    E_0^q[\Psi](u)\les\frac{\ep_0^2}{\left(v^{(j-1)}\right)^{s}}\les\frac{\ep_0^2}{(2^j)^{s}}\les\frac{\ep_0^2}{u^{s}}.
\end{align*}
Next, we apply once again \eqref{BEFp} with $p=0$ to deduce
\begin{equation*}
    EF_0^q[\Psi](u_1,u_2)\les\frac{\ep_0^2}{u_1^s}.
\end{equation*}
Interpolating with \eqref{BEFinitial}, this concludes the proof of Lemma \ref{meanvalue}.
\end{proof}
\begin{lem}\label{uother}
Let $s>1$, $q=0,1$ and let $1\leq u_1\leq u_2\leq u_c(t)$. We assume that
    \begin{align}\label{BFp}
        BF_p^q[\Psi](u_1,u_2)\les\frac{\ep_0^2}{u_1^{s-p}},\qquad p\in[0,1],
    \end{align}
    and
    \begin{align}\label{E0}
        E_0^q[\Psi](u_2)\les E_0^q[\Psi](u_1)+\frac{\ep_0^2}{u_1^{s}}.
    \end{align}
    Then, we have
    \begin{align}\label{BEF0q}
        BEF_0^q[\Psi](u_1,u_2)\les\frac{\ep_0^2}{u_1^{s}}.
    \end{align}
\end{lem}
\begin{proof}
    Applying \eqref{BFp} with $p=1$, $u_1=2^j$ and $u_2=2^{j+1}$, we obtain
    \begin{align*}
        \int_{2^j}^{2^{j+1}}E_0[\Psi](u)du\les\frac{\ep_0^2}{(2^j)^{s-1}},
    \end{align*}
    which implies that there exists a sequence $u^{(j)}\in[2^j,2^{j+1}]$ such that
    \begin{align*}
        E_0[\Psi](u^{(j)})\les\frac{\ep_0^2}{(2^j)^{s}}.
    \end{align*}
    Next, applying \eqref{E0} with $u_1=u^{(j-1)}$ and $u_2=u\in[2^{j},2^{j+1}]$, we infer
    \begin{align}\label{E0u}
        E_0^q[\Psi](u)\les\frac{\ep_0^2}{(2^j)^s}\les\frac{\ep_0^2}{u^s}.
    \end{align}
    Finally, applying once again \eqref{BFp} with $p=0$ and \eqref{E0} and combining with \eqref{E0u}, we obtain \eqref{BEF0q}. This concludes the proof of Lemma \ref{uother}. 
\end{proof}
\begin{rk}\label{wonder}
The nonlinear error terms $\ee_p^q[\psi_{(1)},h_{(1)}]$ and $\ee_p^q[\psi_{(2)},h_{(2)}]$ in the proof of Propositions \ref{linearcontrol} and \ref{esta4} below are the same as in Propositions \ref{estab}--\ref{estaaa}. Thus, we have from the second part of Theorem \ref{wonderfulrp}\footnote{We always have $p_1,p_2,p\geq 0$ in the proof of Propositions \ref{linearcontrol} and \ref{esta4}, so that the additional assumption of the second part of Theorem \ref{wonderfulrp} is satisfied.} 
\begin{align*}
    \ee_p^q[\psi_{(1)},h_{(1)}]\les\frac{\ep_0^2}{u_1^{s-p}},\qquad \ee_p^q[\psi_{(2)},h_{(2)}]\les\frac{\ep_0^2}{u_1^{s-p}}.
\end{align*}
In the proof of Propositions \ref{linearcontrol} and \ref{esta4} below, we thus focus on the linear terms.
\end{rk}
\begin{prop}\label{linearcontrol}
We have the following estimates for all $0\leq p\leq s$:
    \begin{align*}
        EF_p^1[\a](u_1,u_2)+pB_p^1[\a](u_1,u_2)+BF_p^1[\b](u_1,u_2)&\les E_p^1[\a](u_1)+\frac{\ep_0^2}{u_1^{s-p}},\\
         BEF_p^1[\b](u_1,u_2)+BF_p^1[\rhoc,\sic](u_1,u_2)&\les E_p^1[\a,\b](u_1)+\frac{\ep_0^2}{u_1^{s-p}},\\
         BEF^1_p[\rhoc,\sic](u_1,u_2)+(2-p)B_p^1[\bb](u_1,u_2)+F^1_p[\bb](u_1,u_2)&\les E_p^1[\a,\b,\rhoc,\sic](u_1)+\frac{\ep_0^2}{u_1^{s-p}},\\
        BEF_0^1[\bb](u_1,u_2)+F_0^1[\aa](u_1,u_2)&\les E_0^1[\a,\b,\rhoc,\sic,\bb](u_1)+\frac{\ep_0^2}{u_1^{s}}.
    \end{align*}
\end{prop}
\begin{proof}
The proof is largely analogous to Propositions \ref{estab}--\ref{estaaa}. So we only provide a sketch. Throughout the proof, we always assume that $0\leq p\leq s$. \\ \\
Applying Proposition \ref{keyintegralint} with $\psi_{(1)}=\dko\a$, $\psi_{(2)}=\dko\b$, $a_{(1)}=\frac{1}{2}$, $a_{(2)}=2$, we obtain
\begin{align}\label{EFa}
EF^1_p[\a](u_1,u_2)+pB_p^1[\a](u_1,u_2)+BF^1_p[\b](u_1,u_2)\les E_p^1[\a](u_1)+\frac{\ep_0^2}{u_1^{s-p}}.
\end{align}
Next, applying Proposition \ref{keyintegralint} with $\psi_{(1)}=\dkb\b$, $\psi_{(2)}=(\dkb\rho,-\dkb\si)$, $a_{(1)}=1$, $a_{(2)}=\frac{3}{2}$, we deduce
\begin{align}\label{EFb}
EF^1_p[\b](u_1,u_2)+BF^1_p[\rho,\si](u_1,u_2)\les E_p^1[\b](u_1)+B_p^1[\b](u_1,u_2)+\frac{\ep_0^2}{u_1^{s-p}}.
\end{align}
Combining with \eqref{EFa} and \eqref{renorr}, we obtain
\begin{align}\label{BEFbeta}
BEF^1_p[\b](u_1,u_2)+BF^1_p[\rhoc,\sic](u_1,u_2)\les E_p^1[\a,\b](u_1)+\frac{\ep_0^2}{u_1^{s-p}}.
\end{align}
Next, applying Proposition \ref{keyintegralint} with $\psi_{(1)}=(\dkb\rhoc,\dkb\sic)$, $\psi_{(2)}=\dkb\bb$, $a_{(1)}=\frac{3}{2}$, $a_{(2)}=1$, we infer
\begin{align}\label{EFrho}
\begin{split}
    &EF^1_p[\rhoc,\sic](u_1,u_2)+(2-p)B_p^1[\bb](u_1,u_2)+F^1_p[\bb](u_1,u_2)\\
    \les&\; E_p^1[\rhoc,\sic](u_1)+B_p^1[\rhoc,\sic](u_1,u_2)+\frac{\ep_0^2}{u_1^{s-p}}.
\end{split}
\end{align}
Combining with \eqref{BEFbeta}, we have
\begin{align}\label{BEFrho}
\begin{split}
    &BEF^1_p[\rhoc,\sic](u_1,u_2)+(2-p)B_p^1[\bb](u_1,u_2)+F^1_p[\bb](u_1,u_2)\\
    \les&\; E_p^1[\a,\b,\rhoc,\sic](u_1)+\frac{\ep_0^2}{u_1^{s-p}}.    
\end{split}
\end{align}
Finally, applying Proposition \ref{keyintegralint} with $\psi_{(1)}=\dko\bb$, $\psi_{(2)}=\dko\aa$, $a_{(1)}=2$, $a_{(2)}=\frac{1}{2}$ and $p=0$, we deduce
\begin{align*}
    EF_0^1[\bb](u_1,u_2)+F_0^1[\aa](u_1,u_2)\les E_0^1[\bb](u_1)+B_0^1[\bb](u_1,u_2)+\frac{\ep_0^2}{u_1^{s}}.
\end{align*}
Combining with \eqref{BEFrho} in the case $p=0$, we obtain
\begin{align*}
    BEF_0^1[\bb](u_1,u_2)+F_0^1[\aa](u_1,u_2)\les E_0^1[\a,\b,\rhoc,\sic,\bb](u_1)+\frac{\ep_0^2}{u_1^{s}}.
\end{align*}
This concludes the proof of Proposition \ref{linearcontrol}.
\end{proof}
\begin{prop}\label{linearBEF}
We have the following estimates for all $0\leq p\leq s$:
\begin{align}
\begin{split}\label{linearBEFeq}
    pB_p^1[\a](u_1,u_2)+B_p^1[\b,\rhoc,\sic,\bb](u_1,u_2)+EF_p^1[\a,\b,\rhoc,\sic](u_1,u_2)+F^1_p[\bb](u_1,u_2)&\les\frac{\ep_0^2}{u_1^{s-p}},\\
    E_0^1[\bb](u_1,u_2)+F_0^1[\aa](u_1,u_2)&\les\frac{\ep_0^2}{u_1^s}.
\end{split}
\end{align}
\end{prop}
\begin{proof}
Recall that we have from Propositions \ref{estab}--\ref{estaaa}
\begin{align}\label{E1control}
    E_s^1[\a,\b,\rhoc,\sic](1)+E_0^1[\bb](1)\les\ep_0^2.
\end{align}
Applying Proposition \ref{linearcontrol} in the case $p=s$, we obtain immediately the first estimate of \eqref{linearBEFeq} in the case $p=s$. We then focus on the case $p=0$.\\ \\
We have from Proposition \ref{linearcontrol}
\begin{align}\label{BEFalpha}
pB_p^1[\a](u_1,u_2)+EF_p^1[\a](u_1,u_2)\les E_p^1[\a](u_1)+\frac{\ep_0^2}{u_1^{s-p}}.
\end{align}
Applying Lemma \ref{meanvalue} with $\Psi=\a$, we obtain
\begin{align}\label{aest}
    EF_p^1[\a](u_1,u_2)\les\frac{\ep_0^2}{u_1^{s-p}},\qquad \forall p\in[0,s].
\end{align}
Next, we have from Proposition \ref{linearcontrol} and \eqref{aest}
\begin{align*}
    BF_p^1[\b](u_1,u_2)\les E_p^1[\a](u_1)+\frac{\ep_0^2}{u_1^{s-p}}\les\frac{\ep_0^2}{u_1^{s-p}},\qquad\forall\; p\in[0,s],
\end{align*}
and
\begin{align*}
E_p^1[\b](u_1,u_2)&\les E_p^1[\a,\b](u_1)+\frac{\ep_0^2}{u_1^{s-p}}\les E_p^1[\b](u_1)+\frac{\ep_0^2}{u_1^{s-p}},\qquad\forall\; p\in[0,s].
\end{align*}
Applying Lemma \ref{uother} with $\Psi=\b$, we obtain
\begin{align*}
    BEF_0^1[\b](u_1,u_2)\les\frac{\ep_0^2}{u_1^{s}}.
\end{align*}
We also have from Proposition \ref{linearcontrol} and \eqref{E1control}
\begin{align*}
    BEF_s^1[\b](u_1,u_2)\les E_s^1[\a,\b](u_1)+\ep_0^2\les E_s^1[\a,\b](1)+\ep_0^2\les\ep_0^2.
\end{align*}
Combining the above estimates, we infer
\begin{align}\label{best}
    BEF_p^1[\b](u_1,u_2)\les\frac{\ep_0^2}{u_1^{s-p}},\qquad \forall p\in[0,s].
\end{align}
Then, we have from Proposition \ref{linearcontrol}, \eqref{aest} and \eqref{best}
\begin{align*}
    BF_p^1[\rhoc,\sic](u_1,u_2)\les E_p^1[\a,\b](u_1)+\frac{\ep_0^2}{u_1^{s-p}}\les\frac{\ep_0^2}{u_1^{s-p}},\qquad\forall\; p\in[0,s],
\end{align*}
and
\begin{align*}
E_p^1[\rhoc,\sic](u_1,u_2)\les E_p^1[\a,\b,\rhoc,\sic](u_1)+\frac{\ep_0^2}{u_1^{s-p}}\les E_p^1[\rhoc,\sic](u_1)+\frac{\ep_0^2}{u_1^{s-p}},\qquad\forall\; p\in[0,s].
\end{align*}
Applying Lemma \ref{uother}, we obtain
\begin{align*}
    BEF_0^1[\rhoc,\sic](u_1,u_2)\les\frac{\ep_0^2}{u_1^{s}}.
\end{align*}
We also have from Proposition \ref{linearcontrol} and \eqref{E1control}
\begin{align*}
    BEF_s^1[\rhoc,\sic](u_1,u_2)\les\ep_0^2.
\end{align*}
Combining the above estimates, we deduce
\begin{align}\label{rest}
    BEF_p^1[\rhoc,\sic](u_1,u_2)\les\frac{\ep_0^2}{u_1^{s-p}},\qquad \forall p\in[0,s].
\end{align}
Next, we have from Proposition \ref{linearcontrol} and \eqref{aest}--\eqref{rest}
\begin{align*}
BF_p^1[\bb](u_1,u_2)\les E_p^1[\a,\b,\rhoc,\sic](u_1)+\frac{\ep_0^2}{u_1^{s-p}}\les\frac{\ep_0^2}{u_1^{s-p}},\qquad p\in[0,s],
\end{align*}
and
\begin{align*}
E_0^1[\bb](u_1,u_2)\les E_0^1[\a,\b,\rhoc,\sic,\bb](u_1)+\frac{\ep_0^2}{u_1^{s}}\les E_0^1[\bb](u_1)+\frac{\ep_0^2}{u_1^{s}}.
\end{align*}
Applying Lemma \ref{uother}, we deduce
\begin{align}\label{bbest}
    BEF_0^1[\bb](u_1,u_2)\les\frac{\ep_0^2}{u_1^s}.
\end{align}
Finally, we have from Proposition \ref{linearcontrol} and \eqref{aest}--\eqref{bbest}
\begin{align}
    F_0^1[\aa](u_1,u_2)\les E_0^1[\a,\b,\rhoc,\sic,\bb](u_1)+\frac{\ep_0^2}{u_1^s}\les\frac{\ep_0^2}{u_1^s}.
\end{align}
Combining the above estimates, this concludes the proof of Proposition \ref{linearBEF}.
\end{proof}
\begin{prop}\label{esta4}
We have the following estimate:
\begin{equation*}
    EF_p[\a_4](u_1,u_2)\les\frac{\ep_0^2}{u_1^{s-p}},\qquad\forall\; p\in[0,s].
\end{equation*}
\end{prop}
\begin{proof}
We have from \eqref{teu}
\begin{align*}
\nabs_3\a_4&=-2\sld_2^*\as+\frac{4\a}{r}+(\O_s\cdot\F_s)^{(1)},\\
\nabs_4\as+\frac{5}{2}\trch\,\as&=\sld_2\ac+(\O_s\cdot\F_s)^{(1)}.
\end{align*}
Applying Proposition \ref{keyintegralint} with $\psi_{(1)}=\a_4$, $\psi_{(2)}=\as$, $a_{(1)}=\frac{1}{2}$, $a_{(2)}=2$ and $0\leq p\leq s$, we deduce from Remark \ref{wonder}
\begin{align*}
    BEF_p[\a_4](u_1,u_2)+BF_p^1[\a](u_1,u_2)\les E_p[\a_4](u_1)+\frac{\ep_0^2}{u_1^{s-p}}+\ee_{p-1}^0[\a_4,\a].
\end{align*}
Applying Cauchy-Schwarz inequality, we obtain for any $\de_0>0$
\begin{align*}
    \ee_{p-1}^0[\a_4,\a]\leq\de_0\int_V r^{p-1}|\a_4|^2+\frac{1}{4\de_0}\int_V r^{p-1}|\a|^2\leq \de_0 B_p[\a_4](u_1,u_2)+(4\de_0)^{-1}B_p[\a](u_1,u_2).
\end{align*}
Combining with Proposition \ref{linearcontrol}, we obtain for $\de_0$ small enough, 
\begin{align*}
    BEF_p[\a_4](u_1,u_2)+BF_p^1[\a](u_1,u_2)\les E_p[\a_4](u_1)+\frac{\ep_0^2}{u_1^{s-p}},\qquad\forall\; p\in[0,s].
\end{align*}
Applying Lemma \ref{meanvalue}, we deduce
\begin{align*}
    EF_p[\a_4](u)\les\frac{\ep_0^2}{u^{s-p}},\qquad \forall\; p\in[0,s].
\end{align*}
This concludes the proof of Proposition \ref{esta4}.
\end{proof}
\subsection{Time derivative estimates in the interior region}\label{ssecTime}
\subsubsection{Bianchi equations commuted with \texorpdfstring{$\nabs_T$}{}}
\begin{lem}\label{LieTGag}
We have
\begin{align}
    \begin{split}\label{firstLieT}
        \nabs_T\hch&=r^{-1}\O_1^{(1)},\qquad\quad\, \nabs_T\ze=r^{-1}\O_{s-1}^{(1)},\qquad\quad\,\nabs_T\hchb=r^{-1}\O_{-1}^{(1)},\\
        \nabs_T\trch&=r^{-1}\O_{1}^{(1)},\qquad \nabs_T\trchb=r^{-1}\O_{s-1}^{(1)},
    \end{split}
\end{align}
and
\begin{align}
    \begin{split}\label{secondLieT}
        \nabs_T\eta&\in r^{-1}\O_{-1}^{(1)},\qquad\,\, \nabs_T\omb\in r^{-1}\O_{s-1}^{(1)},\qquad \,\,\nabs_T\xib\in r^{-1}\O_{-1}^{(1)},\\ \nabs_T\etab&\in r^{-1}\O_{s-1}^{(1)},\qquad \nabs_T\om\in r^{-1}\O_{s-1}^{(1)}.
    \end{split}
\end{align}
Moreover, we have
\begin{align}
    \begin{split}\label{thirdLieT}
        \nabs_T\a&\in \F_{s+2}^{(1)}(s+2),\qquad\, \nabs_T\b\in\F^{(1)}_{2s-\db}(s+2),\\
        \nabs_T(\rho,\si)&\in\Fb^{(1)}_{s+1}(s+2),\qquad\; \nabs_T\bb\in\Fb_2^{(1)}(s+2),\qquad\; \nabs_T\aa\in\Fb_0^{(1)}(s+2),
    \end{split}
\end{align}
and
\begin{align}
    \begin{split}\label{fourthLieT}
        r^{-1}\a&\in \F_{s+2}(s+2),\qquad\; r^{-1}\b\in\F_{s+2}(s+2),\qquad\;
        r^{-1}(\rho,\si)\in\F_{s+2}(s+2),\\
        r^{-1}\bb&\in\Fb_{s+2}(s+2),\qquad\,\, r^{-1}\aa\in\Fb_2(s+2).
    \end{split}
\end{align}
\end{lem}
\begin{proof}
     It follows directly from \eqref{6.6}, \eqref{B1} and Corollaries \ref{LieTGa} and \ref{LieTR}.
\end{proof}
\begin{prop}\label{BianchiLieTeq}
We have the following equations:\footnote{Here, we use the equations for $\nabs_T(\rho,\si)$ instead of $\nabs_T(\rhoc,\sic)$ since the decay of nonlinear terms are sufficient.}
\begin{align*}
\nabs_3(\nabs_T\a)+\frac{1}{2}\trchb(\nabs_T\a)&=-2\sld_2^*(\nabs_T\b)+(\O_{-1}\c\F_{2s-\db}(s+2))^{(1)}\\
&+(\O_s\c\Fb_{s+1}(s+2))^{(1)},\\
\nabs_4(\nabs_T\b)+2\trch(\nabs_T\b)&=\sld_2(\nabs_T\a)+(\O_{s-1}\c\F_{s+2}(s+2))^{(1)},\\
\nabs_3(\nabs_T\b)+\trchb(\nabs_T\b)&=-\sld_1^*(\nabs_T(\rho,-\si))+(\O_{-1}\c\F_{s+2-\db}(s+2))^{(1)}\\
&+(\O_s\c\Fb_{2}(s+2))^{(1)},\\
\nabs_4(\nabs_T(\rho,-\si))+\frac{3}{2}\trch(\nabs_T(\rho,-\si))&=\sld_2(\nabs_T\b)+(\O_{-1}\c\F_{s+2}(s+2))^{(1)},\\
\nabs_3(\nabs_T(\rho,\si))+\frac{3}{2}\trchb(\nabs_T(\rho,\si))&=-\sld_1(\nabs_T\bb)+(\O_s\c\Fb_0(s+2))^{(1)},\\
\nabs_4(\nabs_T\bb)+\trch(\nabs_T\bb)&=\sld_1^*(\nabs_T\rho)+(\O_s\c\Fb_2(s+2))^{(1)}+(\O_{-1}\c\F_{s+2}(s+2))^{(1)},\\
\nabs_3(\nabs_T\bb)+2\trchb(\nabs_T\bb)&=-\sld_2(\nabs_T\aa)+(\O_1\c\Fb_0(s+2))^{(1)},\\
\nabs_4(\nabs_T\aa)+\frac{1}{2}\trch(\nabs_T\aa)&=2\sld_2^*(\nabs_T\bb)+(\O_s\c\Fb_0(s+2))^{(1)}.
\end{align*}
\end{prop}
\begin{proof}
It follows directly from Proposition \ref{Bianchieq}, Corollary \ref{Tcomm} and Lemma \ref{LieTGag}.
\end{proof}
\begin{thm}\label{wonderfulrpLieT}
Let $0\leq p_1,p_2,p\leq s+2$, $\ell\leq s$ and let $s>1$. Then, we have the following properties.
\begin{enumerate}
    \item In the case $2p<p_1+p_2+\ell+1$, we have
\begin{equation*}
    \ee^1_p\left[\F_{p_1}(s+2),\O_\ell\c\F_{p_2}(s+2)\right]\les\frac{\ep^3}{u_1^{s+2-p}}.
\end{equation*}
    \item In the case $2p<p_1+p_2+\ell$ and $2p+1<p_1+p_2+s$, we have
\begin{align*}
    \ee^1_p\left[\F_{p_1}(s+2),\O_\ell\c\Fb_{p_2}(s+2)\right]&\les\frac{\ep^3}{u_1^{s+2-p}},\\
    \ee^1_p\left[\Fb_{p_1}(s+2),\O_\ell\c\F_{p_2}(s+2)\right]&\les\frac{\ep^3}{u_1^{s+2-p}}.
\end{align*}
    \item In the case $2p<p_1+p_2+\ell-1$, we have
\begin{equation*}
\ee^1_p\left[\Fb_{p_1}(s+2),\O_\ell\c\Fb_{p_2}(s+2)\right]\les\frac{\ep^3}{u_1^{s+2-p}}.
\end{equation*}
\end{enumerate}
\end{thm}
\begin{proof}
The proof is largely analogous to Theorem \ref{wonderfulrp} and left to the reader.
\end{proof}
\begin{prop}\label{keyintegralintT}
Let $\psi_{(1)}$ and $\psi_{(2)}$ be time derivative of curvature components, i.e.
\begin{align*}
    \psi_{(1)},\psi_{(2)}\in \{\nabs_T\a,\nabs_T\b,\nabs_T\rho,\nabs_T\si,\nabs_T\bb,\nabs_T\aa\}.
\end{align*}
We also assume that $\psi_{(1)}$ and $\psi_{(2)}$ satisfy \eqref{bianchi1} or \eqref{bianchi2}. Then, we have for $p\geq 0$
\begin{align*}
&EF_p[\psi_{(1)}](u_2)+F_p[\psi_{(2)}](u_1,u_2)\\
+&(2+p-4a_{(1)})B_p[\psi_{(1)}](u_1,u_2)+(4a_{(2)}-2-p)B_p[\psi_{(2)}](u_1,u_2)\\
\les&\;E_p[\psi_{(1)}](u_1)+\ee_p^0[\psi_{(1)},h_{(1)}]+\ee_p^0[\psi_{(2)},h_{(2)}].
\end{align*}
\end{prop}
\begin{proof}
Applying Theorem 7.1 in \cite{Maxwell}, we construct a sequence of initial data $(\Si_0,g_n,k_n)$ which approximate $(\Si_0,g,k)$ in $H^6(\Si_0)$. By local existence and Sobolev inequality, we obtain a sequence of $C^4$--solutions $(\M_n,\g_n)$ of Einstein vacuum equations with initial data $(\Si_0,g_n,k_n)$ which approximate $(\M,\g)$. Denoting $(\D_n,\R_n)$ the Levi-Civita connection and curvature components of $\g_n$, we have near the symmetry axis $\Vphi$
\begin{align*}
    E_n,H_n=O(1),\qquad (\D_n)_{T_n}(E_n,H_n)=O(1).
\end{align*}
Recalling that $(e_4)_n=T_n+N_n$, $(e_3)_n=T_n-N_n$ and $\|T_n\|,\|N_n\|=1$, we deduce on $\pr\C_\de$
\begin{align*}
    (\nabs_n)_{T_n}R_n=O(1),\qquad \mbox{ where }\; R\in\{\a,\b,\rho,\si,\bb,\aa\}.
\end{align*}
Integrating \eqref{div} or \eqref{div2} in $V\setminus\C_{\de}$\footnote{Recall that $\C_\de$ is defined in Lemma \ref{axislimitR}.} and proceeding as in Proposition \ref{keyintegralint}, we obtain\footnote{Note that in \eqref{cylinderlimit}--\eqref{finallimit}, the constants involved in $\les$ are independent of $n$.}
\begin{align}
\begin{split}\label{cylinderlimit}
&\int_{C_{u_2}^V\setminus\C_\de}r^p |(\psi_n)_{(1)}|^2+\int_{\ucuv\setminus\C_\de} r^p(|(\psi_n)_{(1)}|^2+|(\psi_n)_{(2)}|^2)\\
+&\int_{V} (2+p-4a_{(1)})r^{p-1}|(\psi_n)_{(1)}|^2+(4a_{(2)}-2-p)r^{p-1}|(\psi_n)_{(2)}|^2 \\ 
\les&\int_{C_{u_1}^V}r^p|(\psi_n)_{(1)}|^2+\int_{V\cap\pr\C_\de(u_1,u_2)}|(\psi_n)_{(1)}|^2+|(\psi_n)_{(2)}|^2\\
+&\int_{V} r^p|(\psi_n)_{(1)}||(h_n)_{(1)}|+r^p|(\psi_n)_{(2)}||(h_n)_{(2)}|,
\end{split}
\end{align}
where we used $r^p\leq 1$ on $\pr\C_\de$ since $p\geq 0$. Note that we have
\begin{equation}\label{cde0}
\lim_{\de\to 0}\int_{V\cap\pr\C_\de(u_1,u_2)}|(\psi_n)_{(1)}|^2+|(\psi_n)_{(2)}|^2\les C(n) t\lim_{\de\to 0} \de^2 =0,
\end{equation}
where $C(n)$ is a constant which depends on $n$. Taking $\de\to 0$ in \eqref{cylinderlimit}, we obtain
\begin{align}
\begin{split}\label{finallimit}
&\int_{C_{u_2}^V}r^p |(\psi_n)_{(1)}|^2+\int_{\ucuv} r^p(|(\psi_n)_{(1)}|^2+|(\psi_n)_{(2)}|^2)\\
+&\int_{V} (2+p-4a_{(1)})r^{p-1}|(\psi_n)_{(1)}|^2+(4a_{(2)}-2-p)r^{p-1}|(\psi_n)_{(2)}|^2 \\ 
\les&\int_{C_{u_1}^V}r^p|(\psi_n)_{(1)}|^2+\int_{V} r^p|(\psi_n)_{(1)}||(h_n)_{(1)}|+r^p|(\psi_n)_{(2)}||(h_n)_{(2)}|.
\end{split}
\end{align}
Letting $n\to +\infty$ in \eqref{finallimit}, this concludes the proof of Proposition \ref{keyintegralintT}.
\end{proof}
\subsubsection{Estimates for the Bianchi pair \texorpdfstring{$(\nabs_T\a,\nabs_T\b)$}{}}\label{sssec9.1}
\begin{prop}\label{estTab}
We have the following estimate for all $0\leq p\leq s+1$:
\begin{equation}
pB_p[\nabs_T\a](u_1,u_2)+EF_{p}[\nabs_T\a](u_1,u_2)+BF_{p}[\nabs_T\b](u_1,u_2)\les E_{p}[\nabs_T\a](u_1)+\frac{\ep_0^2}{u_1^{s+2-p}}.
\end{equation}
\end{prop}
\begin{proof}
We recall from Proposition \ref{BianchiLieTeq} that
\begin{align*}
\nabs_3(\nabs_T\a)+\frac{1}{2}\trchb(\nabs_T\a)&=-2\sld_2^*(\nabs_T\b)+(\O_{-1}\c\F_{2s-\db}(s+2))^{(1)}+(\O_s\c\Fb_{s+1}(s+2))^{(1)},\\
\nabs_4(\nabs_T\b)+2\trch(\nabs_T\b)&=\sld_2(\nabs_T\a)+(\O_{s-1}\c\F_{s+2}(s+2))^{(1)}.
\end{align*}
Applying Proposition \ref{keyintegralint} with $\psi_{(1)}=\nabs_T\a$, $\psi_{(2)}=\nabs_T\b$, $a_{(1)}=\frac{1}{2}$, $a_{(2)}=2$, $h_{(1)}=(\O_{-1}\c\F_{2s-\db}(s+2))^{(1)}+(\O_s\c\Fb_{s+1}(s+2))^{(1)}$, $h_{(2)}=(\O_{s-1}\c\F_{s+2}(s+2))^{(1)}$ and $0\leq p\leq s+1$, we obtain from \eqref{deferr} and Remark \ref{wonderfulrk}
\begin{align*}
&pB_p[\nabs_T\a](u_1,u_2)+EF_p[\nabs_T\a](u_1,u_2)+BF_{p}[\nabs_T\b](u_1,u_2)\\
\les& E_{p}[\nabs_T\a](u_1)+\ee_p^1[\F_{s+2}(s+2),\O_{-1}\c\F_{2s-\db}(s+2)]\\
+&\ee_p^1[\F_{s+2}(s+2),\O_s\c\Fb_{s+1}(s+2)]+\ee_p^1[\F_{2s-\db}(s+2),\O_{s-1}\c\F_{s+2}(s+2)].
\end{align*}
Applying Theorem \ref{wonderfulrpLieT}, we obtain for $0\leq p\leq s+1$
\begin{align*}
&\ee_p^1[\F_{s+2}(s+2),\O_{-1}\c\F_{2s-\db}(s+2)]+\ee_p^1[\F_{s+2}(s+2),\O_s\c\Fb_{s+1}(s+2)]\\
+&\ee_p^1[\F_{2s-\db}(s+2),\O_{s-1}\c\F_{s+2}(s+2)]\les\frac{\ep_0^2}{u_1^{s+2-p}}.
\end{align*}
Hence, we have for $0\leq p\leq s+1$
\begin{equation*}
pB_p[\nabs_T\a](u_1,u_2)+EF_{p}[\nabs_T\a](u_1,u_2)+BF_{p}[\nabs_T\b](u_1,u_2)\les E_p[\nabs_T\a](u_1)+\frac{\ep_0^2}{u_1^{s+2-p}}.
\end{equation*}
This concludes the proof of Proposition \ref{estTab}.
\end{proof}
\subsubsection{Estimates for the Bianchi pair \texorpdfstring{$(\nabs_T\b,\nabs_T(\rho,-\si))$}{}}\label{sssec9.2}
\begin{prop}\label{estTbr}
We have the following estimate for $0\leq p\leq 2$:
\begin{equation}\label{eqTbr}
BEF_p[\nabs_T\b](u_1,u_2)+BF_p[\nabs_T\rho,\nabs_T\si](u_1,u_2)\les E_p[\nabs_T\a,\nabs_T\b](u_1)+\frac{\ep_0^2}{u_1^{s+2-p}}.
\end{equation}
\end{prop}
\begin{proof}
We recall from Proposition \ref{BianchiLieTeq}
\begin{align*}
\nabs_3(\nabs_T\b)+\trchb(\nabs_T\b)&=-\sld_1^*(\nabs_T(\rho,-\si))+(\O_{-1}\c\F_{s+2-\db}(s+2))^{(1)}\\
&+(\O_s\c\Fb_{2}(s+2))^{(1)},\\
\nabs_4(\nabs_T(\rho,-\si))+\frac{3}{2}\trch(\nabs_T(\rho,-\si))&=\sld_2(\nabs_T\b)+(\O_{-1}\c\F_{s+2}(s+2))^{(1)}.
\end{align*}
Applying Proposition \ref{keyintegralint} with $\psi_{(1)}=\nabs_T\b$, $\psi_{(2)}=\nabs_T(\rho,-\si)$, $a_{(1)}=1$, $a_{(2)}=\frac{3}{2}$, $h_{(1)}=(\O_{-1}\c\F_{s+2-\db}(s+2))^{(1)}+(\O_s\c\Fb_{2}(s+2))^{(1)}$, $h_{(2)}=(\O_{-1}\c\F_{s+2}(s+2))^{(1)}$ and $0\leq p\leq 2$ and combining with Proposition \ref{estTab}, we obtain
\begin{align*}
&BEF_p[\nabs_T\b](u_1,u_2)+BF_p[\nabs_T\rho,\nabs_T\si](u_1,u_2)\\
\les&\;E_p[\nabs_T\b](u_1)+B_p[\nabs_T\b](u_1,u_2)+\ee_p^1[\F_{2s-\db}(s+2),\O_{-1}\c\F_{s+2-\db}(s+2)]\\
+&\;\ee_p^1[\F_{2s-\db}(s+2),\O_{s}\c\Fb_{2}(s+2)]+\ee_p^1[\Fb_{s+1}(s+2),\O_{-1}\c\F_{s+2}(s+2)]\\
\les&\;E_p[\nabs_T\a,\nabs_T\b](u_1)+\frac{\ep_0^2}{u_1^{s+2-p}}+\ee_p^1[\F_{2s-\db}(s+2),\O_{-1}\c\F_{s+2-\db}(s+2)]\\
+&\;\ee_p^1[\F_{2s-\db}(s+2),\O_{s}\c\Fb_{2}(s+2)]+\ee_p^1[\Fb_{s+1}(s+2),\O_{-1}\c\F_{s+2}(s+2)].
\end{align*}
Applying Theorem \ref{wonderfulrpLieT} in corresponding cases, we infer for $0\leq p\leq 2$
\begin{align*}
&\ee_p^1[\F_{2s-\db}(s+2),\O_{-1}\c\F_{s+2-\db}(s+2)]+\ee_p^1[\F_{2s-\db}(s+2),\O_{s}\c\Fb_{2}(s+2)]\\
+&\ee_p^1[\Fb_{s+1}(s+2),\O_{-1}\c\F_{s+2}(s+2)]\les\frac{\ep_0^2}{u_1^{s+2-p}}.
\end{align*}
Hence, we deduce for $0\leq p\leq 2$
\begin{equation*}
BEF_p[\nabs_T\b](u_1,u_2)+BF_p[\nabs_T\rho,\nabs_T\si](u_1,u_2)\les E_p[\nabs_T\a,\nabs_T\b](u_1)+\frac{\ep_0^2}{u_1^{s+2-p}}.
\end{equation*}
This concludes the proof of Proposition \ref{estTbr}.
\end{proof}
\subsubsection{Estimates for the Bianchi pair \texorpdfstring{$(\nabs_T(\rho,\si),\nabs_T\bb)$}{}}\label{sssec9.3}
\begin{prop}\label{estTrb}
We have the following estimate for $0\leq p\leq 1$:
\begin{align}
\begin{split}
& BEF_p[\nabs_T\rho,\nabs_T\si](u_1,u_2)+BF_p[\nabs_T\bb](u_1,u_2)\\
\les\;& E_p[\nabs_T\a,\nabs_T\b,\nabs_T\rho,\nabs_T\si](u_1)+\frac{\ep_0^2}{u_1^{s+2-p}}.\label{Trb}
\end{split}
\end{align}
\end{prop}
\begin{proof}
We recall from Proposition \ref{BianchiLieTeq}
\begin{align*}
\nabs_3(\nabs_T(\rho,\si))+\frac{3}{2}\trchb(\nabs_T(\rho,\si))&=-\sld_1(\nabs_T\bb)+(\O_s\c\Fb_0(s+2))^{(1)},\\
\nabs_4(\nabs_T\bb)+\trch(\nabs_T\bb)&=\sld_1^*(\nabs_T\rho)+(\O_s\c\Fb_2(s+2))^{(1)}+(\O_{-1}\c\F_{s+2}(s+2))^{(1)}.
\end{align*}
Applying Proposition \ref{keyintegralint} with $\psi_{(1)}=\nabs_T(\rho,\si)$, $\psi_{(2)}=\nabs_T\bb$, $a_{(1)}=1$, $a_{(2)}=\frac{3}{2}$, $h_{(1)}=(\O_s\c\Fb_0(s+2))^{(1)}$, $h_{(2)}=(\O_s\c\Fb_2(s+2))^{(1)}+(\O_{-1}\c\F_{s+2}(s+2))^{(1)}$ and $0\leq p\leq 1$ and combining with Proposition \ref{estTbr}, we obtain
\begin{align*}
&BEF_p[\nabs_T\rho,\nabs_T\si](u_1,u_2)+BF_p[\nabs_T\bb](u_1,u_2)\\
\les &\;E_p[\nabs_T\a,\nabs_T\b,\nabs_T\rho,\nabs_T\si](u_1)+\frac{\ep_0^2}{u_1^{s+2-p}}+\ee_p^1[\Fb_{s+1}(s+2),\O_{s}\c\Fb_{0}(s+2)]\\
+&\;\ee_p^1[\Fb_{2}(s+2),\O_{s}\c\Fb_2(s+2)]+\ee_p^1[\Fb_2(s+2),\O_{-1}\c\F_{s+2}(s+2)].
\end{align*}
Applying Theorem \ref{wonderfulrpLieT} in corresponding cases, we deduce for $0\leq p\leq 1$
\begin{align*}
&\ee_p^1[\Fb_{s+1}(s+2),\O_{s}\c\Fb_{0}(s+2)]+\ee_p^1[\Fb_{2}(s+2),\O_{s}\c\Fb_{2}(s+2)]\\
+&\ee_p^1[\Fb_{2}(s+2),\O_{-1}\c\F_{s+2}(s+2)]\les\frac{\ep_0^2}{u_1^{s+2-p}}.
\end{align*}
Thus, we obtain for $0\leq p\leq 1$
\begin{align*}
BEF_p[\nabs_T\rho,\nabs_T\si](u_1,u_2)+BF_p[\nabs_T\bb](u_1,u_2)
\les E_p[\nabs_T\a,\nabs_T\b,\nabs_T\rho,\nabs_T\si](u_1)+\frac{\ep_0^2}{u_1^{s+2-p}}.
\end{align*}
This concludes the proof of Proposition \ref{estTrb}.
\end{proof}
\subsubsection{Estimates for the Bianchi pair \texorpdfstring{$(\nabs_T\bb,\nabs_T\aa)$}{}}\label{sssec9.4}
\begin{prop}\label{estTba}
We have the following estimate:
\begin{equation}\label{Tbbaa}
BEF_0[\nabs_T\bb](u_1,u_2)+F_0[\nabs_T\aa](u_1,u_2)\les E_0[\nabs_T\a,\nabs_T\b,\nabs_T\rho,\nabs_T\si,\nabs_T\bb](u_1)+\frac{\ep_0^2}{u_1^{s+2}}.
\end{equation}
\end{prop}
\begin{proof}
We have from Proposition \ref{BianchiLieTeq}
\begin{align*}
\nabs_3(\nabs_T\bb)+2\trchb(\nabs_T\bb)&=-\sld_2(\nabs_T\aa)+(\O_1\c\Fb_0(s+2))^{(1)},\\
\nabs_4(\nabs_T\aa)+\frac{1}{2}\trch(\nabs_T\aa)&=2\sld_2^*(\nabs_T\bb)+(\O_s\c\Fb_0(s+2))^{(1)}.
\end{align*}
Applying Proposition \ref{keyintegralint} with $\psi_{(1)}=\nabs_T\bb$, $\psi_{(2)}=\nabs_T\aa$, $a_{(1)}=2$, $a_{(2)}=\frac{1}{2}$, $h_{(1)}=(\O_1\c\Fb_0(s+2))^{(1)}$, $h_{(2)}=(\O_s\c\Fb_0(s+2))^{(1)}$ and $p=0$ and combining with Proposition \ref{estTrb}
\begin{align*}
& BEF_0[\nabs_T\bb](u_1,u_2)+F_0[\nabs_T\aa](u_1,u_2)\\
\les&\; E_0[\nabs_T\a,\nabs_T\b,\nabs_T\rho,\nabs_T\si,\nabs_T\bb](u_1)+\frac{\ep_0^2}{u_1^{s+2}}\\
+&\;\ee_0^1[\Fb_{2}(s+2),\O_1\c\Fb_0(s+2)]+\ee^1_0[\Fb_0(s+2),\O_s\c\Fb_0(s+2)].
\end{align*}
Applying Theorem \ref{wonderfulrpLieT} in corresponding cases, we obtain
\begin{equation*}
    \ee_0^1[\Fb_{2}(s+2),\O_1\c\Fb_0(s+2)]+\ee^1_0[\Fb_0(s+2),\O_s\c\Fb_0(s+2)]\les\frac{\ep_0^2}{u_1^{s+2}}.
\end{equation*}
Hence, we infer
\begin{align*}
BEF_0[\nabs_T\bb](u_1,u_2)+F_0[\nabs_T\aa](u_1,u_2)\les E_0[\nabs_T\a,\nabs_T\b,\nabs_T\rho,\nabs_T\si,\nabs_T\bb](u_1)+\frac{\ep_0^2}{u_1^{s+2}}.
\end{align*}
This concludes the proof of Proposition \ref{estTba}.
\end{proof}
\subsubsection{\texorpdfstring{$u$}{}--decay estimates}
Combining Propositions \ref{estTab}--\ref{estTba}, we have the following proposition.
\begin{prop}\label{linearcontrolLieT}
We have the following estimates:
\begin{itemize}
    \item We have for $0\leq p\leq s+1$
\begin{align}
\begin{split}\label{Talpha}
&pB_p[\nabs_T\a](u_1,u_2)+EF_p[\nabs_T\a](u_1,u_2)+BF_p[\nabs_T\b](u_1,u_2)
\\
\les\; & E_p[\nabs_T\a](u_1)+\frac{\ep_0^2}{u_1^{s+2-p}}.
\end{split}
\end{align}
\item We have for $0\leq p\leq 2$
\begin{align}
\begin{split}\label{Tbeta}
&BEF_p[\nabs_T\b](u_1,u_2)+BF_p[\nabs_T\rhoc,\nabs_T\sic](u_1,u_2)\\
\les\; &E_p[\nabs_T\a,\nabs_T\b](u_1)+\frac{\ep_0^2}{u_1^{s+2-p}}.
\end{split}
\end{align}
\item We have for $0\leq p\leq 1$
\begin{align}
\begin{split}\label{Trho}
&BEF_p[\nabs_T\rho,\nabs_T\si](u_1,u_2)+BF_p[\nabs_T\bb](u_1,u_2)\\
\les\; &E_p[\nabs_T\a,\nabs_T\b,\nabs_T\rho,\nabs_T\si](u_1)+\frac{\ep_0^2}{u_1^{s+2-p}}.
\end{split}
\end{align}
\item We have
\begin{align}
    \begin{split}\label{Tbb}
 &BEF_0[\nabs_T\bb](u_1,u_2)+F_0[\nabs_T\aa](u_1,u_2)\\
 \les\; &E_0[\nabs_T\a,\nabs_T\b,\nabs_T\rhoc,\nabs_T\sic,\nabs_T\bb](u_1)+\frac{\ep_0^2}{u_1^{s+2}}.
    \end{split}
\end{align}
\end{itemize}
\end{prop}
\begin{prop}\label{udecayLieT}
We have the following estimates:
    \begin{align*}
    EF_0[\nabs_T\a,\nabs_T\b,\nabs_T\rho,\nabs_T\si,\nabs_T\bb](u_1,u_2)+F_0[\nabs_T\aa](u_1,u_2)\les\frac{\ep_0^2}{u_1^{s+2}}.
    \end{align*}
\end{prop}
\begin{proof}
The proof is largely analogous to Proposition \ref{linearBEF}. So, we only provide a sketch.\\ \\
Recalling that
\begin{align*}
    \nabs_T\a=r^{-1}O(\a_4,\a^{(1)},\b^{(1)}), 
\end{align*}
we have from Proposition \ref{linearcontrol}
    \begin{equation}\label{BEFs+2}
        BEF_{s+2}[\nabs_T\a](u_1,u_2)\les BEF_s[\a_4](u_1,u_2)+BEF_s^1[\a,\b](u_1,u_2)\les\ep_0^2.
    \end{equation}
Taking $u_1=2^j$ and $u_2=2^{j+1}$, we obtain
\begin{align*}
    \int_{2^j}^{2^{j+1}} E_{s+1}[\nabs_T\a](u)du\les \ep_0^2,
\end{align*}
which implies that there exists a sequence $u^{(j)}\in[2^j,2^{j+1}]$ such that
\begin{align*}
    E_{s+1}[\nabs_T\a](u^{(j)})\les\frac{\ep_0^2}{u^{(j)}}.
\end{align*}
Applying \eqref{Talpha} with $p=s+1$, $u_1=u^{(j-1)}$ and $u_2=u\in[2^j,2^{j+1}]$, we deduce
\begin{align*}
    BEF_{s+1}[\nabs_T\a](u^{(j-1)},u)\les E_{s+1}[\nabs_T\a](u^{(j-1)})+\frac{\ep_0^2}{u^{(j-1)}}\les\frac{\ep_0^2}{u},
\end{align*}
which implies 
\begin{align*}
    E_{s+1}[\nabs_T\a](u)\les\frac{\ep_0^2}{u}.
\end{align*}
Applying once again \eqref{Talpha} with $p=s+1$, we deduce
\begin{align*}
    BEF_{s+1}[\nabs_T\a](u_1,u_2)\les\frac{\ep_0^2}{u_1}.
\end{align*}
Next, proceeding as in Lemma \ref{meanvalue}, we obtain
\begin{align*}
EF_0[\nabs_T\a](u_1,u_2)\les \frac{\ep_0^2}{u_1^{s+2}}.
\end{align*}
Combining with \eqref{BEFs+2}, we deduce
\begin{equation}\label{estaT}
    EF_p[\nabs_T\a](u_1,u_2)\les\frac{\ep_0^2}{u_1^{s+2-p}},\qquad \forall\; p\in[0,s+2].
\end{equation}
Next, we have from \eqref{Talpha}, \eqref{Tbeta} and \eqref{estaT}
\begin{align}
\begin{split}\label{nabsTb}
BF_p[\nabs_T\b](u_1,u_2)&\les\frac{\ep_0^2}{u_1^{s+2-p}},\qquad\qquad\qquad\qquad\quad
p\in[0,2],\\
E_p[\nabs_T\b](u_2)&\les E_p[\nabs_T\b](u_1)+\frac{\ep_0^2}{u_1^{s+2-p}},\qquad \; p\in[0,2].
\end{split}
\end{align}
Applying Lemma \ref{uother}\footnote{In fact, we apply Lemma \ref{uother} by replacing $s$ with $s+2$.}, we infer
\begin{align*}
    BEF_0[\nabs_T\b](u_1,u_2)\les\frac{\ep_0^2}{u_1^{s+2}}.
\end{align*}
Moreover, we have from Corollary \ref{LieTR} and Proposition \ref{linearcontrol}
\begin{align*}
    BEF_{2}[\nabs_T\b](u_1,u_2)\les BEF_{0}^1[\a,\b,\rho,\si,\bb](u_1,u_2)\les\frac{\ep_0^2}{u_1^s}.
\end{align*}
Hence, we obtain
\begin{equation}\label{estbT}
    BEF_p[\nabs_T\b](u_1,u_2)\les\frac{\ep_0^2}{u_1^{s+2-p}},\qquad \forall\; p\in[0,2].
\end{equation}
Then, we have from \eqref{Trho}, \eqref{estaT} and \eqref{estbT}
\begin{align}
\begin{split}\label{nabsTrho}
BF_p[\nabs_T\rho,\nabs_T\si](u_1,u_2)&\les\frac{\ep_0^2}{u_1^{s+2-p}},\qquad\qquad\qquad\qquad\qquad\quad\;\,\forall\;
p\in[0,2],\\
E_p[\nabs_T\rho,\nabs_T\si](u_2)&\les E_p[\nabs_T\rho,\nabs_T\si](u_1)+\frac{\ep_0^2}{u_1^{s+2-p}},\qquad \;\forall\; p\in[0,1].
\end{split}
\end{align}
Applying Lemma \ref{uother}, we obtain
\begin{align*}
    BEF_0[\nabs_T\rho,\nabs_T\si](u_1,u_2)\les\frac{\ep_0^2}{u_1^{s+2}}.
\end{align*}
Moreover, taking $p=2$ in the first estimate in \eqref{nabsTrho}, we deduce
\begin{align*}
    \int_{u_1}^{u_2} E_{1}[\nabs_T\rho,\nabs_T\si](u)du\les\frac{\ep_0^2}{u_1^s}.
\end{align*}
Proceeding as in Lemma \ref{meanvalue}, we obtain a sequence $u^{(j)}\sim 2^j$ satisfying
\begin{align*}
    E_1[\nabs_T\rho,\nabs_T\si](u^{(j)})\les\frac{\ep_0^2}{(u^{(j)})^{s+1}}.
\end{align*}
Combining with the second estimate in \eqref{nabsTrho}, we deduce
\begin{align*}
    E_1[\nabs_T\rho,\nabs_T\si](u)\les\frac{\ep_0^2}{u^{s+1}}.
\end{align*}
Hence, we have
\begin{equation}\label{estrT}
    BEF_p[\nabs_T\rho,\nabs_T\si](u_1,u_2)\les\frac{\ep_0^2}{u_1^{s+2-p}},\qquad\forall\; p\in[0,1].
\end{equation}
Finally, we have from \eqref{Tbb}, \eqref{estaT}, \eqref{estbT} and \eqref{estrT}
\begin{align}
\begin{split}\label{nabsTbb}
BF_p[\nabs_T\bb](u_1,u_2)&\les\frac{\ep_0^2}{u_1^{s+2-p}},\qquad\quad \forall\; p\in[0,1],\\
E_0[\nabs_T\bb](u_2)&\les E_0[\nabs_T\bb](u_1)+\frac{\ep_0^2}{u_1^{s+2}}.
\end{split}
\end{align}
Applying Lemma \ref{uother}, we obtain
\begin{align}\label{estbbT}
    BEF_0[\nabs_T\bb](u_1,u_2)\les\frac{\ep_0^2}{u_1^{s+2}}.
\end{align}
Injecting it into \eqref{Tbb}, we infer
\begin{align}\label{estaaT}
    F_0[\nabs_T\aa](u_1,u_2)\les\frac{\ep_0^2}{u_1^{s+2}}.
\end{align}
Combining the above estimates, this concludes the proof of Proposition \ref{udecayLieT}.
\end{proof}
\subsection{End of the proof of Theorem \ref{M1}}\label{endM1}
\begin{prop}\label{L4r}
We have the following estimates:
\begin{align}
\begin{split}\label{L4est}
    \DD_{s}[\a,\b,\rho,\si]+\DD_{s-1,1}[\bb]+\DD_{-1,s+1}[\aa]\les\ep_0^2.
\end{split}
\end{align}
\end{prop}
\begin{proof}
We have from Propositions \ref{estab}--\ref{estaa3}
\begin{align*}
\EEe_s^1[\a,\b,\rhoc,\sic]&\les\ep_0^2,\qquad\qquad\;\FFe_s^1[\bb]\les\ep_0^2,\qquad\qquad\;\;\FFe_{0,s}^1[\aa]\les\ep_0^2,\\
\EEe_{0,s}^1[\bb]&\les\ep_0^2,\qquad\qquad \EEe_s[\a_4]\les\ep_0^2,\qquad\;\,\,\FFe_{-2,s+2}[\aa_3]\les\ep_0^2.
\end{align*}
Applying Propositions \ref{fluxsobolev} and \ref{bianchi}, we deduce
\begin{align*}
    \DDe_{s}[\a,\b,\rhoc,\sic]\les\ep_0^2,\qquad \DDe_{s-1,1}[\bb]\les\ep_0^2,\qquad
    \DDe_{-1,s+1}[\aa]\les\ep_0^2.
\end{align*}
Recalling \eqref{renorr}, we obtain
\begin{align}
\begin{split}\label{rhorhoc}
    \sup_{p\in[2,4]}|r^{-\frac{2}{p}}(\rho,\si)|_{p,S}&\les \sup_{p\in[2,4]}|r^{-\frac{2}{p}}(\rhoc,\sic)|_{p,S}+|\Gag|_{\infty,S}\sup_{p\in [2,4]}|r^{-\frac{2}{p}}\Gaw|_{p,S}\\
    &\les\frac{\ep_0}{r^\frac{s+3}{2}}+\frac{\ep^2}{r^\frac{s+1}{2}r\ujp^\frac{s-1}{2}}\\
    &\les\frac{\ep_0}{r^\frac{s+3}{2}},
\end{split}
\end{align}
which implies
\begin{equation*}
    \DDe_{s}[\rho,\si]\les\ep_0^2.
\end{equation*}
Similarly, we have from Propositions \ref{estTba} and \ref{linearcontrolLieT} that in $\Vi$
\begin{align*}
    \EEi_s^1[\a,\b,\rhoc,\sic]&\les\ep_0^2,\qquad\qquad \;\FFi_s^1[\bb]\les\ep_0^2,\qquad\qquad\quad\;\FFi_{0,s}^1[\aa]\les\ep_0^2,\\
    \EEi_{0,s}^1[\bb]&\les\ep_0^2,\qquad\qquad \EEi_{s}[\a_4]\les\ep_0^2,\qquad\quad \FFi_{0,s+2}[\nabs_T\aa]\les\ep_0^2.
\end{align*}
Recalling that $2T=e_4+e_3$ and applying Propositions \ref{fluxsobolev} and \ref{bianchi}, we obtain
\begin{align*}
    \DDi_{s}[\a,\b,\rho,\si]\les\ep_0^2,\qquad
    \DDi_{s-1,1}[\bb]\les\ep_0^2,\qquad
    \DDi_{-1,s+1}[\aa]\les\ep_0^2,
\end{align*}
where we used \eqref{rhorhoc}. This concludes the proof of Proposition \ref{L4r}.
\end{proof}
We then focus on the $u$--decay estimates in the interior region.
\begin{df}\label{truncation}
We define the truncation function on $\Si_t$
    \begin{align*}
        f(x):=h\left(\frac{r(x)}{t}\right),
    \end{align*}
where $h$ is a smooth cut-off function such that $h=1$ for $s\leq \frac{1}{2}$ and $h=0$ for $s\geq \frac{3}{4}$. For any tensor field $X$ on $\Si_t$, we denote:
\begin{align*}
    \Xg:=f X.
\end{align*}
\end{df}
\begin{prop}\label{EHest}
We have the following estimates:
\begin{align*}
    \FFi_{-2,s+2}^1[\a,\b,\rho,\si,\bb,\aa]&\les\ep_0^2,\\
    \DDi_{-2,s+2}[\a,\b,\rho,\si,\bb,\aa]&\les\ep_0^2.
\end{align*}
\end{prop}
\begin{proof}
We have from Propositions \ref{linearBEF} and \ref{udecayLieT}
\begin{align*}
    \FFi_{0,s+2}[\nabs_T(\a,\b,\rho,\si,\bb,\aa)]+\FFi_{0.s}^{1}[\a,\b,\rho,\si,\bb,\aa]\les\ep_0^2,
\end{align*}
which implies from \eqref{EHdecomposition} and \eqref{dfLieh}
    \begin{align*}
        \FFi_{0,s+2}[\Lieh_T E]+\FFi_{0,s+2}[\Lieh_T H]\les\ep_0^2.
    \end{align*}
We recall from Proposition \ref{EHmaxwell}
\begin{align*}
    \sdiv E&=k\wedge H,\\
    \sdiv H&=-k\wedge E,\\
    \curl E&=\Lieh_T H-\nab\log\phi\wedge E-\frac{1}{2}k\times H,\\
    \curl H&=-\Lieh_T E-\nab\log\phi\wedge H+\frac{1}{2}k\times E.
\end{align*}
Commuting with $f$, we easily deduce\footnote{See (7.7.6g)--(7.7.6j) in \cite{Ch-Kl} for the explicit formulae.}
\begin{align*}
    \sdiv\Eg&=\Gaw\c\Hg+\nab f\c\Eg,\\
    \sdiv\Hg&=\Gaw\c\Eg+\nab f\c\Hg,\\
    \curl\Eg&=f\Lieh_TH+\Gaw\c(\Eg,\Hg)+\nab f\c\Eg,\\
    \curl\Hg&=-f\Lieh_TE+\Gaw\c(\Eg,\Hg)+\nab f\c\Hg.
\end{align*}
Applying Proposition \ref{hodgerank2}, we obtain\footnote{Notice from Proposition \ref{EHidentity} that $\sRic=O(E)+\Gaw\c\Gaw$.}
\begin{align*}
    \int_{\Si_t}|\nab\Eg|^2+|\nab\Hg|^2&\les \int_{\Si_t} f^2|\Lieh_TE,\Lieh_TH|^2+|\Gaw\c(\Eg,\Hg)|^2+|\nab f|^2|(\Eg,\Hg)|^2+|(\Eg,\Hg)|^3\\
    &\les \int_{\Si_t\left(r\leq \frac{3t}{4}\right)} |\Lieh_TE,\Lieh_TH|^2+\int_{\frac{t}{4}}^t r^2|r^{-\frac{2}{4}}\Gaw|_{4,S}^2|r^{-\frac{2}{4}}(\Eg,\Hg)|_{4,S}^2du\\
    &+\int_{\frac{t}{2}}^{\frac{3t}{4}}\frac{r^2}{t^2}|r^{-\frac{2}{2}}(\Eg,\Hg)|_{2,S}^2du+\int_{\frac{t}{4}}^t r^2|r^{-\frac{2}{3}}(\Eg,\Hg)|^3_{3,S}du\\
    &\les\frac{\ep_0^2}{t^{s+2}}+\int_{\frac{t}{4}}^t r^2 \frac{\ep^4}{t^{s+1}t^{s+3}}du+\int_{\frac{t}{2}}^{\frac{3t}{4}}\frac{\ep_0^2}{r^{s+3}}du+\int_{\frac{t}{4}}^t r^2 \frac{\ep^3}{t^\frac{3(s+3)
    }{2}}du\\
    &\les\frac{\ep_0^2}{t^{s+2}}+\frac{\ep^4}{t^{2s+1}}+\frac{\ep_0^2}{t^{s+2}}+\frac{\ep^3}{t^\frac{3s+5}{2}}\\
    &\les\frac{\ep_0^2}{t^{s+2}}.
\end{align*}
Combining with Proposition \ref{Hardy}, we obtain
\begin{equation*}
    F_{-2}^1[\Eg,\Hg](\Vi)\les\frac{\ep_0^2}{t^{s+2}}.
\end{equation*}
Next, applying Proposition \ref{fluxsobolev}, we infer
\begin{align*}
    |(\Eg,\Hg)|_{4,S}\les\frac{\ep_0}{t^{\frac{s+2}{2}}}.
\end{align*}
Recalling Definition \ref{truncation}, we deduce
\begin{align*}
    \FFi_{-2,s+2}^1[E,H](\Vii)\les\ep_0^2,\qquad\DDi_{-2,s+2}[E,H](\Vii)\les\ep_0^2.
\end{align*}
The estimates in $\Vie$ follows directly from Propositions \ref{linearBEF} and \ref{udecayLieT}. Combining with \eqref{EHdecomposition}, this concludes the proof of Proposition \ref{EHest}.
\end{proof}
Combining Propositions \ref{estab}--\ref{estaa3}, \ref{linearcontrol}--\ref{linearBEF}, \ref{estTab}--\ref{L4r} and \ref{EHest}, this concludes the proof of Theorem \ref{M1}.
\subsection{Improved estimate for \texorpdfstring{$\b$}{}}
In this section, we prove an improved estimate for $\b$ which will be used in Section \ref{secslu}.
\begin{lem}\label{bulkbeta}
    We have the following estimate:
    \begin{align*}
        \sup_{p\in[-1,s]}\BBi_{p,s-p}^1[\b]\les\ep_0^2.
    \end{align*}
\end{lem}
\begin{proof}
We have from Theorem \ref{M1}
\begin{align*}
    \int_{\Si_t\cap\Vi}r^{p}|\b^{(1)}|^2\les\frac{\ep_0^2}{u_1^{s-p}},\qquad \forall  -2\leq p\leq s.
\end{align*}
Thus, we have
\begin{align*}
    \int_{\Vi}r^{-2}|\b^{(1)}|^2&\les\int_{t_c(u_1)}^t dt \int_{\Si_t\cap\Vii} r^{-2}|\b^{(1)}|^2+\int_{t_c(u_1)}^t dt\int_{\Si_t\cap\Vie} r^{-2}|\b^{(1)}|^2\\
    &\les\ep_0^2\int_{t_c(u_1)}^t\frac{dt}{t^{s+2}}+\int_{t_c(u_1)}^t \frac{dt}{t^2}\int_{\Si_t\cap\Vie}|\b^{(1)}|^2\\
    &\les\frac{\ep_0^2}{u_1^{s+1}}.
\end{align*}
Interpolating with
\begin{align*}
    \int_\Vi r^{s-1}|\b^{(1)}|^2\les\ep_0^2,
\end{align*}
which follows directly from Theorem \ref{M1}, this concludes the proof of Lemma \ref{bulkbeta}.
\end{proof}
\begin{prop}\label{betaimproved}
    We have the following estimate:
    \begin{equation*}
        \EEi_{-1,s+1}^1[\b]\les\ep_0^2.
    \end{equation*}
\end{prop}
\begin{proof}
We recall from Proposition \ref{bianchischematic}
    \begin{align*}
        \nabs_3\b+\trchb\,\b&=-\sld_1^*(\rho,-\si)+\Gag\c\bb+\Gab\c\rho,\\
        \nabs_4(\rho,-\si)+\frac{3}{2}\trch(\rho,-\si)&=\sld_1\b+\Gaw\c\b.
    \end{align*}
Applying Propositions \ref{commutation} and \ref{keyintegralint}, we deduce for $-1\leq p\leq 0$
\begin{align}\label{EFbetaimproved}
    EF_p^1[\b](u_1,u_2)+BF_p^1[\rho,\si](u_1,u_2)\les B_p^1[\b](u_1,u_2)+E_p^1[\b](u_1)+\ee_p^1[R,\Gaw\c R](\Vi),
\end{align}
where
\begin{equation*}
    R\in\{\b,\rho,\si,\bb\}.
\end{equation*}
We first have from Lemma \ref{bulkbeta} that for $-1\leq p\leq 0$
\begin{align*}
    B_p^1[\b](u_1,u_2)\les\frac{\ep_0^2}{u_1^{s-p}}.
\end{align*}
Moreover, we have from the proof of Proposition \ref{estbr} that for $-1\leq p\leq 0$
\begin{equation*}
    \ee_p^1[R,\Gaw\c R](\Vie)\les\frac{\ep_0^2}{u_1^{s-p}}.
\end{equation*}
We also have
\begin{align*}
\ee_p^1[R,\Gaw\c R](\Vii)&\les\int_\Vii r^p |R^{(1)}| |\Gaw| |R^{(1)}|+\int_\Vii r^p |R^{(1)}||\Gaw^{(1)}||R|\\
&\les\ep\int_{t_c(u_1)}^t \frac{dt}{t^\frac{s+1}{2}}\int_{\Si_t\cap \Vii}r^p |R^{(1)}|^2+\left(\int_\Vii r^{2p}|R^{(1)}|^2\right)^\frac{1}{2}\left(\int_\Vii |\Gaw^{(1)}|^2 |R|^2\right)^\frac{1}{2}\\
&\les\int_{t_c(u_1)}^t\frac{dt}{t^\frac{s+1}{2}}\frac{\ep_0^2}{u_1^{s-p}}+\left(\int_{t_c(u_1)}^tdt\int_{\Si_t\cap\Vii} r^{2p}|R^{(1)}|^2\right)^\frac{1}{2}\left(\int_{u_1}^{u_2}du\int_{t_c(u)}^t \frac{\ep^4}{t^{2(s+1)}}dt\right)^\frac{1}{2}\\
&\les\frac{\ep_0^2}{u_1^{s-p}}+\left(\int_{t_c(u_1)}^t \frac{\ep^2}{t^{s-2p}}dt\right)^\frac{1}{2}\frac{\ep^2}{u_1^s}\\
&\les\frac{\ep_0^2}{u_1^{s-p}}.
\end{align*}
Injecting the above estimates into \eqref{EFbetaimproved}, we infer
\begin{align}\label{estbimproved}
    EF_p^1[\b](u_1,u_2)+BF_p^1[\rho,\si](u_1,u_2)\les E_p^1[\b](u_1)+\frac{\ep_0^2}{u_1^{s-p}}.
\end{align}
We recall from Lemma \ref{bulkbeta}
\begin{equation*}
    B_0^1[\b](u_1,u_2)\les\frac{\ep_0^2}{u_1^s},
\end{equation*}
which implies that there exists a sequence $u^{(j)}\sim 2^j$ satisfying
\begin{align*}
    E_{-1}^1[\b](u^{(j)})\les \frac{\ep_0^2}{(u^{(j)})^{s+1}}.
\end{align*}
Applying \eqref{estbimproved} with $u_1=u^{(j-1)}$, $u_2=u\in[2^j,2^{j+1}]$ and $p=-1$, we obtain
\begin{align*}
    E_{-1}^1[\b](u)\les\frac{\ep_0^2}{u^{s+1}}.
\end{align*}
This concludes the proof of Proposition \ref{betaimproved}.
\end{proof}
\section{Maximal connection estimates (Theorem \ref{M2})}\label{seck}
In this section, we apply the elliptic estimates in Section \ref{3D} to prove Theorem \ref{M2}.
\subsection{Preliminaries}
\begin{prop}\label{divcurlk}
We have the following elliptic system on $\Si_t$:
\begin{align*}
\sdiv k&=0,\\
\curl k&=H,\\
\tr k&=0.
\end{align*}
\end{prop}
\begin{proof}
See (11.1.1a) in \cite{Ch-Kl}.
\end{proof}
\begin{prop}\label{sdivcurlk}
We have the following system on $\Si_t$:
\begin{align}
    \begin{split}\label{keyelliptic}
        \sdivs\eps&=-\nab_N\de-\frac{3}{2}\tr\th\,\de+\kah\c\thh-2(a^{-1}\nabs a)\c\eps,\\
        \curls\eps&=\si+\thh\wedge\kah,\\
        \nabs_N\eps+\tr\th\,\eps&=\frac{1}{2}(\b-\bb)+\nabs\de-2\thh\c\eps+\frac{3}{2}(a^{-1}\nabs a)\c\de-\kah\c(a^{-1}\nabs a),\\
        \sdivs\kah&=\frac{1}{2}(-\bb+\b)-\frac{1}{2}\nabs\de+\thh\c\eps-\frac{1}{2}\tr\th\,\eps,\\
        \nabs_N\kah+\frac{1}{2}\tr\th\,\kah&=\frac{1}{4}(-\aa+\a)+\frac{1}{2}\nabs\hot\ep+\frac{3}{2}\de\,\thh+(a^{-1}\nabs a)\hot\eps,
    \end{split}
\end{align}
where we recall
\begin{align*}
    \ka_{AB}=k_{AB},\qquad \eps_{A}=k_{AN},\qquad \de=k_{NN},
\end{align*}
and $\kah$ denotes the traceless part of $\ka$.
\end{prop}
\begin{proof}
    See (11.1.2) in \cite{Ch-Kl}.
\end{proof}
\begin{prop}\label{nabTk}
We have the following identities:
    \begin{align*}
        \nabs_T\de&=-\phi^{-1}\nab_N\nab_N\phi+\rho+\de^2-\eta\c\etab+\eta\c\eps-\etab\c\eps,\\
        \nabs_T\eps&=-\phi^{-1}\nabs\nab_N\phi+\frac{1}{2}(\b+\bb)+(a^{-1}\nabs a)\phi^{-1}\nab_N\phi\\
        &-\frac{3}{2}(\eta-\phi^{-1}\nabs\phi)\de+(\eta-\phi^{-1}\nabs\phi+\eps)\c\kah+\frac{1}{2}\de\,\eps,\\
        \nabs_T\kah&=-\phi^{-1}\nabs\hot\nabs\phi+\frac{1}{4}(\a+\aa)-\de\,\kah+\eps\hot\eps-(\eta-\phi^{-1}\nabs\phi)\hot\eps.
    \end{align*}
\end{prop}
\begin{proof}
    See (11.4.2) in \cite{Ch-Kl}.
\end{proof}
\begin{prop}\label{prpZ}
We define the position vectorfield:
\begin{equation}\label{dfZ}
    Z:=rN.
\end{equation}
We denote $\pi$ the deformation tensor of $Z$:
\begin{equation}
    \pi_{ij}:=\frac{1}{2}\left(\g(\D_{e_i}Z,e_j)+\g(\D_{e_j}Z,e_i)\right),
\end{equation}
and $\pih$ its traceless part. Then, we have
\begin{align*}
    \pih_{ij}&=2r\thh_{ij}+\frac{1}{3}ra^{-1}\widecheck{a\tr\th}(g_{ij}-N_iN_j)-ra^{-1}(N_i\nabs_j a+N_j\nabs_i a),\\
    \tr\pi&=ra^{-1}(2a\tr\th+\ov{a\tr\th}),\\
    \pi_{ij}&=\pih_{ij}+\frac{1}{3}\tr\pi g_{ij}.
\end{align*}
\end{prop}
\begin{proof}
    See (4.1.3) in \cite{Ch-Kl}.
\end{proof}
\begin{prop}\label{divcurliZk}
We define the following $1$--form:
\begin{equation}\label{dfiZk}
    (i_Zk)_{i}:=k_{ij}Z^j.
\end{equation}
Then, we have
\begin{align*}
    \sdiv(i_Zk)&=\frac{1}{2}\pih^{ij}k_{ij},\\
    \curl(i_Zk)_i&=H_{ij}Z^j+\frac{1}{2}{\in_i}^{mn}\left({\pih_m}^j-ra^{-1}(N_m\nabs^ja-N^j\nabs_ma)\right)k_{nj}.
\end{align*}
\end{prop}
\begin{proof}
See (11.2.1) in \cite{Ch-Kl}.
\end{proof}
\begin{prop}\label{rpdivcurl}
    Let $\xi$ be a $1$--form on $\Si$ satisfying
    \begin{align*}
        \sdiv\xi&=D(\xi),\\
        \curl\xi&=A(\xi).
    \end{align*}
    Then, the following estimate holds for all $|p|<1$:
    \begin{align*}
    \int_{\Si_t}r^p|\nab\xi|^2+r^{p-2}|\xi|^2\les\int_{\Si_t}r^p\left(|A(\xi)|^2+|D(\xi)|^2+|\sRic||\xi|^2\right).
    \end{align*}
\end{prop}
\begin{proof}
    Note that we have
    \begin{align*}
        \sdiv(r^\frac{p}{2}\xi)=\nab^i(r^\frac{p}{2}\xi_i)=r^\frac{p}{2}D(\xi)+\frac{p}{2}r^{\frac{p-2}{2}}N(r)\xi_N.
    \end{align*}
    We also have
    \begin{align*}
    \curl(r^\frac{p}{2}\xi)_i&=r^\frac{p}{2}A(\xi)_i+{\in_{i}}^{lm}\nab_l(r^\frac{p}{2})\xi_m\\
    &=r^\frac{p}{2}A(\xi)_i+\frac{p}{2}{\in_i}^{Nm}r^{\frac{p-2}{2}}N(r)\xi_m,
    \end{align*}
    which implies
    \begin{align*}
        \curl(r^\frac{p}{2}\xi)=r^\frac{p}{2}A(\xi)+\frac{p}{2}r^\frac{p-2}{2}N(r)(-\xi_2,\xi_1,0).
    \end{align*}
    Applying Proposition \ref{hodgerank1}, we deduce that for any constant $\de_0>0$
    \begin{align*}
        \int_{\Si_t}|\nab(r^\frac{p}{2}\xi)|^2&=\frac{p^2(1+\de_0)}{4}\int_{\Si_t}r^{p-2}N(r)^2|\xi|^2+\left(1+\frac{1}{\de_0}\right)\int_{\Si_t}r^p\left(|A(\xi)|^2+|D(\xi)|^2\right)\\
        &-\int_{\Si_t} r^p\sRic^{ij}\xi_i\xi_j.
    \end{align*}
    Recalling that $N(r)=1+O(\ep)$, we obtain
    \begin{align*}
        \int_{\Si_t}|\nab(r^\frac{p}{2}\xi)|^2-\frac{p^2(1+\de_0)(1+O(\ep))}{4}\frac{|r^\frac{p}{2}\xi|^2}{r^2}\les\int_{\Si_t} r^p\left(|A(\xi)|^2+|D(\xi)|^2+|\sRic||\xi|^2\right).
    \end{align*}
    Note that we have for $\de_0$ and $\ep$ small enough
    \begin{align*}
        \frac{p^2(1+\de_0)(1+O(\ep))}{4}<c_H^{-2}-\de_0,
    \end{align*}
    where $c_H=2+O(\ep)$ is defined in Proposition \ref{Hardy}. Thus, we obtain
    \begin{align*}
        \int_{\Si_t}|\nab(r^\frac{p}{2}\xi)|^2-(c_H^{-2}-\de_0)\int_{\Si_t}r^{p-2}|\xi|^2\les\int_{\Si_t} r^p\left(|A(\xi)|^2+|D(\xi)|^2+|\sRic||\xi|^2\right).
    \end{align*}
    Applying Proposition \ref{Hardy}, we deduce
    \begin{align*}
        \int_{\Si_t}|\nab(r^\frac{p}{2}\xi)|^2+r^{p-2}|\xi|^2\les\int_{\Si_t} r^p\left(|A(\xi)|^2+|D(\xi)|^2+|\sRic||\xi|^2\right).
    \end{align*}
    Recalling that $|\nab r|=1+O(\ep)$, this concludes the proof of Proposition \ref{rpdivcurl}.
\end{proof}
\subsection{Estimates for \texorpdfstring{$\eps$}{} and \texorpdfstring{$\de$}{}}
\begin{prop}\label{estepsde}
We have the following estimate:
\begin{align*}
\FF_{s}[\nab\eps,\nab\de]\les\ep_0^2,\qquad\FF_{s-2}[\eps,\de]\les\ep_0^2,\qquad\DD_{s-2}[\eps,\de]\les\ep^2_0.
\end{align*}
\end{prop}
\begin{proof}
We have from Propositions \ref{prpZ} and \ref{divcurliZk}
    \begin{align*}
        \sdiv(i_Z k)&=r\Gaw\c\Gaw,\\
        \curl(i_Z k)&=i_ZH+r\Gaw\c\Gaw.
    \end{align*}
Applying Proposition \ref{rpdivcurl} with $p=s-2$, we obtain
\begin{align*}
&\;\;\;\,\,\,\int_{\Si_t}r^{s-4}|i_Zk|^2+r^{s-2}|\nab i_Zk|^2\\
&\les\int_{\Si_t}r^{s-2}|i_ZH|^2+r^{s}|\Gaw\c\Gaw|^2+r^{s-2}|\sRic||i_Zk|^2\\
&\les\int_{\Si_t}r^s|(\b,\si,\bb)|^2+r^s|\Gaw\c\Gaw|^2+r^s|E||\Gag\c\Gag|\\
&\les\ep_0^2+\int_{-\infty}^{u_c(t)}\left(r^{s-2}|r^{1-\frac{2}{4}}\Gaw|_{4,S}^2|r^{1-\frac{2}{4}}\Gaw|_{4,S}^2+|E|_{2,S}|r^{\frac{s}{2}}\Gag|_{4,S}|r^\frac{s}{2}\Gag|_{4,S}\right)du\\
&\les\ep_0^2+\int_{-\infty}^{1} \left(r^{s-2}\frac{\ep^4}{\ujp^{2(s-1)}}+\frac{\ep^3}{\ujp^\frac{s+1}{2}}\right)du+\int_{1}^{u_c(t)}\left(\frac{\ep^4}{u^s}+\frac{\ep^3}{u^\frac{s+1}{2}}\right)du\\
&\les\ep_0^2.
\end{align*}
Combining with \eqref{dfiZk}, we obtain
\begin{align*}
    \FF_{s-2}[\eps,\de]\les\ep_0^2,\qquad\FF_s[\nab\eps,\nab\de]\les\ep_0^2.
\end{align*}
Applying Proposition \ref{fluxsobolev}, we deduce
\begin{align*}
    \DD_{s-2}[\eps,\de]\les\ep_0^2.
\end{align*}
This concludes the proof of Proposition \ref{estepsde}.
\end{proof}
\begin{prop}\label{iZknab1}
We have the following estimates:
\begin{align*}
\FF_s^1[(\nabs,\nabs_N)\eps,(\nabs,\nabs_N)\de]\les\ep_0^2,\qquad\DD_{s-2}^1[\eps,\de]\les\ep_0^2,\qquad\DD_{s-1,1}[\nabs_N\eps,\nabs_N\de]\les\ep_0^2.
\end{align*}
\end{prop}
\begin{proof}
We have from Proposition \ref{divcurliZk}
\begin{align*}
\sdiv(i_Z k)&=r\Gaw\c\Gaw,\\
\curl(i_Z k)&=i_ZH+r\Gaw\c\Gaw.
\end{align*}
Commuting with $r\nabs_A$, we obtain\footnote{See (6.119) in \cite{Bieri} for the precise equations and its derivation.}
\begin{align*}
    \sdiv(r\nabs_A(i_Z k))&=r\Gaw\c\Gaw^{(1)},\\
    \curl(r\nabs_A(i_Z k))&=r\nabs_A(i_ZH)+r\Gaw\c\Gaw^{(1)}.
\end{align*}
Applying Proposition \ref{rpdivcurl} with $p=s-2$, we obtain
\begin{align*}
&\;\;\;\,\,\,\int_{\Si_t}r^s|\nab\nabs_A i_Zk|^2+r^{s-2}|\nabs_Ai_Zk|^2\\
&\les\int_{\Si_t}r^{s}|\nabs i_ZH|^2+r^{s}|\Gaw\c\Gaw^{(1)}|^2+r^{s}|\sRic||\nabs i_Zk|^2\\
&\les\int_{\Si_t}r^s|(\b,\si,\bb)^{(1)}|^2+r^s|\Gaw\c\Gaw|^2+r^s|E||\Gag^{(1)}\c\Gag^{(1)}|,\\
&\les\ep_0^2+\int_{-\infty}^{u_c(t)}\left(|r^{1-\frac{2}{4}}\Gaw|_{4,S}^2|r^{\frac{s}{2}-\frac{2}{4}}\Gaw^{(1)}|_{4,S}^2+|E|_{2,S}|r^{\frac{s}{2}}\Gag^{(1)}|_{4,S}|r^\frac{s}{2}\Gag^{(1)}|_{4,S}\right)du\\
&\les\ep_0^2+\int_{-\infty}^1\left(\frac{\ep^4}{\ujp^{s}}+\frac{\ep^3}{\ujp^\frac{s+1}{2}}\right)du+\int_{1}^{u_c(t)}\left(\frac{\ep^4}{u^s}+\frac{\ep^3}{u^\frac{s+1}{2}}\right)du\\
&\les\ep_0^2.
\end{align*}
Combining with Proposition \ref{estepsde}, we infer
\begin{align*}
    \FF_s^1[(\nabs,\nabs_N)\eps,(\nabs,\nabs_N)\de]\les\ep_0^2.
\end{align*}
Applying Proposition \ref{fluxsobolev}, we deduce
\begin{align*}
    \DD_{s-2}^1[\eps,\de]\les\ep_0^2.
\end{align*}
Finally, we have from Proposition \ref{sdivcurlk}
\begin{align*}
    \sdivs\eps&=-\nab_N\de-\frac{3}{2}\tr\th\,\de+\Gaw\c\Gaw,\\
    \nabs_N\eps+\tr\th\,\eps&=\frac{1}{2}(\b-\bb)+\nabs\de+\Gaw\c\Gab,
\end{align*}
which implies
\begin{align*}
    |r^{-\frac{2}{p}}\nabs_N(\eps,\de)|_{p,S}\les |r^{-\frac{2}{p}}(\b,\bb)|+|r^{-\frac{2}{p}}(\nabs\eps,\nabs\de)|_{p,S}+|r^{-\frac{2}{p}}\Gaw\c\Gaw|_{p,S}\les\frac{\ep_0}{r^\frac{s+2}{2}u^\frac{1}{2}}.
\end{align*}
Hence, we obtain
\begin{equation*}
    \DD_{s-1,1}[\nabs_N\eps,\nabs_N\de]\les\ep_0^2.
\end{equation*}
This concludes the proof of Proposition \ref{iZknab1}.
\end{proof}
\subsection{Estimate for \texorpdfstring{$\kah$}{}}
\begin{prop}\label{estkah}
We have the following estimates:
\begin{align}\label{estkaheq}
\FF_{s-2}^2[\kah]\les\ep_0^2,\qquad\FF_{s}^1[\nabs_N\kah]\les\ep_0^2,\qquad
\DD_{s-3,1}^1[\kah]\les\ep_0^2,\qquad\DD_{-1,s+1}[\nabs_N\kah]\les\ep_0^2.
\end{align}
Moreover, we have
\begin{equation}\label{Gawgain}
    \DD_{-1,s-1}[\kah]\les\ep_0^2.
\end{equation}
\end{prop}
\begin{rk}
Note that \eqref{Gawgain} implies that $\kah$ decays better than $\dkb\kah$.
\end{rk}
\begin{proof}
We have from Proposition \ref{sdivcurlk}
\begin{align*}
    \sdivs\kah&=\frac{1}{2}(-\bb+\b)-\frac{1}{2}\nabs\de-\frac{1}{2}\tr\th\,\eps+\Gag\c\Gaw,\\
    \nabs_N\kah+\frac{1}{2}\tr\th\,\kah&=\frac{1}{4}(-\aa+\a)+\frac{1}{2}\nabs\hot\ep+\Gag\c\Gaw.
\end{align*}
Applying Propositions \ref{Lpestimates} and \ref{iZknab1}, we immediately obtain \eqref{estkaheq}. Next, applying Lemma \ref{evolution}, we have for $p\in[2,4]$ and $u\leq 1$
\begin{align*}
    |r^{1-\frac{2}{p}}\kah|_{p,S(t,u)}&\les\lim_{u\to\infty}|r^{1-\frac{2}{p}}\kah|_{p,S(t,u)}+\int_{-\infty}^u\left(|r^{1-\frac{2}{p}}(\a,\aa,\nabs\eps)|_{p,S}+|r^{1-\frac{2}{p}}\Gag\c\Gaw|_{p,S}\right)du\\
    &\les\int_{-\infty}^u \left(\frac{\ep_0}{\ujp^\frac{s+1}{2}}+\frac{\ep^2}{\ujp^s}\right)du\\
    &\les\frac{\ep_0}{\ujp^\frac{s-1}{2}},
\end{align*}
and similarly for $1\leq u\leq u_c(t)$
\begin{align*}
    |r^{1-\frac{2}{p}}\kah|_{p,S(t,u)}&\les \lim_{u\to u_c(t)}|r^{1-\frac{2}{p}}\kah|_{p,S(t,u)}+\int_{u}^{u_c(t)}\left(|r^{1-\frac{2}{p}}(\a,\aa,\nabs\eps)|_{p,S}+|r^{1-\frac{2}{p}}\Gag\c\Gaw|_{p,S}\right)du\\
    &\les\int_{u}^{u_c(t)}\left(\frac{\ep_0}{u^\frac{s+1}{2}}+\frac{\ep^2}{u^s}\right)du\\
    &\les\frac{\ep_0}{u^\frac{s-1}{2}}.
\end{align*}
Hence, we obtain \eqref{Gawgain}. This concludes the proof of Proposition \ref{estkah}.
\end{proof}
\subsection{Estimate for the lapse function \texorpdfstring{$\phi$}{}}
We have from \eqref{lapseequation} and \eqref{B1} that
\begin{equation*}
    \De\phi=\phi|k|^2,\qquad\quad \lim_{u\to -\infty}\phi=1.
\end{equation*}
In view of the maximum principle and Harnarck inequality\footnote{See for example Theorem 3.1 in \cite{GT} for the maximum principle and Theorem 8.20 in \cite{GT} for Harnarck inequality.}, we infer
\begin{align*}
    0<\phi\leq 1.
\end{align*}
Thus, we can define the following negative scalar function:
\begin{equation}\label{logphi}
    \vphi:=\log\phi.
\end{equation}
\begin{lem}\label{nablog}
We have the following equation for $\vphi$:
\begin{equation*}
    \De\vphi=|k|^2-|\nab\vphi|^2,\qquad\quad \lim_{u\to-\infty}\vphi=0.
\end{equation*}
\end{lem}
\begin{proof}
We have from \eqref{logphi}
    \begin{align*}
\De\vphi=\nab^i\nab_i(\log\phi)=\nab^i\left(\phi^{-1}\nab_i\phi\right)=\phi^{-1}\De\phi-\phi^{-2}\nab^i\phi\nab_i\phi=|k|^2-|\nab\vphi|^2,
    \end{align*}
as stated. We also have from \eqref{logphi}
    \begin{equation*}
        \lim_{u\to-\infty}\vphi=\log\left(\lim_{u\to-\infty}\phi\right)=0.
    \end{equation*}
This concludes the proof of Lemma \ref{nablog}.
\end{proof}
\begin{prop}\label{lapse}
We have the following estimate:
\begin{align*}
    \FF_{s-2}[\nab\vphi]\les\ep_0^2,\qquad\quad \DD_{s-4}[\vphi]\les\ep_0^2.
\end{align*}
\end{prop}
\begin{proof}
Recalling that $N(r)=1+r\Gab$, we have from Lemma \ref{nablog} 
\begin{align*}
    r^{s-2}|\nab\vphi|^2&=\nab^i(r^{s-2}\vphi\nab_i\vphi)-\nab^i(r^{s-2})\vphi\nab_i\vphi-r^{s-2}\vphi\De\vphi\\
    &=\sdiv(r^{s-2}\vphi\nab\vphi)-(s-2)r^{s-3}N(r)\vphi N(\vphi)-r^{s-2}\vphi\left(|k|^2-|\nab\vphi|^2\right)\\
    &=\sdiv(r^{s-2}\vphi\nab\vphi)-(s-2)r^{s-3}\vphi N(\vphi)+r^{s-1}\Gag\c\Gaw\c\Gaw.
\end{align*}
We also have from \eqref{divNfor}
\begin{align*}
    \int_{\Si_t} 2r^{s-3}\vphi N(\vphi)&=-\int_{\Si_t} \sdiv(r^{s-3}N)|\vphi|^2\\
    &=-\int_{\Si_t} r^{s-3}\sdiv(N)|\vphi|^2-(s-3)\int_{\Si_t}r^{s-4}N(r)|\vphi|^2\\
    &=-(s-1)\int_{\Si_t}r^{s-4}|\vphi|^2+\int_{\Si_t} r^{s-1}\Gag\c\Gag\c\Gab.
\end{align*}
Combining the above identities, we infer
\begin{align*}
    \int_{\Si_t}r^{s-2}|\nab\vphi|^2&=-\int_{\Si_t}(s-2)r^{s-3}\vphi N(\vphi)+\int_{\Si_t}r^{s-1}\Gag\c\Gaw\c\Gaw\\
    &=\frac{(s-2)(s-1)}{2}\int_{\Si_t}r^{s-4}|\vphi|^2+\int_{\Si_t}r^{s-1}\Gag\c\Gaw\c\Gaw\\
    &\les\int_{-\infty}^u r^{s+1}|r^{-\frac{2}{2}}\Gag|_{2,S}|r^{-\frac{2}{4}}\Gaw|_{4,S}|r^{-\frac{2}{4}}\Gaw|_{4,S}\\
    &\les\int_{-\infty}^1 r^{s+1}\frac{\ep}{r^\frac{s+1}{2}}\frac{\ep}{r\ujp^\frac{s-1}{2}}\frac{\ep}{r\ujp^\frac{s-1}{2}}+\int_{1}^{u_c(t)}r^{s+1}\frac{\ep}{r^\frac{s+1}{2}}\frac{\ep}{ru^\frac{s-1}{2}}\frac{\ep}{r^\frac{s-1}{2}u}\\
    &\les\int_{-\infty}^1 \frac{\ep^3}{\ujp^\frac{s+1}{2}}du+\int_{1}^{u_c(t)}\frac{\ep^3}{u^\frac{s+1}{2}}du\\
    &\les\ep_0^2,
\end{align*}
where we used the fact that
\begin{equation*}
    (s-2)(s-1)\leq 0.
\end{equation*}
Hence, we deduce
    \begin{align*}
    \FF_{s-2}[\nab\vphi]\les\ep_0^2.
    \end{align*}
Applying Proposition \ref{fluxsobolev}, we obtain
    \begin{align*}
        \DD_{s-4}[\vphi]\les\ep_0^2.
    \end{align*}
This concludes the proof of Proposition \ref{lapse}.
\end{proof}
\begin{prop}\label{lapsenab}
We have the following estimates:
\begin{align*}
    \FF_s[\nab^2\vphi]\les\ep_0^2,\qquad\quad\DD_{s-2}[\nab\vphi]\les\ep_0^2.
\end{align*}
\end{prop}
\begin{proof}
We have from Lemma \ref{nablog}
\begin{align*}
    \sdiv(\nab\vphi)=\Gaw\c\Gaw,\qquad\curl(\nab\vphi)=0.
\end{align*}
Thus, we obtain
\begin{align*}
    \sdiv(r^\frac{s}{2}\nab\vphi)=\nab^i(r^\frac{s}{2}\nab_i\vphi)=\frac{s}{2}r^{\frac{s-2}{2}}N(r)\nab_N\vphi+r^\frac{s}{2}\sdiv(\nab\vphi)=r^\frac{s-2}{2}O(\nab\vphi)+r^\frac{s}{2}\Gaw\c\Gaw,
\end{align*}
and similarly
\begin{align*}
    \curl(r^\frac{s}{2}\nab\vphi)=r^\frac{s-2}{2}O(\nab\vphi).
\end{align*}
Applying Propositions \ref{fluxsobolev} and \ref{lapse}, we infer
\begin{align*}
    \int_{\Si_t}|\nab(r^\frac{s}{2}\nab\vphi)|^2&\les\int_{\Si_t}r^{s-2}|\nab\vphi|^2+r^s|\Gaw\c\Gaw|^2+|\sRic||r^\frac{s}{2}\nab\vphi|^2\\
    &\les\ep_0^2+\int_{-\infty}^{u_c(t)}|r^{1-\frac{2}{4}}\Gaw|_{4,S}^2|r^{\frac{s}{2}-\frac{2}{4}}\Gaw|_{4,S}^2+|r^{1-\frac{2}{2}}E|_{2,S}|r^{\frac{s+1}{2}-\frac{2}{4}}\Gag|_{4,S}|r^{\frac{s+1}{2}-\frac{2}{4}}\Gag|_{4,S}\\
    &\les\ep_0^2+\int_{-\infty}^{1}\left(\frac{\ep^4}{\ujp^{s}}+\frac{\ep^3}{\ujp^\frac{s+1}{2}}\right)du+\int_1^{u_c(t)}\left(\frac{\ep^4}{u^{s}}+\frac{\ep^3}{u^\frac{s+1}{2}}\right)du\\
    &\les\ep_0^2.
\end{align*}
Combining with Proposition \ref{lapse}, we obtain
\begin{align*}
    \FF_s[\nab^2\vphi]\les\ep_0^2.
\end{align*}
Finally, applying Proposition \ref{fluxsobolev}, we deduce
\begin{align*}
    \DD_{s-2}[\nab\vphi]\les\ep_0^2.
\end{align*}
This concludes the proof of Proposition \ref{lapsenab}.
\end{proof}
\begin{prop}\label{lapsenabnab}
We have the following estimates:
\begin{align*}
\FF^1_s[\nab_N\nab_N\vphi]+\FF_{s-2}^2[\nab_N\vphi]+\FF_{s-4}^3[\vphi]&\les\ep_0^2,\\
\DD_{s-1,1}[\nab_N\nab_N\vphi]+\DD_{s-2}^1[\nab_N\vphi]+\DD_{s-4}^2[\vphi]&\les\ep_0^2.
\end{align*}
\end{prop}
\begin{proof}
We have from Lemma \ref{nablog}
\begin{align*}
\De\vphi=\Gaw\c\Gaw,
\end{align*}
which implies
\begin{align*}
    \De(r^{\frac{s-2}{2}}\vphi)&=\De(r^\frac{s-2}{2})\vphi+2\nab^i(r^\frac{s-2}{2})\nab_i\vphi+r^\frac{s-2}{2}\De\vphi\\
    &=r^{\frac{s-6}{2}}O(\vphi)+r^\frac{s-4}{2}O(\nab\vphi)+r^\frac{s-2}{2}\Gaw\c\Gaw.
\end{align*}
Applying Proposition \ref{hodgerank0}, we deduce
\begin{align*}
&\;\;\;\,\,\int_{\Si_t}r^4|\nabs^3(r^{\frac{s-2}{2}}\vphi)|^2+r^4|\nabs^2\nab_N(r^{\frac{s-2}{2}}\vphi)|^2+r^4|\nabs\nab_N^2(r^{\frac{s-2}{2}}\vphi)|^2\\
&\les\int_{\Si_t}r^{s-4}|\vphi|^2+r^{s-2}|\nab\vphi|^2+r^s|\nabs\nab\vphi|^2+r^s|\Gaw\c\Gaw^{(1)}|^2\\
&\les\ep_0^2+\int_{-\infty}^{u_c(t)}|r^{1-\frac{2}{4}}\Gaw|_{4,S}^2|r^{\frac{s}{2}-\frac{2}{4}}\Gaw^{(1)}|^2_{4,S}\\
&\les\ep_0^2+\int_{-\infty}^{1}\frac{\ep^2}{\ujp^{s-1}}\frac{\ep^2}{\ujp}du+\int_{1}^{u_c(t)}\frac{\ep^2}{u^{s-1}}\frac{\ep^2}{u}du\\
&\les\ep_0^2.
\end{align*}
Combining with Proposition \ref{lapsenab}, we obtain
\begin{align*}
\FF^1_s[\nab_N\nab_N\vphi]+\FF_{s-2}^2[\nab_N\vphi]+\FF_{s-4}^3[\vphi]\les\ep_0^2.
\end{align*}
Applying Proposition \ref{fluxsobolev}, we infer
\begin{align*}
    \DD_{s-2}^1[\nab_N\vphi]+\DD_{s-4}^2[\vphi]\les\ep_0^2.
\end{align*}
Finally, we have from Lemma \ref{nablog}
\begin{align*}
    \nab_N\nab_N\vphi=-g^{AB}\nab_A\nab_B\vphi+\Gaw\c\Gaw,
\end{align*}
which implies
\begin{align*}
    \DD_{s-1,1}[\nab_N\nab_N\vphi]\les\ep_0^2.
\end{align*}
This concludes the proof of Proposition \ref{lapsenabnab}.
\end{proof}
\subsection{\texorpdfstring{$u$}{}--decay estimates}
In this section, we always assume that $u\geq 1$.
\begin{prop}\label{estkint}
We have the following estimates:
\begin{align*}
\FFi_{-2,s}[k]&\les\ep_0^2,\qquad\qquad \FFi_{0,s}[\nab k]\les\ep_0^2,\qquad\qquad\FFi_{0,s+2}[\nab^2k]\les\ep_0^2,\\ ^i\DDinf_{-3,s+1}[k]&\les\ep_0^2,\qquad\;\, \DDi_{-2,s+2}[\nab k]\les\ep_0^2.
\end{align*}
\end{prop}
\begin{proof}
We recall from Proposition \ref{divcurlk}
\begin{align*}
    \sdiv k&=0,\\
    \curl k&=H,\\
    \tr k&=0.
\end{align*}
Commuting with the truncation function $f$ defined in Definition \ref{truncation}, we deduce
\begin{align*}
    \sdiv\kg&=O(k)\nab f,\\
    \curl\kg&=\Hg+O(k)\nab f,\\
    \tr\kg&=0.
\end{align*}
Applying Proposition \ref{hodgerank2}, we infer
\begin{align*}
\int_{\Si_t}|\nab \kg|^2&\les\int_{\Si_t}|\Hg|^2+|\nab f|^2|k|^2+|\sRic||\kg|^2\\
&\les\frac{\ep_0^2}{t^s}+\int_{\frac{t}{2}\leq r\leq \frac{3t}{4}}\frac{\ep_0^2}{t^{s+1}}+\int_{0\leq r\leq \frac{3t}{4}}|r^{1-\frac{2}{2}}\sRic|_{2,S}|r^{\frac{1}{2}-\frac{2}{4}}\kg|_{4,S}^2\\
&\les\frac{\ep_0^2}{t^s}+\int_{0\leq r\leq \frac{3t}{4}}\frac{\ep^3}{t^\frac{s+1}{2}t^s}\\
&\les\frac{\ep_0^2}{t^s},
\end{align*}
and similarly
\begin{align*}
\int_{\Si_t}|\nab^2\kg|^2&\les\int_{\Si_t}|\nab\Hg|^2+|\nab^2f|^2|k|^2+|\nab f|^2|\nab k|^2+|\sRic||\nab k|^2+|\sRic|^2|k|^2\\
&\les\frac{\ep_0^2}{t^{s+2}}+\int_{\frac{t}{2}\leq r\leq\frac{3t}{4}}\frac{\ep_0^2}{t^{s+3}}+\frac{\ep_0^2}{t^{s+3}}+\int_{0\leq r\leq \frac{3t}{4}}\frac{\ep^3}{t^{\frac{3s+7}{2}}}+\frac{\ep^4}{t^{2s+2}}\\
&\les\frac{\ep_0^2}{t^{s+2}},
\end{align*}
which implies from Proposition \ref{fluxsobolev} that
\begin{align*}
    \DDi_{-2,s+2}[\nab\kg]\les\ep_0^2.
\end{align*}
Moreover, applying Proposition \ref{Hardy}, we infer
\begin{align*}
    \int_{\Si_t}r^{-2}|\nab\kg|^2\les\frac{\ep_0^2}{t^{s+2}}.
\end{align*}
Finally, applying Propositions \ref{fluxsobolev} and \ref{GN}, we obtain
\begin{align*}
    |\kg|_{\infty,S(t,u)}\les \|\kg\|_{6,\Si_t}^\frac{1}{2}\|\nab\kg\|_{6,\Si_t}^\frac{1}{2}\les\|\nab\kg\|_{2,\Si_t}^\frac{1}{2}\|\nab^2\kg\|_{2,\Si_t}^\frac{1}{2}\les \frac{\ep_0}{t^\frac{s+1}{2}}.
\end{align*}
Combining with Propositions \ref{estepsde}--\ref{estkah}, this concludes the proof of Proposition \ref{estkint}.
\end{proof}
\begin{prop}\label{lapseint}
We have the following estimates:
\begin{align*}
\FFi_{-4,s}[\vphi]&\les\ep_0^2,\qquad\quad\; \FFi_{-2,s}[\nab\vphi]\les\ep_0^2,\qquad\quad\;\;\FFi_{0,s}[\nab^2\vphi]\les\ep_0^2,\qquad \FFi_{0,s+2}[\nab^3\vphi]\les\ep_0^2,\\
^i\DDinf_{-3,s-1}[\vphi]&\les\ep_0^2,\qquad\; ^i\DDinf_{-3,s+1}[\nab\vphi]\les\ep_0^2,\qquad\DDi_{-2,s+2}[\nab^2\vphi]\les\ep_0^2.
\end{align*}
\end{prop}
\begin{proof}
We recall from Lemma \ref{nablog}
    \begin{align*}
        \sdiv(\nab\vphi)=\Gaw\c\Gaw,\qquad \curl(\nab\vphi)=0,
    \end{align*}
which implies
    \begin{align*}
        \sdiv(f\nab\vphi)&=f\Gaw\c\Gaw+\nab f\c \nab\vphi,\\
        \curl(f\nab\vphi)&=\nab f\c\nab\vphi.
    \end{align*}
Applying Proposition \ref{hodgerank2} and proceeding as in Proposition \ref{estkint}, we deduce
\begin{align*}
    \int_{\Si_t}f^2|\nab^2\vphi|^2&\les\int_{\Si_t}f^2|\Gaw\c\Gaw|^2+|\nab f|^2|\nab\vphi|^2+f^2|\sRic||\nab\vphi|^2\\
    &\les\int_{0\leq r\leq \frac{3t}{4}}|\Gaw|_{4,S}^2|\Gaw|_{4,S}^2+|\sRic|_{2,S}|\nab\vphi|^2_{4,S}+\int_{\frac{t}{2}\leq r\leq \frac{3t}{4}}|\nab\vphi|^2\\
    &\les\frac{\ep_0^2}{t^s}.
\end{align*}
We also have from Proposition \ref{hodgerank2}
\begin{align*}
\int_{\Si_t} f^2|\nab^3\vphi|^2&\les\int_{\Si_t}|\nab(f\Gaw\c\Gaw)|^2+|\nab(\nab f\c\nab\vphi)|^2+|\sRic||\nab(f\nab\vphi)|^2+|\sRic|^2|f\nab\vphi|^2\\
&\les\int_{\Si_t} f^2|\Gaw\c\nab\Gaw|^2+|\nab f|^2|\Gaw\c\Gaw|^2+|\nab^2 f|^2|\nab\vphi|^2+|\nab f|^2|\nab^2\vphi|^2\\
&+\int_{\Si_t}|\sRic||\nab f|^2|\nab\vphi|^2+f^2|\sRic||\nab^2\vphi|^2+f^2|\sRic|^2|\nab\vphi|^2\\
&\les\int_{0\leq r\leq \frac{3t}{4}}|\Gaw|^2_{4,S}|\nab\Gaw|_{4,S}^2+|\sRic|_{2,S}|\nab^2\vphi|_{4,S}^2+|\sRic|^2_{4,S}|\nab\vphi|_{4,S}^2\\
&+\int_{\frac{t}{2}\leq r\leq \frac{3t}{4}} r^{-2}|\Gaw|^2_{4,S}|\Gaw|^2_{4,S}+r^{-4}|\nab\vphi|_{2,S}^2+r^{-2}|\nab^2\vphi|_{2,S}^2+r^{-2}|\sRic|_{2,S}|\nab\vphi|_{4,S}^2\\
&\les\frac{\ep_0^2}{t^{s+2}}+\int_{0\leq r\leq\frac{3t}{4}}\frac{\ep^4}{t^{2s+2}}\\
&\les\frac{\ep_0^2}{t^{s+2}}.
\end{align*}
Applying Proposition \ref{Hardy}, we deduce
\begin{align*}
\FFi_{-4,s}[\vphig]\les\ep_0^2,\qquad \FFi_{-2,s}[\nab\vphig]\les\ep_0^2,\qquad \FFi_{0,s}[\nab^2\vphig]\les\ep_0^2,\qquad \FFi_{0,s+2}[\nab^3\vphig]\les\ep_0^2.
\end{align*}
Thus, we have from Proposition \ref{fluxsobolev}
\begin{align*}
    \DDi_{-4,s}[\vphig]\les\ep_0^2,\qquad\quad \DDi_{-2,s+2}[\nab^2\vphig]\les\ep_0^2.
\end{align*}
Finally, applying Propositions \ref{fluxsobolev} and \ref{GN}, we obtain
\begin{align*}
    |\nab\vphig|_{\infty,S(t,u)}\les\|\nab\vphig\|_{6,\Si_t}^\frac{1}{2}\|\nab^2\vphig\|_{6,\Si_t}^\frac{1}{2}\les\|\nab^2\vphig\|_{2,\Si_t}^\frac{1}{2}\|\nab^3\vphig\|_{2,\Si_t}^\frac{1}{2}\les\frac{\ep_0}{t^\frac{s+1}{2}},
\end{align*}
which implies
\begin{equation*}
    ^i\DDinf_{-3,s}[\nab\vphig]\les\ep_0^2.
\end{equation*}
Combining with Propositions \ref{lapse}--\ref{lapsenabnab}, this concludes the proof of Proposition \ref{lapseint}.
\end{proof}
\subsection{Time derivative estimates in the interior region}
\begin{prop}\label{LieTk}
We have the following estimates:
\begin{align*}
    \sup_{p\in[-2,s-1]}\DDi_{p,s-p}[\nabs_T\de,\nabs_T\eps]\les\ep_0^2,\qquad\quad \sup_{p\in[-2,-1]}\DDi_{p,s-p}[\nabs_T\kah]\les\ep_0^2.
\end{align*}
\end{prop}
\begin{proof}
We have from Proposition \ref{nabTk}
\begin{align*}
    \nabs_T\de&=-\phi^{-1}\nab_N\nab_N\phi+\rho+\Gag\c\Gab,\\
    \nabs_T\eps&=-\phi^{-1}\nabs\nab_N\phi+\frac{1}{2}(\b+\bb)+\Gab\c\Gaw,\\
    \nabs_T\kah&=-\phi^{-1}\nabs\hot\nabs\phi+\frac{1}{4}(\a+\aa)+\Gag\c\Gaw.
\end{align*}
Combining with Proposition \ref{lapsenabnab}, we infer
\begin{align*}
    \DDi_{s-1,1}[\nabs_T\de]\les\ep_0^2,\qquad \DDi_{s-1,1}[\nabs_T\eps]\les\ep_0^2,\qquad \DDi_{-1,s+1}[\nabs_T\kah]\les\ep_0^2.
\end{align*}
Moreover, we have from Proposition \ref{lapseint}
\begin{align*}
    \DDi_{-2,s+2}[\nabs_T\de,\nabs_T\eps,\nabs_T\kah]\les\ep_0^2.
\end{align*}
This concludes the proof of Proposition \ref{LieTk}.
\end{proof}
\begin{prop}\label{LieTphi}
We have the following estimate:
\begin{align*}
    \sup_{p\in[-2,s-2]}\DDi_{p,s-2-p}[T\phi]\les\ep_0^2,\qquad\quad\sup_{p\in[-2,s-1]}\DDi_{p,s-p}[\nab_T\nab\phi]\les\ep_0^2.
\end{align*}
\end{prop}
\begin{proof}
We first introduce the spacetime scaling operator
\begin{equation}\label{defscaling}
    S:=\tb\, T+Z,\qquad \tb:=r+u,
\end{equation}
where $Z=rN$ is defined in \eqref{dfZ}. Commuting $S$ with \eqref{lapseequation}, we obtain\footnote{See (12.0.5f) in \cite{Ch-Kl} for the precise equation and its derivation.}
\begin{align}\label{Sphieq}
    \De(S\phi)=2\phi\kah^{AB}(\phi\,\tb\sRic_{AB}+r\nab_N\kah_{AB})+\Gaw\c\Gaw.
\end{align}
Moreover, we have from Proposition \ref{sdivcurlk}
\begin{align*}
    \nab_N\kah=-\frac{1}{4}\aa+r^{-1}\Gab^{(1)}.
\end{align*}
Combining with Propositions \ref{EHiden} and \ref{EHidentity}, we have
\begin{align*}
    2\phi\kah^{AB}(\phi\,\tb\sRic_{AB}+r\nab_N\kah_{AB})&=\Gaw\c\left(\tb E-\frac{1}{4}r\aa+\Gab^{(1)}\right)\\
    &=\Gaw\c(u\,\aa)+\Gaw\c\Gab^{(1)}\\
    &=\Gaw\c\Gab^{(1)},
\end{align*}
where we used $u\,\aa\in\Gab^{(1)}$. Injecting it into \eqref{Sphieq}, we infer
\begin{align*}
    \De(S\phi)=\Gaw\c\Gab^{(1)}.
\end{align*}
Thus, we deduce that $S\phi$ satisfies a similar equation as $\vphi$. Proceeding as in Propositions \ref{lapse}, \ref{lapsenab} and \ref{lapseint}, we obtain
\begin{align*}
\sup_{p\in[-3,s-4]}\DDi_{p,s-4-p}[S\phi]\les\ep_0^2,\qquad\quad\sup_{p\in[-2,s-2]}\DDi_{p,s-2-p}[\nab(S\phi)]\les\ep_0^2.
\end{align*}
Next, we recall from Propositions \ref{lapsenabnab} and \ref{lapseint}
\begin{align*}
\sup_{p\in[-2,s-2]}\DDi_{p,s-2-p}[\nab\vphi]\les\ep_0^2,\qquad\quad\sup_{p\in[-2,s-3]}\DDi_{p,s-2-p}[\nab\nab_Z\vphi]\les\ep_0^2.
\end{align*}
Combining the above estimates, we infer
\begin{align*}
\sup_{p\in[-2,s-4]}\DDi_{p,s-4-p}[\tb\,T\phi]\les\ep_0^2,\qquad\quad\sup_{p\in[-2,s-3]}\DDi_{p,s-2-p}[\nab(\tb\,T\phi)]\les\ep_0^2.
\end{align*}
Recalling that $\tb=r+u$, this concludes the proof of Proposition \ref{LieTphi}.
\end{proof}
Combining Propositions \ref{lapse}--\ref{LieTphi}, this concludes the proof of Theorem \ref{M2}.
\section{Null connection estimates (Theorem \ref{M3})}\label{sec10}
In this section, we apply Lemma \ref{evolution} to prove Theorem \ref{M3}. Throughout this section, we always assume that $p\in[2,4]$ and we denote
\begin{align*}
    S:=S(t,u),\qquad S_2:=S(2,u),\qquad S_c:=\lim_{t\to t_c(u)}S(t,u).
\end{align*}
\subsection{Null transport equations}
The following equations will be used repeatedly in Section \ref{sec10}.
\begin{prop}\label{slutrans}
We have the following transport equations:
\begin{align*}
\nabs_4\slu+\frac{3}{2}\trch\,\slu&=-\trch\,\b-\hch\c\slu-\nabs|\hch|^2+\trch\,\hch\c\etab-2|\hch|^2\eta,\\
\nabs_4\mu+\trch\,\mu&=\frac{1}{2}\trch\left(-\sdivs\etab-\rhoc-|\etab|^2\right)+\hch\c\nabs\hot\eta\\
&+(\eta-\etab)\c\slu-\frac{1}{4}\trchb|\hch|^2+2\eta\c\hch\c\etab-2\eta\c\b.
\end{align*}
\end{prop}
\begin{proof}
    See (13.1.6c) and (13.1.11) in \cite{Ch-Kl}.
\end{proof}
\begin{prop}\label{recalleq}
We have the following equations:
\begin{align*}
    \nabs_4\eta+\frac{1}{2}\trch\,\eta&=\frac{1}{2}\trch\,\etab-\hch\cdot(\eta-\etab)-\b,\\
    \sdivs\eta&=-\mu-\rhoc,\\
    \curls\eta&=\sic,\\
    \nabs_4\trch+\frac{1}{2}(\trch)^2&=-2\om\trch-|\hch|^2,\\
    \sdivs\hch&=\frac{1}{2}\slu-\b-\ze\c\hch.
\end{align*}
\begin{proof}
    It follows directly from Proposition \ref{nulles}, \eqref{dfmu} and \eqref{renorr}.
\end{proof}
\end{prop}
\subsection{Estimate for \texorpdfstring{$\slu$}{}}\label{secslu}
\begin{prop}\label{extslu}
We have the following estimates:
\begin{align*}
    \DD_s[\slu]+\DDi_{-2,s+2}[\slu]\les\ep_0^2.
\end{align*}
\end{prop}
\begin{proof}
    We have from Proposition \ref{slutrans}
    \begin{align*}
        \nabs_4\slu+\frac{3}{2}\trch\,\slu=-\trch\,\b+r^{-1}\Gag\c\Gag^{(1)}.
    \end{align*}
    Applying Lemma \ref{evolution}, we obtain for $u\leq 1$
    \begin{align*}
        |r^{3-\frac{2}{p}}\slu|_{p,S}&\les |r^{3-\frac{2}{p}}\slu|_{p,S_2}+\int_{\ujp}^{r}|r^{2-\frac{2}{p}}\b|_{p,S}+|r^{2-\frac{2}{p}}\Gag\c\Gag^{(1)}|_{p,S}\\
        &\les \frac{\ep_0}{\ujp^\frac{s-3}{2}}+\int_{|u|}^{r}\frac{\ep_0}{r^\frac{s-1}{2}}+\frac{\ep^2}{r^{s-1}}\\
        &\les \frac{\ep_0}{r^\frac{s-3}{2}},
    \end{align*}
    which implies
    \begin{align}\label{slue}
        \DDe_{s}[\slu]\les\ep_0.
    \end{align}
    Note that we have from \eqref{B1}
    \begin{align*}
        \lim_{t\to t_c(u)}|r^{3-\frac{2}{p}}\slu|_{p,S}=\lim_{t\to t_c(u)}r^\frac{5}{2}|r^{\frac{1}{2}-\frac{2}{p}}\slu|_{p,S}\les\lim_{r\to 0}\frac{\ep r^\frac{5}{2}}{u^\frac{s+2}{2}}=0.
    \end{align*}
    Thus, for any $S\subseteq\Vie$, we have from Lemma \ref{evolution}
    \begin{align*}
        |r^{3-\frac{2}{p}}\slu|_{p,S}&\les|r^{3-\frac{2}{p}}\slu|_{p,S_c}+\int_{t_c(u)}^{t}|r^{2-\frac{2}{p}}\b|_{p,S}+|r^{2-\frac{2}{p}}\Gag\c\Gag^{(1)}|_{p,S}\\
        &\les\int_{t_c(u)}^{t}\frac{\ep_0}{t^\frac{s-1}{2}}+\frac{\ep^2}{t^{s-1}}\\
        &\les\frac{\ep_0}{t^\frac{s-3}{2}},
    \end{align*}
    which implies
    \begin{align}\label{sluVie}
        |r^{\frac{s+3}{2}-\frac{2}{p}}\slu|_{p,S}\les\ep_0^2,\qquad\forall\; S\subseteq \Vie.
    \end{align}
    Moreover, for any $S\subseteq\Vii$, we have from Lemma \ref{evolution}
   \begin{align*}
        |r^{3-\frac{2}{p}}\slu|_{p,S}&\les|r^{3-\frac{2}{p}}\slu|_{p,S_c}+\int_{t_c(u)}^{t}|r^{2-\frac{2}{p}}\b|_{p,S}+|r^{2-\frac{2}{p}}\Gag\c\Gag^{(1)}|_{p,S}\\
        &\les\int_{t_c(u)}^{t}\frac{\ep_0 r^\frac{3}{2}}{u^\frac{s+2}{2}}+\frac{\ep^2 r^\frac{3}{2}}{u^\frac{2s+1}{2}}\\
        &\les\frac{\ep_0 r^\frac{5}{2}}{u^\frac{s+2}{2}},
    \end{align*}
    where we used $|t_c(u)-t|\les r$ for any $S(t,u)\subseteq \Vii$. Hence, we obtain
    \begin{align}\label{sluVii}
        |r^{\frac{1}{2}-\frac{2}{p}}\slu|_{p,S}\les\frac{\ep_0}{u^\frac{s+2}{2}},\qquad \forall\; S\subseteq \Vii.
    \end{align}
    Combining \eqref{sluVie} and \eqref{sluVii}, we infer\footnote{Recall that $u\les r$ in $\Vie$ and $r\les u$ in $\Vii$.}
    \begin{align*}
        \DDi_{2,s-2}[\slu]+\DDi_{s}[\slu]\les\ep_0^2.
    \end{align*}
    Combining with \eqref{slue}, this concludes the proof of Proposition \ref{extslu}.
\end{proof}
\begin{prop}\label{slunab}
We have the following estimates:
\begin{align*}
|r\nabs\slu|_{2,S}\les\frac{\ep_0}{t^\frac{s+1}{2}},\quad \forall\; S\subseteq\Vi,\qquad\qquad |r\nabs\slu|_{2,S}\les\frac{\ep_0}{r^\frac{s+1}{2}},\quad \forall\; S\subseteq\Ve.
\end{align*}
We also have
\begin{align*}
    \EE_{s-\db,\db}^1[\slu]+\EEi_{0,s}^1[\slu]\les\ep_0^2,\qquad\quad \FF_{s-\db,\db}^1[\slu]+\FFi_{0,s}^1[\slu]\les\ep_0^2,
\end{align*}
where $\db$ is defined in \eqref{dfdb}.
\end{prop}
\begin{proof}
We have from Propositions \ref{commutation} and \ref{slutrans}
\begin{align*}
\nabs_4(r\nabs\slu)+\frac{3}{2}\trch(r\nabs\slu)=O(\nabs\b)+\Gag\c(\slu^{(1)},\b^{(1)})+r^{-1}\Gag^{(1)}\c\Gag^{(1)}.
\end{align*}
Applying Lemma \ref{evolution}, we obtain for $u\leq 1$
    \begin{align*}
        |r^{3}\nabs\slu|_{2,S}&\les |r^3\nabs\slu|_{2,S_2}+\int_{\ujp}^r |r\b^{(1)}|_{2,S}+|\Gag|_{\infty,S}|r^2\slu^{(1)}|_{2,S}+|r\Gag^{(1)}\c\Gag^{(1)}|_{2,S}\\
        &\les\frac{\ep_0}{\ujp^\frac{s-3}{2}}+\left(\int_{\ujp}^r r^{s}|\b^{(1)}|_{2,S}^2\right)^\frac{1}{2}\left(\int_{\ujp}^r r^{2-s}dr\right)^\frac{1}{2}+\int_{\ujp}^r\frac{\ep^2}{r^\frac{s+1}{2}r^\frac{s-3}{2}}\\
        &\les\frac{\ep_0}{r^\frac{s-3}{2}},
    \end{align*}
which implies
    \begin{align*}
        |r^\frac{s+3}{2}\nabs\slu|_{2,S}\les\ep_0,\qquad\forall\; S\subseteq\Ve.
    \end{align*}
Similarly, we have
\begin{align*}
    |r^\frac{s+3}{2}\nabs\slu|_{2,S}\les\ep_0,\qquad\forall\; S\subseteq\Vie.
\end{align*}
Next, we have from Lemma \ref{evolution} that for any $S\subseteq \Vii$
\begin{align*}
    |r^{3}\nabs\slu|_{2,S}&\les|r^3\nabs\slu|_{2,S_c}+\int_{t_c(u)}^t |r\b^{(1)}|_{2,S}+|\Gag|_{\infty,S}|r^2\slu^{(1)}|_{2,S}+|r\Gag^{(1)}\c\Gag^{(1)}|_{2,S}\\
    &\les \left(\int_{t_c(u)}^t r^{-1}|\b^{(1)}|_{2,S}^2\right)^\frac{1}{2}\left(\int_{t_c(u)}^t r^3 dt\right)^\frac{1}{2}+\int_{t_c(u)}^t |\Gag|_{\infty,S}|r^2\slu^{(1)}|_{2,S}+|r\Gag^{(1)}|_{4,S}|\Gag^{(1)}|_{4,S},\\
    &\les r^2\frac{\ep_0}{u^\frac{s+1}{2}}+\int_{t_c(u)}^t \frac{\ep}{u^\frac{s+1}{2}}\frac{\ep r}{u^\frac{s+1}{2}}+\frac{\ep^2 r^2}{u^{s+2}}\\
    &\les\frac{\ep_0 r^2}{u^\frac{s+1}{2}},
\end{align*}
where we used Proposition \ref{betaimproved} at the third step. Hence, we obtain
\begin{align*}
    |r\nabs\slu|_{2,S}\les\frac{\ep_0}{u^\frac{s+1}{2}},\qquad \forall \; S\subseteq \Vii.
\end{align*}
Combining the above estimates, we deduce
\begin{align}\label{sluesteq}
    |r\nabs\slu|_{2,S}\les\frac{\ep_0}{t^\frac{s+1}{2}},\quad\, \forall\; S\subseteq\Vi,\qquad\qquad|r\nabs\slu|_{2,S}\les\frac{\ep_0}{r^\frac{s+1}{2}},\quad \forall\; S\subseteq\Ve.
\end{align}
Thus, we obtain for $u\geq 1$
\begin{align*}
    \int_{\cuv}|r\nabs\slu|^2\les\int_{t_c(u)}^t |r\nabs\slu|_{2,S}^2 dt\les\int_{t_c(u)}^t \frac{\ep_0^2}{t^{s+1}} dt\les\frac{\ep_0^2}{u^s},
\end{align*}
and for $u\leq 1$
\begin{align*}
    \int_{\cuv}r^{s-\db}|r\nabs\slu|^2\les\int_{\ujp}^r r^{s-\db}|r\nabs\slu|_{2,S}^2\les\int_{\ujp}^r\frac{\ep_0^2}{r^{1+\db}}\les\frac{\ep_0^2}{\ujp^\db}.
\end{align*}
Hence, we deduce
\begin{equation}
\EE_{s-\db,\db}^1[\slu]\les\ep_0^2,\qquad\quad \EEi_{0,s}^1[\slu]\les\ep_0^2.
\end{equation}
Next, we have from \eqref{sluesteq} that for $u\leq 1$
\begin{align*}
    \int_{\Si_t(-\infty,u)} r^{s-\db}|r\nabs\slu|^2\les\int_{-\infty}^u \frac{\ep_0^2}{r^{1+\db}}du\les\frac{\ep_0^2}{\ujp^\db},
\end{align*}
and similarly for $1\leq u_1\leq u_2\leq u_c(t)$
\begin{equation*}
    \int_{\Si_t(u_1,u_2)}r^{s-\db}|r\nabs\slu|^2\les\int_{u_1}^{u_2}\frac{\ep_0^2}{u^{1+\db}}du\les\frac{\ep_0^2}{u_1^\db}.
\end{equation*}
Finally, we have from Lemma \ref{evolution} that for $u\geq 1$
\begin{align*}
|r^3\nabs\slu|_{2,S}&\les \int_{t_c(u)}^t r|\b^{(1)}|_{2,S}+|\Gag|_{\infty,S}|r^2\slu^{(1)}|_{2,S}+|r\Gag^{(1)}\c\Gag^{(1)}|_{2,S}\\
&\les\left(\int_{t_c(u)}^t r^{-1}|\b^{(1)}|^2_{2,S}\right)^\frac{1}{2}\left(\int_{t_c(u)}^t r^3 dt\right)^\frac{1}{2}+\int_{t_c(u)}^t \frac{\ep^2 r^2}{t^{s+1}}\\
&\les r^2E_{-1}^1[\b](u)^\frac{1}{2}+\frac{\ep_0 r^2}{u^s},
\end{align*}
which implies
\begin{equation*}
    |r\nabs\slu|_{2,S}\les E_{-1}^1[\b](u)^\frac{1}{2}+\frac{\ep_0}{u^s}.
\end{equation*}
Hence, we obtain
\begin{align*}
    \int_{u_1}^{u_2}|r\nabs\slu|^2_{2,S}\les B_0^1[\b](u_1,u_2)+\frac{\ep_0^2}{u_1^{2s-1}}\les\frac{\ep_0^2}{u_1^s}.
\end{align*}
Combining the above estimates, we deduce
\begin{align}
    \FF_{s-\db,\db}^1[\slu]+\FFi_{0,s}^1[\slu]\les\ep_0^2.
\end{align}
This concludes the proof of Proposition \ref{slunab}.
\end{proof}
\subsection{Estimate for \texorpdfstring{$\mu$}{}}
\begin{prop}\label{mu}
We have the following estimate:
\begin{align*}
    \DD_{1,s-1}[\mu]+\DDi_{-2,s+2}[\mu]\les\ep_0^2.
\end{align*}
\end{prop}
\begin{proof}
We have from Proposition \ref{slutrans}
    \begin{align*}
        \nabs_4\mu+\trch\,\mu=r^{-1}O(\rho,\nabs\etab)+r^{-1}\Gag\c\Gab^{(1)}.
    \end{align*}
Applying Lemma \ref{evolution}, we deduce for $u\leq 1$
    \begin{align*}
        |r^{2-\frac{2}{p}}\mu|_{p,S}&\les|r^{2-\frac{2}{p}}\mu|_{p,S_2}+\int_{\ujp}^r |r^{1-\frac{2}{p}}(\rho,\nabs\etab)|_{p,S}+|r^{1-\frac{2}{p}}\Gag\c\Gab^{(1)}|_{p,S}\\
        &\les\frac{\ep_0}{\ujp^\frac{s-1}{2}}+\int_{\ujp}^r \frac{\ep_0}{r^\frac{s+1}{2}}+\frac{\ep^2}{r^\frac{s+1}{2}\ujp^\frac{s-1}{2}}\\
        &\les\frac{\ep_0}{\ujp^\frac{s-1}{2}}.
    \end{align*}
Note that we have from \eqref{B1}
\begin{align*}
    |r^{2-\frac{2}{p}}\mu|_{p,S_c}=\lim_{t\to t_c(u)}|r^{2-\frac{2}{p}}\mu|_{p,S}\les \lim_{r\to 0}\frac{\ep r^\frac{3}{2}}{u^\frac{s+2}{2}}=0.
\end{align*}
Applying Lemma \ref{evolution}, we have for $u\geq 1$
\begin{align*}
    |r^{2-\frac{2}{p}}\mu|_{p,S}&\les|r^{2-\frac{2}{p}}\mu|_{p,S_c}+\int_{t_c(u)}^t |r^{1-\frac{2}{p}}(\rho,\nabs\etab)|_{p,S}+|r^{1-\frac{2}{p}}\Gag\c\Gab^{(1)}|_{p,S}\\
    &\les\int_{t_c(u)}^t \frac{\ep_0}{t^\frac{s+1}{2}}+\frac{\ep^2}{t^s}\\
    &\les\frac{\ep_0}{u^\frac{s-1}{2}}.
\end{align*}
We also have for $u\geq 1$
\begin{align*}
    |r^{2-\frac{2}{p}}\mu|_{p,S}&\les|r^{2-\frac{2}{p}}\mu|_{p,S_c}+\int_{t_c(u)}^t |r^{1-\frac{2}{p}}(\rho,\nabs\etab)|_{p,S}+|r^{1-\frac{2}{p}}\Gag\c\Gab^{(1)}|_{p,S}\\
    &\les\int_{t_c(u)}^t\frac{\ep_0 r^\frac{1}{2}}{t^\frac{s+2}{2}}+\frac{\ep^2 r^\frac{1}{2}}{t^\frac{s+2}{2}}\\
    &\les\frac{\ep_0 r^\frac{3}{2}}{u^\frac{s+2}{2}},
\end{align*}
which implies
\begin{equation*}
    |r^{\frac{1}{2}-\frac{2}{p}}\mu|_{p,S}\les\frac{\ep_0}{u^\frac{s+2}{2}}.
\end{equation*}
Combining the above estimates, we infer
\begin{align*}
    \DD_{1,s-1}[\mu]+\DDi_{-2,s+2}[\mu]\les\ep_0^2.
\end{align*}
This concludes the proof of Proposition \ref{mu}.
\end{proof}
\begin{prop}\label{munabs}
We have the following estimate:
\begin{align*}
    \FF_{s-\db,\db}^1[\mu]+\FFi_{0,s}^1[\mu]\les\ep_0^2.
\end{align*}
\end{prop}
\begin{proof}
We have from Propositions \ref{commutation} and \ref{slutrans}
    \begin{align*}
        \nabs_4(r\nabs\mu)+\trch(r\nabs\mu)=O(\nabs\rho)+O(\nabs^2\etab)+r^{-1}(\Gag\c\Gab)^{(2)}.
    \end{align*}
Applying Lemma \ref{evolution}, we infer for $u\leq 1$
\begin{align*}
    |r^{2}\nabs\mu|_{2,S}\les |r^2\nabs\mu|_{2,S_2}+\int_{\ujp}^r|X|_{2,S}\les\frac{\ep_0}{\ujp^\frac{s-1}{2}}+\int_{\ujp}^r |X|_{2,S},
\end{align*}
where we denoted
\begin{align*}
    X:=\left(r\nabs\rho,r\nabs^2\etab,(\Gab\c\Gag)^{(2)}\right),
\end{align*}
which satisfies
    \begin{equation}\label{Xeq}
        \FF_{s}[X]\les\ep_0^2,\qquad \FFi_{-2,s+2}[X]\les\ep_0^2.
    \end{equation}
Hence, we obtain for $u\leq 1$
    \begin{align*}
        |r^{2}\nabs\mu|_{2,S}^2&\les\frac{\ep_0^2}{\ujp^{s-1}}+\left(\int_{\ujp}^r |X|_{2,S}\right)^2\\
        &\les\frac{\ep_0^2}{\ujp^{s-1}}+\int_{\ujp}^r r^{s-1-\db}|X|_{2,S}^2 \int_{\ujp}^r r^{1-s+\db}\\
        &\les\frac{\ep_0^2}{\ujp^{s-1}}+r^{2-s+\db}E_{s-1-\db}[X](u),
    \end{align*}
which implies
\begin{align*}
    \int_{-\infty}^u |r^\frac{s+2-\db}{2}\nabs\mu|_{2,S}^2&\les\int_{-\infty}^u\frac{\ep_0^2}{r^{s-2+\db}\ujp^{s-1}}+B_{s-\db}[X](u)\\
    &\les\int_{-\infty}^u\frac{\ep_0^2}{\ujp^{1+\db}}+\int_{2}^t\frac{dt}{(t+|u|)^{1+\db}}\int_\ucuv r^{s}|X|^2\\
    &\les\frac{\ep_0^2}{\ujp^\db}.
\end{align*}
Moreover, we have from Lemma \ref{evolution} that for $u\geq 1$
\begin{align*}
|r^2\nabs\mu|_{2,S}^2\les\left(\int_{t_c(u)}^t |X|_{2,S}\right)^2\les \int_{t_c(u)}^t r^{s-1-\db}|X|_{2,S}^2 \int_{0}^r r^{1-s+\db}\les r^{2-s+\db}E_{s-1-\db}[X],
\end{align*}
which implies for $1\leq u_1\leq u_2\leq u_c(t)$\footnote{Recall that $t\les r$ in $\Vie$ while $t\les u$ in $\Vii$.}
\begin{align*}
\int_{u_1}^{u_2}|r^{\frac{s+2-\db}{2}}\nabs\mu|_{2,S}^2&\les\int_\Vie r^{s-1-\db}|X|^2+\int_\Vii r^{s-1-\db}|X|^2\\
&\les\int_{t_c(u_1)}^t t^{-1-\db}\int_{\ucuv\cap\Vie} r^s|X|^2+\int_{t_c(u_1)}^t \int_{\ucuv\cap\Vii} r^{s-1-\db}|X|^2\\
&\les\int_{t_c(u_1)}^t \frac{\ep_0^2}{t^{1+\db}}+\int_{t_c(u_1)}^t \frac{\ep_0^2}{t^{1+\db}}\\
&\les\frac{\ep_0^2}{u_1^\db}.
\end{align*}
Thus, we obtain
\begin{equation*}
    \FF_{s-\db,\db}^1[\mu]\les\ep_0^2.
\end{equation*}
Similarly, we have for $u\geq 1$
\begin{align*}
|r\nabs\mu|_{2,S}^2\les\int_{t_c(u)}^t r^{-1}|X|^2_{2,S}=E_{-1}[X](u),
\end{align*}
which implies for $1\leq u_1\leq u_2\leq u_c(t)$
\begin{align*}
    \int_{u_1}^{u_2}|r\nabs\mu|_{2,S}^2&\les\int_{\Vie}r^{-1}|X|^2+\int_\Vii r^{-1}|X|^2\\
    &\les \int_{t_c(u_1)}^t \frac{dt}{t^{1+\db}}\int_{\ucuv\cap\Vie} r^\db|X|^2+\int_{t_c(u_1)}^t dt\int_{\ucuv\cap\Vii}r^{-1}|X|^2\\
    &\les\frac{\ep_0^2}{u_1^{s-\db}}\frac{1}{u_1^{\db}}+\int_{t_c(u_1)}^t \frac{\ep_0^2}{t^{s+1}}dt\\
    &\les\frac{\ep_0^2}{u_1^s}.
\end{align*}
Combining the above estimates, we deduce
\begin{align*}
\FF_{s-\db,\db}^1[\mu]+\FFi_{0,s}^1[\mu]\les\ep_0^2.
\end{align*}
This concludes the proof of Proposition \ref{munabs}.
\end{proof}
\subsection{Estimates for \texorpdfstring{$\trch$}{}, \texorpdfstring{$\hch$}{} and \texorpdfstring{$\eta$}{}}
\begin{prop}\label{slumuint}
We have the following estimates:
\begin{align*}
\DD_{-1,s-1}^1[\eta]+\DDi^1_{-3,s+1}[\eta]+\FF_{s-2-\db,\db}^2[\eta]+\FFi_{-2,s}^2[\eta]&\les\ep_0^2,\\
\DD_{s-2}^1[\hch]+\DDi_{-3,s+1}^1[\hch]+\FF_{s-2-\db,\db}^2[\hch]+\FFi_{-2,s}^2[\hch]&\les\ep_0^2,\\
\DD_{s-2}^1[\trchc]+\DDi_{-3,s+1}^1[\trchc]+\FF_{s-2-\db,\db}^2[\trchc]+\FFi_{-2,s}^2[\trchc]&\les\ep_0^2,\\
\DDinf_{-1,s-1}\left[\trch-\frac{2}{r}\right]+{^i\DDinf_{-3,s+1}}\left[\trch-\frac{2}{r}\right]&\les\ep_0^2.
\end{align*}
\end{prop}
\begin{proof}
We have from Proposition \ref{recalleq}
\begin{align*}
    \sdivs\eta=-\mu-\rhoc,\qquad\quad \curls\eta=\sic.
\end{align*}
Applying Proposition \ref{Lpestimates}, we have
\begin{align*}
    |r^{-\frac{2}{p}}(r\nabs)^{\leq 1}\eta|_{p,S}\les |r^{1-\frac{2}{p}}(\mu,\rhoc,\sic)|_{p,S}.
\end{align*}
Combining with Proposition \ref{mu}, we infer
\begin{align*}
    \DD_{-1,s-1}^1[\eta]+\DDi^1_{-3,s+1}[\eta]\les\ep_0^2.
\end{align*}
Moreover, combining with Proposition \ref{munabs}, we obtain
\begin{align*}
    \FF_{s-2-\db,\db}^1[\eta]+\FFi_{-2,s}^1[\eta]\les\ep_0^2.
\end{align*}
Similarly, we have from Proposition \ref{recalleq}
\begin{align*}
    \sdivs\hch=\frac{1}{2}\slu-\b+\Gag\c\Gag.
\end{align*}
Applying Proposition \ref{Lpestimates}, we deduce
\begin{align*}
|r^{-\frac{2}{p}}(r\nabs)^{\leq 1}\hch|_{p,S}\les |r^{1-\frac{2}{p}}(\slu,\b,\Gag\c\Gag)|_{p,S}.
\end{align*}
Combining with Propositions \ref{extslu} and \ref{slunab}, we infer
\begin{align*}
    \DD_{s-2}^1[\hch]+\DDi_{-3,s+1}^1[\hch]\les\ep_0^2,
\end{align*}
and
\begin{align*}
    \FF_{s-2-\db,\db}^2[\hch]+\FFi_{-2,s}^2[\hch]&\les\ep_0^2.
\end{align*}
Similarly, we have from \eqref{6.6}, \eqref{dfmu} and Proposition \ref{sluesteq}
\begin{align}
\begin{split}\label{trchcnab}
    \DD_{s-2}^1[\trchc]+\DDi_{-3,s+1}^1[\trchc]&\les\ep_0^2,\\
    \FF_{s-2-\db,\db}^2[\trchc]+\FFi_{-2,s}^2[\trchc]&\les\ep_0^2.
\end{split}
\end{align}
Next, we recall from Proposition \ref{recalleq}
\begin{align*}
    \nabs_4\trch+\frac{1}{2}(\trch)^2=-2\om\trch+\Gag\c\Gag.
\end{align*}
On the other hand, we have
\begin{align*}
\nabs_4\left(\frac{2}{r}\right)+\frac{1}{2}\trch\left(\frac{2}{r}\right)=-\frac{2e_4(r)}{r^2}+\frac{\trch}{r}=\frac{\widecheck{\phi\trch}}{\phi r}.
\end{align*}
Taking the difference, we obtain
\begin{equation*}
    \nabs_4\left(\trch-\frac{2}{r}\right)+\frac{1}{2}\trch\left(\trch-\frac{2}{r}\right)=r^{-1}O(\om,\widecheck{\phi\trch})+\Gag\c\Gag.
\end{equation*}
Applying Lemma \ref{evolution}, we infer for $u\leq 1$
\begin{align*}
    \left|r^{1-\frac{2}{p}}\left(\trch-\frac{2}{r}\right)\right|_{p,S}&\les \left|r^{1-\frac{2}{p}}\left(\trch-\frac{2}{r}\right)\right|_{p,S_2}+\int_{\ujp}^r |r^{-\frac{2}{p}}(\om,\widecheck{\phi\trch})|_{p,S}+|r^{1-\frac{2}{p}}\Gag\c\Gag|_{p,S}\\
    &\les\frac{\ep_0}{\ujp^\frac{s-1}{2}}+\int_{\ujp}^r \frac{\ep_0}{r^\frac{s+1}{2}}+\frac{\ep^2}{r^s}\\
    &\les\frac{\ep_0}{\ujp^\frac{s-1}{2}}.
\end{align*}
Noticing that
\begin{align*}
    \left|r^{1-\frac{2}{p}}\left(\trch-\frac{2}{r}\right)\right|_{p,S_c}=\lim_{t\to t_c(u)}\left|r^{1-\frac{2}{p}}\left(\trch-\frac{2}{r}\right)\right|_{p,S}\les\lim_{r\to 0}\frac{\ep r}{u^\frac{s+1}{2}}=0,
\end{align*}
we have for $u\geq 1$
\begin{align*}
    \left|r^{1-\frac{2}{p}}\left(\trch-\frac{2}{r}\right)\right|_{p,S}&\les \left|r^{1-\frac{2}{p}}\left(\trch-\frac{2}{r}\right)\right|_{p,S_c}+\int_{t_c(u)}^t |r^{-\frac{2}{p}}(\om,\widecheck{\phi\trch})|_{p,S}+|r^{1-\frac{2}{p}}\Gag\c\Gag|_{p,S}\\
    &\les\frac{\ep_0}{u^\frac{s-1}{2}}.
\end{align*}
We also have for $u\geq 1$
\begin{align*}
    \left|r^{1-\frac{2}{p}}\left(\trch-\frac{2}{r}\right)\right|_{p,S}&\les \left|r^{1-\frac{2}{p}}\left(\trch-\frac{2}{r}\right)\right|_{p,S_c}+\int_{t_c(u)}^t |r^{-\frac{2}{p}}(\om,\widecheck{\phi\trch})|_{p,S}+|r^{1-\frac{2}{p}}\Gag\c\Gag|_{p,S}\\
    &\les\int_{t_c(u)}^t \frac{\ep_0}{t^\frac{s+1}{2}}+\frac{r\ep^2}{t^{s+1}}\\
    &\les\frac{\ep_0 r}{t^\frac{s+1}{2}}.
\end{align*}
Combining the above estimates and \eqref{trchcnab}, we infer
\begin{align*}
\DDinf_{-1,s-1}\left[\trch-\frac{2}{r}\right]+{^i\DDinf_{-3,s+1}}\left[\trch-\frac{2}{r}\right]\les\ep_0^2.
\end{align*}
This concludes the proof of Proposition \ref{slumuint}.
\end{proof}
\subsection{Estimate for \texorpdfstring{$\nabs_T\eta$}{}}
\begin{prop}\label{LieTeta}
We have the following estimate:
\begin{align*}
    \DDi_{-1,s+1}[\nabs_T\eta]+\DDi_{-2,s+2}[\nabs_T\eta]\les\ep_0^2.
\end{align*}
\end{prop}
\begin{proof}
We have from Propositions \ref{recalleq} and \ref{Tcomm}
    \begin{align*}
        \nabs_4(\nabs_T\eta)+\frac{1}{2}\trch(\nabs_T\eta)=r^{-1}O(\nabs_T\etab)-\nabs_T\b+r^{-1}\Gag\c\Gab^{(1)}.
    \end{align*}
    We have from \eqref{6.6} that in $\Vi$
    \begin{align*}
        |r^{-\frac{2}{p}}\nabs_T\etab|_{p,S}\les |r^{-\frac{2}{p}}\nabs_T\eps|_{p,S}+|r^{-\frac{2}{p}}\nabs_T\nab\phi|_{p,S}\les\frac{\ep_0}{r^\frac{1}{2}u^\frac{1}{2}t^\frac{s+1}{2}}.
    \end{align*}
    Applying Lemma \ref{evolution}, we infer for $u\geq 1$
    \begin{align*}
        |r^{1-\frac{2}{p}}\nabs_T\eta|_{p,S}&\les\int_{t_c(u)}^t |r^{-\frac{2}{p}}\nabs_T\etab|_{p,S}+|r^{1-\frac{2}{p}}\nabs_T\b|_{p,S}+|r^{-\frac{2}{p}}\Gag\c\Gab^{(1)}|_{p,S}\\
        &\les\int_{t_c(u)}^t \frac{\ep_0}{r^\frac{1}{2}t^\frac{s+1}{2}u^\frac{1}{2}}+\int_{t_c(u)}^t\frac{\ep^2}{t^{\frac{s+3}{2}}u^\frac{s-1}{2}}\\
        &\les\frac{\ep_0}{u^\frac{s+1}{2}},
    \end{align*}
    which implies
    \begin{align*}
        \DDi_{-1,s+1}[\nabs_T\eta]\les\ep_0^2.
    \end{align*}
    Similarly, applying Lemma \ref{evolution}, we obtain for $u\geq 1$
    \begin{align*}
        |r^{1-\frac{2}{p}}\nabs_T\eta|_{p,S}&\les\int_{t_c(u)}^t |r^{-\frac{2}{p}}\nabs_T\etab|_{p,S}+|r^{1-\frac{2}{p}}\nabs_T\b|_{p,S}+|r^{-\frac{2}{p}}\Gag\c\Gab^{(1)}|_{p,S}\\
        &\les\int_{t_c(u)}^t \frac{\ep_0}{r^\frac{1}{2}t^\frac{s+1}{2}u^\frac{1}{2}}+\int_{t_c(u)}^t\frac{\ep^2}{t^{\frac{s+3}{2}}u^\frac{s-1}{2}}\\
        &\les\int_{t_c(u)}^t \frac{\ep_0}{r^\frac{1}{2}u^\frac{s+2}{2}}\\
        &\les\frac{\ep_0 r^\frac{1}{2}}{u^\frac{s+2}{2}},
    \end{align*}
    which implies
    \begin{align*}
        \DDi_{-2,s+2}[\nabs_T\eta]\les\ep_0^2.
    \end{align*}
    This concludes the proof of Proposition \ref{LieTeta}.
\end{proof}
Combining Propositions \ref{extslu}--\ref{LieTeta}, this concludes the proof of Theorem \ref{M3}.
\appendix
\renewcommand{\appendixname}{Appendix~\Alph{section}}
\section{Proof of Theorem \ref{wonderfulrp}}\label{secA}
We prove the following lemmas, which directly imply Theorem \ref{wonderfulrp}.
\begin{lem}\label{wonderfulrpA}
Let $p_1,p_2,p,\ell\leq s$ and let $s>1$. Then, we have the following properties:
\begin{enumerate}
    \item In the case $2p<p_1+p_2+\ell+1$, we have
\begin{align}\label{estll}
    \ee^1_p\left[\F_{p_1},\O_\ell\c\F_{p_2}\right](\Ve)&\les\frac{\ep^3}{\ujp^{s-p}}.
\end{align}
    \item In the case $2p<p_1+p_2+\ell$ and $2p+1<p_1+p_2+s$, we have
\begin{align}
    \ee^1_p\left[\F_{p_1},\O_\ell\c\Fb_{p_2}\right](\Ve)&\les\frac{\ep^3}{\ujp^{s-p}},\label{estlr}\\
    \ee^1_p\left[\Fb_{p_1},\O_\ell\c\F_{p_2}\right](\Ve)&\les\frac{\ep^3}{\ujp^{s-p}}.\label{estrl}
\end{align}
    \item In the case $2p<p_1+p_2+\ell-1$, we have
\begin{align}\label{estrr}
\ee^1_p\left[\Fb_{p_1},\O_\ell\c\Fb_{p_2}\right](\Ve)&\les\frac{\ep^3}{\ujp^{s-p}}.
\end{align}
\end{enumerate}
\end{lem}
\begin{proof}
Throughout the proof, we always denote $V:=\Ve$ and $0<\de\ll s-1$ a constant sufficiently small.\\ \\
We first assume that $2p<p_1+p_2+\ell+1$. Then, we have
\begin{align*}
    \int_V r^p|\F_{p_1}^{(1)}||\O_\ell||\F_{p_2}^{(1)}|&\les\int_V\frac{\ep}{\ujp^{\frac{s-\ell}{2}}}\frac{1}{r^{\frac{p_1+p_2+\ell+1}{2}-p}}r^\frac{p_1}{2}|\F_{p_1}^{(1)}|r^\frac{p_2}{2}|\F_{p_2}^{(1)}|\\
    &\les\int_V\frac{\ep}{\ujp^\frac{s+1}{2}}\frac{1}{\ujp^{\frac{p_1-p}{2}}}r^\frac{p_1}{2}|\F_{p_1}^{(1)}|\frac{1}{\ujp^\frac{p_2-p}{2}}r^\frac{p_2}{2}|\F_{p_2}^{(1)}|\\
    &\les\int_V\frac{\ep}{\ujp^\frac{s+1}{2}}\frac{1}{\ujp^{p_1-p}}r^{p_1}|\F_{p_1}^{(1)}|^2+\int_V\frac{\ep}{\ujp^\frac{s+1}{2}}\frac{1}{\ujp^{p_2-p}}r^{p_2}|\F_{p_2}^{(1)}|^2\\
    &\les\int_{-\infty}^udu\frac{\ep}{\ujp^\frac{s+1}{2}}\frac{1}{\ujp^{p_1-p}}\int_\cuv r^{p_1}|\F_{p_1}^{(1)}|^2+\int_{-\infty}^u du\frac{\ep}{\ujp^\frac{s+1}{2}}\frac{1}{\ujp^{p_2-p}}\int_\cuv r^{p_2}|\F_{p_2}^{(1)}|^2\\
    &\les\int_{-\infty}^udu\frac{\ep}{\ujp^\frac{s+1}{2}}\frac{1}{\ujp^{p_1-p}}\frac{\ep^2}{\ujp^{s-p_1}}+\int_{-\infty}^udu\frac{\ep}{\ujp^\frac{s+1}{2}}\frac{1}{\ujp^{p_2-p}}\frac{\ep^2}{\ujp^{s-p_2}}\\
    &\les \int_{-\infty}^u du \frac{\ep^3}{\ujp^\frac{s+1}{2}\ujp^{s-p}}\\
    &\les\frac{\ep^3}{\ujp^{s-p}}.
\end{align*}
We also have
    \begin{align*}
        \int_V r^p|\F_{p_1}^{(1)}||\O_\ell^{(1)}||\F_{p_2}|&\les\int_{-\infty}^u du\left(\int_{\cuv} r^{p_1}|\F_{p_1}^{(1)}|^2\right)^\frac{1}{2}\left(\int_\cuv r^{2p-p_1}|\F_{p_2}|^2|\O_\ell^{(1)}|^2\right)^\frac{1}{2}\\
        &\les\int_{-\infty}^udu\frac{\ep}{\ujp^{\frac{s-p_1}{2}}}\left(\int_{2}^t r^{2p-p_1}|\F_{p_2}|_{4,S}^2|\O_\ell^{(1)}|_{4,S}^2dt\right)^\frac{1}{2}\\
        &\les\int_{-\infty}^udu\frac{\ep}{\ujp^\frac{s-p_1}{2}}\left(\int_{2}^t r^{2p-p_1-p_2-\ell-2}|r^{\frac{p_2+3}{2}-\frac{2}{4}}\F_{p_2}|_{4,S}^2|r^{\frac{\ell+1}{2}-\frac{2}{4}}\O_\ell^{(1)}|_{4,S}^2 dt\right)^\frac{1}{2}\\
        &\les\int_{-\infty}^udu\frac{\ep}{\ujp^\frac{s-p_1}{2}}\left(\int_{2}^tr^{2p-p_1-p_2-\ell-2}\frac{\ep^2}{\ujp^{s-p_2}}\frac{\ep^2}{\ujp^{s-\ell}}dt\right)^\frac{1}{2}\\
        &\les\int_{-\infty}^u du\frac{\ep^3}{\ujp^\frac{s-p_1}{2}}\frac{1}{\ujp^{s-\frac{p_2+\ell}{2}}}\left(\ujp^{2p-p_1-p_2-\ell-1+\de}\int_2^t r^{-1-\de}dt\right)^\frac{1}{2}\\
        &\les\int_{-\infty}^u du\frac{\ep^3}{\ujp^\frac{s-p_1}{2}}\frac{1}{\ujp^{s-\frac{p_2+\ell}{2}}}\ujp^{p-\frac{p_1+p_2+\ell+1-\de}{2}}\\
        &\les\int_{-\infty}^u du\frac{\ep^3}{\ujp^\frac{s+1}{2}}\frac{1}{\ujp^{s-p-\de}}\\
        &\les\frac{\ep^3}{\ujp^{s-p}}.
    \end{align*}
    Hence, we obtain \eqref{estll}.\\ \\
    Next, we assume that $2p<p_1+p_2+\ell$. Then, we have
    \begin{align*}
        \int_V r^p|\F_{p_1}^{(1)}||\O_\ell^{(1)}||\Fb_{p_2}|&\les\int_{-\infty}^u du\left(\int_{\cuv} r^p|\F_{p_1}^{(1)}|^2\right)^\frac{1}{2}\left(\int_\cuv r^{2p-p_1}|\Fb_{p_2}|^2|\O_\ell^{(1)}|^2\right)^\frac{1}{2}\\
        &\les\int_{-\infty}^u du\frac{\ep}{\ujp^{\frac{s-p_1}{2}}}\left(\int_2^t r^{2p-p_1}|\Fb_{p_2}|_{4,S}^2|\O_\ell^{(1)}|_{4,S}^2\right)^\frac{1}{2}\\
        &\les \int_{-\infty}^u du\frac{\ep}{\ujp^\frac{s-p_1}{2}}\left(\int_2^t r^{2p-p_1-p_2-\ell-1}|r^{\frac{p_2+2}{2}-\frac{2}{4}}\Fb_{p_2}|_{4,S}^2|r^{\frac{\ell+1}{2}-\frac{2}{4}}\O_\ell^{(1)}|_{4,S}^2\right)^\frac{1}{2}\\
        &\les\int_{-\infty}^udu\frac{\ep}{\ujp^\frac{s-p_1}{2}}\left(\int_2^t r^{2p-p_1-p_2-\ell-1}\frac{\ep^2}{\ujp^{s-p_2+1}}\frac{\ep^2}{\ujp^{s-\ell}}\right)^\frac{1}{2}\\
        &\les \int_{-\infty}^u du\frac{\ep^3}{\ujp^\frac{s-p_1}{2}}\frac{1}{\ujp^{s-\frac{p_2}{2}+1-\frac{\ell+1}{2}}}\left(\ujp^{2p-p_1-p_2-\ell-1+\de}\int_2^t r^{-1-\de}dt\right)^\frac{1}{2}\\
        &\les\int_{-\infty}^u du\frac{\ep^3}{\ujp^\frac{s+1}{2}}\frac{1}{\ujp^{s-p-\de}}\\
        &\les\frac{\ep^3}{\ujp^{s-p}}.
    \end{align*}
    We also have
    \begin{align*}
    \int_V r^p|\F_{p_1}^{(1)}||\O_\ell\c\Fb_{p_2}^{(1)}|&\les\int_V \frac{\ep}{r^{\frac{\ell+1}{2}-p}\ujp^{\frac{s-\ell}{2}}}|\F_{p_1}^{(1)}||\Fb_{p_2}^{(1)}|\\
    &\les\int_V\frac{\ep}{\ujp^{\frac{s-\ell-1-\de}{2}}}\frac{1}{r^{\frac{p_1+p_2+\ell-\de}{2}-p}}r^\frac{p_1}{2}\ujp^{-\frac{1+\de}{2}}|\F_{p_1}^{(1)}|r^\frac{p_2}{2}r^{-\frac{1+\de}{2}}|\Fb_{p_2}^{(1)}|\\
    &\les\frac{\ep}{\ujp^{\frac{p_1+p_2+s-1}{2}-p-\de}}\int_Vr^\frac{p_1}{2}\ujp^{-\frac{1+\de}{2}}|\F_{p_1}^{(1)}|r^\frac{p_2}{2}r^{-\frac{1+\de}{2}}|\Fb_{p_2}^{(1)}|\\
    &\les\frac{\ep}{\ujp^{\frac{p_1+p_2+s-1}{2}-p-\de}}\left(\int_Vr^{p_1}\ujp^{-1-\de}|\F_{p_1}^{(1)}|^2\right)^\frac{1}{2}\left(\int_V r^{p_2}r^{-{1-\de}}|\Fb_{p_2}^{(1)}|^2\right)^\frac{1}{2}\\
    &\les\frac{\ep}{\ujp^{\frac{p_1+p_2+s-1}{2}-p-\de}}\left(\int_{-\infty}^u \ujp^{-1-\de}\int_\cuv r^{p_1}|\F_{p_1}^{(1)}|^2\right)^\frac{1}{2}\left(\int_{2}^t t^{-1-\de}\int_\ucuv r^{p_2}|\Fb_{p_2}^{(1)}|^2\right)^\frac{1}{2}\\
    &\les\frac{\ep}{\ujp^{\frac{p_1+p_2+s-1}{2}-p-\de}}\frac{\ep}{\ujp^\frac{s-p_1}{2}}\frac{\ep}{\ujp^\frac{s-p_2}{2}}\\
    &\les\frac{\ep^3}{\ujp^{s-p}},
    \end{align*}
    where we used $p_1+p_2+s-1>2p$. Hence, we obtain \eqref{estlr}.\\ \\
    Proceeding as above, we deduce for $p_1+p_2+\ell>2p$ and $p_1+p_2+s-1>2p$
    \begin{align*}
        \int_V r^p|\Fb_{p_1}^{(1)}\c\O_\ell\c\F_{p_2}^{(1)}|\les\frac{\ep^3}{\ujp^{s-p}}.
    \end{align*}
    We also have
    \begin{align*}
        \int_V r^p|\F_{p_2}^{(1)}||\O_\ell^{(1)}\c\Fb_{p_1}|&\les\int_{-\infty}^u du\left(\int_{\cuv} r^{p_2}|\F_{p_2}^{(1)}|^2\right)^\frac{1}{2}\left(\int_\cuv r^{2p-p_2}|\Fb_{p_1}|^2|\O_\ell^{(1)}|^2\right)^\frac{1}{2}\\
        &\les\int_{-\infty}^udu\frac{\ep}{\ujp^{\frac{s-p_2}{2}}}\left(\int_{2}^t r^{2p-p_2}|\Fb_{p_1}|_{4,S}^2|\O_\ell^{(1)}|_{4,S}^2\right)^\frac{1}{2}\\
        &\les\int_{-\infty}^u du\frac{\ep}{\ujp^\frac{s-p_2}{2}}\left(\int_2^t r^{2p-p_1-p_2-\ell-1}|r^{\frac{p_1+2}{2}-\frac{2}{4}}\Fb_{p_1}|_{4,S}^2|r^{\frac{\ell+1}{2}-\frac{2}{4}}\O_\ell^{(1)}|_{4,S}^2\right)^\frac{1}{2}\\
        &\les\int_{-\infty}^u du\frac{\ep}{\ujp^\frac{s-p_2}{2}}\left(\int_2^t r^{2p-p_1-p_2-\ell-1}\frac{\ep^2}{\ujp^{s-p_1+1}}\frac{\ep^2}{\ujp^{s-\ell}}\right)^\frac{1}{2}\\
        &\les\int_{-\infty}^u du \frac{\ep}{\ujp^\frac{s-p_2}{2}}\frac{\ep}{\ujp^\frac{s-p_1+1}{2}}\frac{\ep}{\ujp^\frac{s-\ell}{2}}\frac{1}{\ujp^{\frac{p_1+p_2+\ell}{2}-p}}\\
        &\les\frac{\ep^3}{\ujp^{s-p}}\int_{-\infty}^u \frac{1}{\ujp^\frac{s+1}{2}}du\\
        &\les\frac{\ep^3}{\ujp^{s-p}}.
    \end{align*}
    Hence, we obtain \eqref{estrl}.\\ \\
    Finally, we assume that $2p+1<p_1+p_2+\ell$. Then, we have
    \begin{align*}
    \int_V r^k|\Fb_{p_1}^{(1)}||\O_\ell^{(1)}\c\Fb_{p_2}|&\les\int_2^t dt\left(\int_{\ucuv}r^{p_1}|\Fb_{p_1}^{(1)}|^2\right)^\frac{1}{2}\left(\int_\ucuv r^{2p-p_1}|\Fb_{p_2}|^2|\O_\ell^{(1)}|^2\right)^\frac{1}{2}\\
    &\les\int_2^t dt\frac{\ep}{\ujp^{\frac{s-p_1}{2}}}\left(\int_{-\infty}^u r^{2p-p_1}|\Fb_{p_2}|_{4,S}^2|\O_\ell^{(1)}|_{4,S}^2\right)^\frac{1}{2}\\
    &\les\int_2^t dt\frac{\ep}{\ujp^\frac{s-p_1}{2}}\left(\int_{-\infty}^u  r^{2p-p_1-p_2-\ell-1}|r^{\frac{p_2+2}{2}-\frac{2}{4}}\Fb_{p_2}|_{4,S}^2|r^{\frac{\ell+1}{2}-\frac{2}{4}}\O_\ell^{(1)}|_{4,S}^2\right)^\frac{1}{2}\\
    &\les\int_2^t dt\frac{\ep}{\ujp^\frac{s-p_1}{2}}r^{p-\frac{p_1+p_2+\ell+1}{2}}\left(\int_{-\infty}^u\frac{\ep^2}{\ujp^{s-p_2+1}}\frac{\ep^2}{\ujp^{s-\ell}}\right)^\frac{1}{2}\\
    &\les \int_2^t dt\frac{\ep^3}{\ujp^\frac{s-p_1}{2}}r^{p-\frac{p_1+p_2+\ell+1-\de}{2}}\left(\frac{1}{\ujp^{2s-p_2-\ell+\de}}\right)^\frac{1}{2}\\
    &\les\frac{\ep^3}{\ujp^\frac{s-p_1}{2}}\ujp^{p-\frac{p_1+p_2-\de}{2}}\frac{1}{\ujp^{s-\frac{p_2+1}{2}}}\\
    &\les\frac{\ep^3}{\ujp^{s-p}}.
    \end{align*}
    We also have
    \begin{align*}
        \int_V r^p|\Fb_{p_1}^{(1)}||\O_\ell||\Fb_{p_2}^{(1)}|&\les\int_2^t dt\int_{\ucuv}r^{p-\frac{p_1+p_2}{2}}\frac{\ep}{r^\frac{\ell+1}{2}\ujp^{\frac{s-\ell}{2}}}r^\frac{p_1}{2}|\Fb_{p_1}^{(1)}|r^\frac{p_2}{2}|\Fb_{p_2}^{(1)}|\\
        &\les\frac{\ep}{\ujp^{\frac{s-\ell}{2}}}\int_2^t r^{p-\frac{p_1+p_2+\ell+1}{2}}dt\int_{\ucuv}r^\frac{p_1}{2}|\Fb_{p_1}^{(1)}|r^\frac{p_2}{2} |\Fb_{p_2}^{(1)}|\\
        &\les\frac{\ep}{\ujp^{\frac{s-\ell}{2}}}\int_2^t r^{p-\frac{p_1+p_2+\ell+1}{2}}dt\left(\int_{\ucuv}r^{p_1}|\Fb_{p_1}^{(1)}|^2\right)^\frac{1}{2}\left(\int_\ucuv r^{p_2}|\Fb_{p_2}^{(1)}|^2\right)^\frac{1}{2}\\
        &\les\frac{\ep}{\ujp^{\frac{s-\ell}{2}}}\int_2^t r^{p-\frac{p_1+p_2+\ell+1}{2}} dt\frac{\ep^2}{\ujp^{\frac{s-p_1}{2}}\ujp^{\frac{s-p_2}{2}}}\\
        &\les\frac{\ep}{\ujp^{\frac{s-\ell}{2}}}\ujp^{p-\frac{p_1+p_2+\ell-1-\de}{2}} \frac{\ep^2}{\ujp^{\frac{s-p_1}{2}}\ujp^{\frac{s-p_2}{2}}}\\
        &\les\frac{\ep^3}{\ujp^{s-p}}.
    \end{align*}
    Hence, we obtain \eqref{estrr}. This concludes the proof of Lemma \ref{wonderfulrpA}.
\end{proof}
\begin{lem}\label{wonderfulrpint}
Let $p_1,p_2,p,\ell\leq s$ and $s>1$. Then, we have the following properties.
\begin{enumerate}
    \item In the case $2p<p_1+p_2+\ell+1$, we have
\begin{equation}\label{wonderfulllint}
    \ee^1_p\left[\F_{p_1},\O_\ell\c\F_{p_2}\right](\Vie)\les\frac{\ep^3}{u_1^{s-p}}.
\end{equation}
    \item In the case $2p<p_1+p_2+\ell$ and $2p+1<p_1+p_2+s$, we have
\begin{equation}\label{wonderfullrint}
    \ee^1_p\left[\F_{p_1},\O_\ell\c\Fb_{p_2}\right](\Vie)\les\frac{\ep^3}{u_1^{s-p}}.
\end{equation}
    \item In the case $2p<p_1+p_2+\ell$ and $2p+1<p_1+p_2+s$, we have
\begin{equation}\label{wonderfulrlint}
    \ee^1_p\left[\Fb_{p_1},\O_\ell\c\F_{p_2}\right](\Vie)\les\frac{\ep^3}{u_1^{s-p}}.
\end{equation}
    \item In the case $2p<p_1+p_2+\ell-1$, we have
\begin{equation}\label{wonderfulrrint}
\ee^1_p\left[\Fb_{p_1},\O_\ell\c\Fb_{p_2}\right](\Vie)\les\frac{\ep^3}{u_1^{s-p}}.
\end{equation}
\end{enumerate}
\end{lem}
\begin{proof}
    The proof follows directly by replacing $\int_{-\infty}^u du$ with $\int_{u_1}^{u_2} du$ and $\int_2^t dt$ with $\int_{t_c(u)}^t dt$ in the proof of Lemma \ref{wonderfulrpA}.
\end{proof}
\begin{lem}\label{wonderfulrpVi}
Let $0\leq p_1,p_2,p\leq s$, $\ell\leq s$ and let $s>1$. Then, we have the following properties.
\begin{enumerate}
    \item We have
\begin{equation}\label{wonintll}
    \ee^1_p\left[\F_{p_1},\O_\ell\c\F_{p_2}\right](\Vii)\les\frac{\ep^3}{u_1^{s-p}}.
\end{equation}
    \item In the case $2p+1<p_1+p_2+s$, we have
\begin{align}\label{wonintlr}
    \ee^1_p\left[\F_{p_1},\O_\ell\c\Fb_{p_2}\right](\Vii)&\les\frac{\ep^3}{u_1^{s-p}},\\
\label{wonintrl}
    \ee^1_p\left[\Fb_{p_1},\O_\ell\c\F_{p_2}\right](\Vii)&\les\frac{\ep^3}{u_1^{s-p}},\\
\label{wonintrr}
\ee^1_p\left[\Fb_{p_1},\O_\ell\c\Fb_{p_2}\right](\Vii)&\les\frac{\ep^3}{u_1^{s-p}}.
\end{align}
\end{enumerate}
\end{lem}
\begin{proof}
Throughout the proof, we always denote $V:=\Vii$ and $0<\de\ll s-1$ a constant small enough. Note that we have
\begin{align*}
    \F_{p_1}\subseteq \F_{p_2},\qquad \Fb_{p_1}\subseteq \Fb_{p_2} \qquad\forall\; p_1\geq p_2\geq 0.
\end{align*}
We also recall that
\begin{align*}
    r\les u,\, t\quad\mbox{ in }\;\Vii.
\end{align*}
We first have
\begin{align*}
\int_{\Vii}r^p|\F_{p_1}^{(1)}||\O_{\ell}||\F_{p_2}^{(1)}|&\les\int_{u_1}^{u_2}\frac{\ep}{u^{\frac{s-2p+1}{2}}}du \left(\int_\cuv |\F_{p_1}^{(1)}|^2 \right)^\frac{1}{2}\left(\int_\cuv |\F_{p_2}^{(1)}|^2 \right)^\frac{1}{2}\\
&\les \int_{u_1}^{u_2}\frac{\ep^3}{u^{\frac{3s-2p+1}{2}}}du\\
    &\les \frac{\ep^3}{u_1^{s-p}}.
\end{align*}
We also have
\begin{align}
    \begin{split}\label{EEO}
\int_{\Vii}r^p|\F_{p_1}^{(1)}||\O_\ell^{(1)}||\F_{p_2}|&\les\int_{u_1}^{u_2}u^p du  \int_\cuv |\F_{p_1}^{(1)}||\O_\ell^{(1)}||\F_{p_2}|\\
    &\les \int_{u_1}^{u_2}u^p du \left(\int_\cuv |\F_{p_1}^{(1)}|^2\right)^\frac{1}{2}\left(\int_\cuv |\O_\ell^{(1)}|^2|\F_{p_2}|^2\right)^\frac{1}{2}\\
    &\les \int_{u_1}^{u_2}u^p du \frac{\ep}{u^\frac{s}{2}}\left(\int_{t_c(u)}^t dt |\O_\ell^{(1)}|^2_{4,S}|\F_{p_2}|^2_{4,S}\right)^\frac{1}{2}\\
    &\les \int_{u_1}^{u_2}du \frac{\ep}{u^\frac{s-2p}{2}}\left(\int_{t_c(u)}^t dt \frac{\ep^2}{u^{s}}\frac{\ep^2}{u^{s+2}}\right)^\frac{1}{2}\\
    &\les\int_{u_1}^{u_2}du \frac{\ep}{u^\frac{s-2p}{2}}\frac{\ep^2}{u^\frac{{2s+1}}{2}}\\
    &\les\frac{\ep^3}{u_1^{s-p}}.
    \end{split}
\end{align}
Thus, we obtain \eqref{wonintll}.\\ \\
In the rest of the proof, we always assume that
\begin{align}\label{ass}
    2p+1<p_1+p_2+s.
\end{align}
We have for $p<\frac{s-1}{2}$
\begin{align*}
\int_{\Vii}r^p|\F_{p_1}^{(1)}||\O_\ell||\Fb_{p_2}^{(1)}|&\les\ep\int_{\Vii}\frac{1}{u^{\frac{s+1}{2}-p}}|\F_{p_1}^{(1)}||\Fb_{p_2}^{(1)}|\\
&\les\ep\left(\int_{u_1}^{u_2}\frac{du}{u^{\frac{s+1}{2}-p}}\int_\cuv|\F_{p_1}^{(1)}|^2\right)^\frac{1}{2}\left(\int_{t_c(u_1)}^t\frac{dt}{t^{\frac{s+1}{2}-p}}\int_\ucuv|\Fb_{p_2}^{(1)}|^2\right)^\frac{1}{2}\\
&\les\ep\left(\int_{u_1}^{u_2}\frac{du}{u^{\frac{s+1}{2}-p}}\frac{\ep^2}{u^{s}}\right)^\frac{1}{2}\left(\int_{t_c(u_1)}^t\frac{dt}{t^{\frac{s+1}{2}-p}}\frac{\ep^2}{u_1^{s}}\right)^\frac{1}{2}\\
&\les\frac{\ep^3}{u_1^s}\frac{1}{u_1^{\frac{s-1}{2}-p}}\\
&\les\frac{\ep^3}{u_1^{s-p}},
\end{align*}
and for $p\geq \frac{s-1}{2}$ and $0<\de\ll s-1$ sufficiently small
\begin{align*}
&\;\;\;\,\,\,\int_{\Vii}r^p|\F_{p_1}^{(1)}||\O_\ell||\Fb_{p_2}^{(1)}|\\
&\les\ep\int_{\Vii}\frac{1}{u^{1+\de}}r^{p+\de-\frac{s-1}{2}}|\F_{p_1}^{(1)}||\Fb_{p_2}^{(1)}|\\
&\les\ep\int_{\Vii}\frac{1}{u^{1+\de}}r^{\left(p+\de-\frac{s-1}{2}\right)\frac{p_1}{p_1+p_2}}|\F_{p_1}^{(1)}|r^{\left(p+\de-\frac{s-1}{2}\right)\frac{p_2}{p_1+p_2}}|\Fb_{p_2}^{(1)}|\\
&\les\ep\left(\int_{u_1}^{u_2}\frac{du}{u^{1+\de}}\int_\cuv r^{\left(p+\de-\frac{s-1}{2}\right)\frac{2p_1}{p_1+p_2}}|\F_{p_1}^{(1)}|^2\right)^\frac{1}{2}\left(\int_{t_c(u_1)}^{t}\frac{dt}{t^{1+\de}}\int_\ucuv r^{\left(p+\de-\frac{s-1}{2}\right)\frac{2p_2}{p_1+p_2}}|\Fb_{p_2}^{(1)}|^2\right)^\frac{1}{2}\\\
&\les\frac{\ep^3}{u_1^{s-\left(p+\de-\frac{s-1}{2}\right)}}\left(\int_{u_1}^{u_2}\frac{du}{u^{1+\de}}\right)^\frac{1}{2}\left(\int_{t_c(u_1)}^t \frac{dt}{t^{1+\de}}\right)^\frac{1}{2}\\
&\les\frac{\ep^3}{u_1^{s-p}},
\end{align*}
where we used
\begin{align*}
    \left(p+\de-\frac{s-1}{2}\right)\frac{2p_1}{p_1+p_2}< p_1,\qquad \left(p+\de-\frac{s-1}{2}\right)\frac{2p_2}{p_1+p_2}< p_2,
\end{align*}
which follows from \eqref{ass} and the smallness of $\de$. Proceeding as in \eqref{EEO}, we have
\begin{align*}
    \int_{\Vii}r^p|\F_{p_1}^{(1)}||\O_\ell^{(1)}||\Fb_{p_2}|\les \frac{\ep^3}{u_1^{s-p}}.
\end{align*}
Thus, we obtain \eqref{wonintlr}. \\ \\
Next, we have for $p\leq \frac{p_1}{2}$ and $0<\de\ll s-1$ sufficiently small
\begin{align*}
\int_{\Vii}r^p|\Fb_{p_1}^{(1)}||\O_\ell^{(1)}||\F_{p_2}|&\les\int_{\Vii} r^{p}|\Fb_{p_1}^{(1)}||\O_\ell^{(1)}||\F_{p_2}|\\
&\les\left(\int_{\Vii}\frac{r^{2p}}{u^{1+\de}}|\Fb_{p_1}^{(1)}|^2\right)^\frac{1}{2}\left(\int_{\Vii}u^{1+\de}|\O_\ell^{(1)}|^2|\F_{p_2}|^2\right)^\frac{1}{2}\\
&\les\left(\int_{t_c(u_1)}^t\frac{dt}{t^{1+\de}}\int_\ucuv r^{2p}|\Fb_{p_1}^{(1)}|^2\right)^\frac{1}{2}\left(\int_{u_1}^{u_2}du\int_{t_c(u)}^t dt u^{1+\de}|\O_\ell^{(1)}|^2_{4,S}|\F_{p_2}|^2_{4,S}\right)^\frac{1}{2}\\
&\les\frac{\ep}{u_1^{\frac{s-2p}{2}}}\left(\int_{u_1}^{u_2}du\int_{t_c(u)}^t dt \frac{\ep^4}{u^{2s+1-\de}}\right)^\frac{1}{2}\\
&\les\frac{\ep^3}{u_1^{s-p}},
\end{align*}
and also for $p\geq \frac{p_1}{2}$
\begin{align}
\begin{split}\label{rpFOE}
\int_{\Vii}r^p|\Fb_{p_1}^{(1)}||\O_\ell^{(1)}||\F_{p_2}|&\les\int_{\Vii}r^{\frac{p_1}{2}}|\Fb_{p_1}^{(1)}|r^{p-\frac{p_1}{2}}|\O_\ell^{(1)}||\F_{p_2}| \\
&\les\left(\int_{\Vii}u^{-1-\de}r^{p_1}|\Fb_{p_1}^{(1)}|^2\right)^\frac{1}{2}\left(\int_{\Vii}u^{2p-p_1+1+\de}|\O_{\ell}^{(1)}|^2|\F_{p_2}|^2\right)^\frac{1}{2}\\
&\les\frac{\ep}{u_1^\frac{s-p_1}{2}}\left(\int_{u_1}^{u_2}du u^{2p-p_1+1+\de}\int_{t_c(u)}^t dt |\O_\ell^{(1)}|^2_{4,S}|\F_{p_2}|^2_{4,S}\right)^\frac{1}{2}\\
&\les\frac{\ep}{u_1^\frac{s-p_1}{2}}\left(\int_{u_1}^{u_2}du u^{2p-p_1+1+\de}\int_{t_c(u)}^t dt \frac{\ep^4}{t^{2s+2}}\right)^\frac{1}{2}\\
&\les\frac{\ep^3}{u_1^\frac{s-p_1}{2}}\left(\int_{u_1}^{u_2}du \frac{1}{u^{2s-2p+p_1-\de}}\right)^\frac{1}{2}\\
&\les\frac{\ep^3}{u_1^{s-p}},
\end{split}
\end{align}
where we used $2p+1<p_1+2s$.\footnote{Recalling that $p_2\leq s$ and $p_1+p_2+s> 2p+1$, we deduce that $2p+1<p_1+2s$.} Thus, we obtain \eqref{wonintrl}.\\ \\
Finally, we have for $p<\frac{s-1}{2}$
\begin{align*}
\int_{\Vii}r^p|\Fb_{p_1}^{(1)}||\O_\ell||\Fb_{p_2}^{(1)}|&\les\ep\int_{\Vii}\frac{1}{u^{\frac{s+1}{2}-p}}|\Fb_{p_1}^{(1)}||\Fb_{p_2}^{(1)}|\\
&\les\ep\left(\int_{t_c(u_1)}^{t}\frac{dt}{t^{\frac{s+1}{2}-p}}\int_\ucuv|\Fb_{p_1}^{(1)}|^2\right)^\frac{1}{2}\left(\int_{t_c(u_1)}^t\frac{dt}{t^{\frac{s+1}{2}-p}}\int_\ucuv|\Fb_{p_2}^{(1)}|^2\right)^\frac{1}{2}\\
&\les\ep\left(\int_{t_c(u_1)}^{t}\frac{dt}{t^{\frac{s+1}{2}-p}}\frac{\ep^2}{u_1^{s}}\right)^\frac{1}{2}\left(\int_{t_c(u_1)}^t\frac{dt}{t^{\frac{s+1}{2}-p}}\frac{\ep^2}{u_1^{s}}\right)^\frac{1}{2}\\
&\les\frac{\ep^3}{u_1^{s-p}},
\end{align*} 
and also for $p\geq\frac{s-1}{2}$ and $0<\de\ll s-1$ sufficiently small
    \begin{align*}
        &\;\;\;\,\,\,\int_{\Vii} r^p|\Fb_{p_1}^{(1)}||\O_\ell||\Fb_{p_2}^{(1)}|\\
        &\les \int_{\Vii} \frac{\ep}{u^{1+\de}}r^{p+\de-\frac{s-1}{2}}|\Fb_{p_1}^{(1)}||\Fb_{p_2}^{(1)}|\\
        &\les\ep\left(\int_{\Vii}u^{-1-\de}r^{\left(p+\de-\frac{s-1}{2}\right)\frac{2p_1}{p_1+p_2}}|\Fb_{p_1}^{(1)}|^2\right)^\frac{1}{2}\left(\int_{\Vii}u^{-1-\de}r^{\left(p+\de-\frac{s-1}{2}\right)\frac{2p_2}{p_1+p_2}}|\Fb_{p_2}^{(1)}|^2\right)^\frac{1}{2}\\
        &\les\ep\left(\int_{t_c(u_1)}^t\frac{dt}{t^{1+\de}}\int_\ucuv r^{\left(p+\de-\frac{s-1}{2}\right)\frac{2p_1}{p_1+p_2}}|\Fb_{p_1}^{(1)}|^2\right)^\frac{1}{2}\left(\int_{t_c(u_1)}^t\frac{dt}{t^{1+\de}}\int_\ucuv r^{\left(p+\de-\frac{s-1}{2}\right)\frac{2p_2}{p_1+p_2}}|\Fb_{p_2}^{(1)}|^2\right)^\frac{1}{2}\\
        &\les\frac{\ep^3}{u_1^{s-p}},
    \end{align*}
    where we used $2p+1<p_1+p_2+s$. Proceeding as in \eqref{rpFOE}, we have
    \begin{align*}
    \int_{V_i}r^p|\Fb_{p_1}^{(1)}||\O_\ell^{(1)}||\Fb_{p_2}|\les\frac{\ep^3}{u_1^{s-p}}.
    \end{align*}
    Thus, we obtain \eqref{wonintrr}. This concludes the proof of Lemma \ref{wonderfulrpVi}.
\end{proof}
Combining Lemmas \ref{wonderfulrpA}--\ref{wonderfulrpVi}, this concludes the proof of Theorem \ref{wonderfulrp}.\footnote{Note that the assumptions in Lemma \ref{wonderfulrpVi} are weaker than those of Lemma \ref{wonderfulrpint}.}

\end{document}